\documentclass{article}
\usepackage[numbers,sort&compress]{natbib}
\usepackage[english]{babel}
\usepackage{placeins}
\usepackage[letterpaper,top=2cm,bottom=2cm,left=2.8cm,right=3cm,marginparwidth=1.75cm]{geometry}
\usepackage{amssymb}
\usepackage{booktabs}
\usepackage{enumitem}

\usepackage{algpseudocode}

\usepackage{algorithm}
\usepackage{breqn}
\usepackage{verbatim}
\usepackage{caption}
\usepackage{subcaption}
\usepackage{tikz}
\usetikzlibrary{decorations.pathreplacing,decorations.markings}
\usepackage{float}
\usepackage{pifont}
\usepackage{enumerate}
\usepackage{geometry} 
\usepackage{amsmath}
\usepackage{graphicx}
\usepackage[colorlinks=true, allcolors=blue]{hyperref}
\newtheorem{theorem}{Theorem}[section]
\newtheorem{conjecture}[theorem]{Conjecture}
\newtheorem{problem}[theorem]{Problem}
\newtheorem{definition}[theorem]{Definition}
\newtheorem{lemma}[theorem]{Lemma}

\newtheorem{corollary}[theorem]{Corollary}
\newtheorem{case}{\indent Case}[section]

\newtheorem{claim}{\indent Claim}[theorem]      
\newenvironment{proof}{
\noindent {\bf Proof.}\rm}
{\mbox{}\hfill\rule{0.5em}{0.809em}\par}

\usepackage{authblk}

\newcommand{\iskfour}{\textnormal{ISK}_4}
\newcommand{\isk}{\textnormal{ISK}_4}
\newcommand{\iskt}{\left\{\textnormal{ISK}_4, \textnormal{triangle}\right\}}

\newcommand{\iskk}{\left\{\textnormal{ISK}_4, K_{3,3}\right\}}
\newcommand{\iskdb}{\left\{\textnormal{ISK}_4, \textnormal{diamond}, \textnormal{bowtie}\right\}}
\newcommand{\iskdbk}{\left\{\textnormal{ISK}_4, \textnormal{diamond}, \textnormal{bowtie}, K_{3,3}\right\}}
\newcommand{\iskdbpk}{\left\{\textnormal{ISK}_4, \textnormal{diamond}, \textnormal{bowtie}, \textnormal{prism}, K_{3,3}\right\}}
\newcommand{\isktk}{\left\{\textnormal{ISK}_4, \textnormal{triangle}, K_{3,3}\right\}}

\newcommand{\iskdbp}{\left\{\textnormal{ISK}_4, \textnormal{diamond}, \textnormal{bowtie}, \textnormal{prism}\right\}}
\newcommand{\db}{\left\{ \textnormal{diamond}, \textnormal{bowtie}\right\}}

\title{ \bf On the structures of \{diamond, bowtie\}-free graphs that do not contain an induced subdivision of $K_4$}

	\author{\bf Feng Liu\footnote{Email: liufeng0609@126.com.}}
	
	\author{\bf Shuang Sun\footnote{Email: chocolatesun@sjtu.edu.cn.}}
		
	\author{\bf Yan Wang\footnote{Email: yan.w@sjtu.edu.cn (corresponding author).}}

     \affil{\footnotesize School of Mathematics Sciences, Shanghai Jiao Tong University, 800 Dongchuan Road, Shanghai, 200240, China}
\begin{document}
\maketitle

\begin{abstract}
A graph is $\mathrm{ISK}_4$-free if it contains no induced subdivision of $K_4$. L\'ev\^eque et al. [\emph{J. Combin. Theory Ser. B} \textbf{102} (2012) 924--947] conjectured that all $\mathrm{ISK}_4$-free graphs are 4-colorable. Chen et al. [\emph{J. Graph Theory} \textbf{96} (2021) 554--577] proved that $\{\mathrm{ISK}_4, \mathrm{diamond}, \mathrm{bowtie}\}$-free graphs are 4-colorable and asked whether such graphs are 3-colorable, where a diamond is $K_4$ minus one edge and a bowtie consists of two triangles sharing a vertex. In this paper, we characterize the structures of $\{\mathrm{ISK}_4, \mathrm{diamond}, \mathrm{bowtie}\}$-free graphs and prove that such graphs are 3-colorable, which answers a question of Chen et al. [\emph{J. Graph Theory} \textbf{96} (2021) 554--577] affirmatively and extends a result of Chudnovsky et al. [\emph{J. Graph Theory} \textbf{92} (2019) 67--95]. Furthermore, our structural theorem yields a polynomial-time algorithm for decomposing $\{\mathrm{ISK}_4, \mathrm{diamond}, \mathrm{bowtie}\}$-free graphs, and consequently a polynomial-time algorithm for coloring this class of graphs.

\smallskip
 \noindent{\bf Keywords: Chromatic number; Induced subgraph; $\mathrm{ISK_4}$-free; $K_4$-subdivision}
			
\smallskip
\noindent{\bf AMS Subject Classification: 05C17, 05C75, 05C15}  
\end{abstract}

\section{Introduction }

All graphs in this paper are finite and simple. 
A graph $G$ is $k$-colorable if there exists a mapping $\phi: V(G)\rightarrow [k]$ such that $\phi(u)\neq \phi(v)$ whenever $uv\in E(G)$. The \emph{chromatic number} $\chi(G)$ of $G$ is the minimum integer $k$ such that $G$ is $k$-colorable. The clique number $\omega(G)$ of $G$ is the maximum integer $k$ such that $G$ contains a complete graph of size $k$. 
Determining the chromatic number $\chi(H)$ and the clique number $\omega(H)$ for a graph $H$ are NP-complete in general~\cite{Karp}.  

For a graph $H$, we say that $G$ is $H$-free if $G$ has no induced subgraph isomorphic to $H$. Let $\mathcal{F}$ be a family of graphs. We say that $G$ is $\mathcal{F}$-free if $G$ is $F$-free for every member $F$ of $\mathcal{F}$. 
A class of graphs is $\chi$-\emph{bounded} if there exists a function $f$ such that every graph of the class satisfies $\chi(G)\leq f(\omega(G))$. The concept of $\chi$-boundedness was raised by Gy\'{a}rf\'{a}s in 1975 \cite{g2}. 
One important research direction in the study of $\chi$-boundedness is to determine graph families~$\mathcal{H}$ such that the class of $\mathcal{H}$-free graphs is $\chi$-bounded, and to find the smallest $\chi$-binding function for such hereditary classes. A graph $G$ is {\em perfect} if $\chi(H)=\omega(H)$ for each induced subgraph $H$ of $G$.
	Perfect graphs are a well-known hereditary $\chi$-bounded graph class.
    It is obvious that the identity function is a $\chi$-binding function for perfect graphs.
	A {\em hole} in a graph is an induced subgraph which is a cycle of length at least
	four, and a hole is {\em even} (resp., {\em odd}) if its length is even (resp., {\em odd}).
	An {\em antihole} of a graph $G$ is an induced subgraph
	of $G$ whose complement graph is a cycle of length at least four.
	Chudnovsky et al. \cite{MC06} proved the famous Strong Perfect Graph Theorem, which shows that perfect graphs are equivalent to the class of $\{\text{odd hole, odd antihole}\}$-free graphs. 
By a probabilistic construction due to Erd\H{o}s \cite{Erdos1959}, if $\mathcal{H}$ is finite and none of the graphs in $\mathcal{H}$ is acyclic, then the class of $\mathcal{H}$-free graphs is not $\chi$-bounded.
Gy\'arf\'as~\cite{g2} and Sumner~\cite{Sumner1981} independently conjectured that for every forest~\(T\), the class of \(T\)-free graphs is \(\chi\)-bounded.
This conjecture has been confirmed for some special trees (\cite{Chudnovsky20192,g2,g3,Kierstead1994,Kierstead2004,scott1997,ss2020}).
We refer the readers to \cite{ss} for a survey on $\chi$-bounded problems.

A subdivision of a graph is obtained by replacing some edges with paths whose internal vertices all have degree two. Scott \cite{scott1997} conjectured that for any graph $H$, the class of those graphs that do not contain any subdivision of $H$ as an induced subgraph is $\chi$-bounded.
This conjecture was disproved by Pawlik et al. \cite{AP2014}. They showed that Scott's conjecture is false whenever $H$ is obtained from a non-planar graph by subdividing every edge at least once. 
However, it has been shown that Scott's conjecture is still true for several graphs $H$.
Scott et al. \cite{scott2016} proved the conjecture in the case when $H$ is a forest. Chudnovsky et al. \cite{Chudnovsky2016} showed that Scott’s conjecture holds when $H$ is a paw or a bull, where a $paw$ is the graph obtained by attaching a pendant edge to one vertex of a triangle, and a $bull$ is the graph obtained by attaching a pendant edge to two distinct vertices of a triangle.
In this paper, we focus on the case when~$H = K_4$.

 The structure of graphs containing no subdivision of $K_4$ as a (not necessarily induced) subgraph has been extensively studied. 
They are \emph{series-parallel} graphs. 
Note that every series-parallel graph has a vertex of degree at most two~\cite{duffin1965}.
An $\isk$ is a graph that is isomorphic to a subdivision of $K_4$. 
However, the structure of $\iskfour$-free graphs is still not fully understood.
 The study of such graphs was initiated by L\'ev\^eque et al.~\cite{Leveque2012},
who established a decomposition theorem for $\iskfour$-free graphs.
To state their result, we first introduce the necessary definitions and notation.

Given a graph $H$, the \emph{line graph} of $H$, denoted by $L(H)$, is the graph whose vertex set is $E(H)$, where two distinct vertices of $L(H)$ are adjacent if and only if the corresponding edges of $H$ share a common endpoint.
Two disjoint vertex sets $X$ and $Y$ are \emph{anticomplete} if no vertex in $X$ is adjacent to any vertex in $Y$.
 A \emph{cutset} of a (possibly disconnected) graph $G$ is a nonempty set of vertices whose removal increases the number of connected components of $G$.
 A cutset $X$ is a \emph{clique cutset} if $X$ is a clique in $G$. 
It is a \emph{$k$-cutset} if $|X|=k$.
A \emph{proper $2$-cutset} of a graph $G$ is a $2$-cutset $T=\{x,y\}$ such that $xy\notin E(G)$, there exists a partition $(T,X,Y)$ of $V(G)$ with $X,Y\neq\emptyset$ and $X$ is anticomplete to $Y$, and neither $G[X\cup T]$ nor $G[Y\cup T]$ is a path between $x$ and $y$. A \emph{star-cutset} of a graph is a cutset $S$ such that $S$ contains a vertex that is adjacent to every other vertex in $S$. A \emph{double star-cutset} of a graph is a cutset $D$ such that $D$ contains two adjacent vertices $u$ and $v$, and every vertex of $D$ is adjacent to at least one of $u$ and $v$. We now state the decomposition theorem for $\iskfour$-free graphs due to L\'ev\^eque et al.~\cite{Leveque2012}.
\begin{theorem}[L\'ev\^eque et al. \cite{Leveque2012}]\label{decomplsition-isk4-Main}
Let $G$ be an $\iskfour$-free graph. Then either $G$ is series-parallel, or $G$ is the line graph of a graph with maximum degree at most three, or $G$ has a clique-cutset, a proper $2$-cutset, a star-cutset, or a double star-cutset.
\end{theorem}

Furthermore, L\'ev\^eque et al. \cite{Leveque2012} showed
that the chromatic number of $\iskfour$-free graphs is bounded
by a constant $c \ge 2^{512}$, which implies that Scott's conjecture is true when $H=K_4$. The proof is based on a decomposition theorem for $\isk$-free
graphs in \cite{Leveque2012} and a result of K$\ddot{\mathrm{u}}$hn et al.
\cite{DKDO2004}.

A $wheel$ in a graph is a pair $W = (C,x)$, where $C$ is a hole and $x$ has at least three neighbors in $V(C)$. L\'ev\^eque et al.~\cite{Leveque2012}  gave a structural characterization of $\{\mathrm{ISK}_4,\textnormal{wheel}\}$-free graphs and used it to show that this class is $3$-colorable.
L\'ev\^eque et al.~\cite{Leveque2012} asked whether there exists a more powerful decomposition theorem for $\isk$-free graphs.
Moreover, they proposed the following conjecture.

\begin{conjecture}[L\'ev\^eque et al.~\cite{Leveque2012}]\label{4color}
Every $\isk$-free graph is $4$-colorable.
\end{conjecture}

Le \cite{NKL2017} showed that every $\mathrm{ISK_4}$-free graph is $24$-colorable by applying the layering method,
which was further improved to $8$ in \cite{CCCFL2020}. Trotignon et al.~\cite{Trotignon2017} showed
that every $\mathrm{ISK_4}$-free graph with girth at least $5$ is $3$-colorable.
They further conjectured that every $\iskt$-free graph is $3$-colorable. 
Le \cite{NKL2017} showed that every $\iskt$-free graph is $4$-colorable.
Recently, Chudnovsky et al. \cite{Chudnovsky2019} confirmed this conjecture. 

\begin{theorem}[Chudnovsky et al. \cite{Chudnovsky2019}]\label{thm:ISK4-triangle}
Every $\iskt$-free graph is $3$-colorable.
\end{theorem}

Chudnovsky et al. \cite{Chudnovsky2019} also studied the structure of $\iskt$-free graphs and obtained a decomposition theorem. 
\begin{theorem}[Chudnovsky et al. \cite{Chudnovsky2019}]
\label{thm:ISK4-triangle-structure}
    Let $G$ be an $\iskt$-free graph. Then, either $G$ has a clique cutset (of size at most two), $G$ is complete bipartite, or $G$ has a vertex of degree at most two.
\end{theorem}

A natural next step is to determine the optimal upper bound on the chromatic number of $\iskfour$-free graphs that contain triangles, settling Conjecture~\ref{4color}.
A {\it diamond} is the graph obtained from $K_4$ by removing one edge, a {\it bowtie} is the graph consisting of two triangles with one vertex identified 
(see Figure~\ref{Figure-diamond-bowtie-tripod}).  Chen et al.~\cite{Chen2021}~considered the case in which triangles are relatively far from each other.
In particular, they proved the following. 
\begin{theorem}[Chen et al. \cite{Chen2021}]\label{ISK4db}
Every $\iskdb$-free graph is $4$-colorable.
\end{theorem}
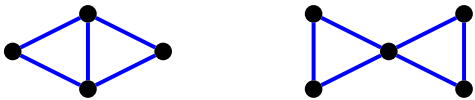
\begin{figure}[ht]
	\begin{center}
		\begin{tikzpicture}
			
			\tikzset{std node fill/.style={draw=black, circle,fill=black, line width=1pt, inner sep=2pt}}
			
			\node[std node fill] (x1) at (0,-2.5) {};
			\node[std node fill] (x2) at (1,-3) {};
			\node[std node fill] (x4) at (1,-2) {};
			\node[std node fill] (x5) at (2,-2.5) {};
			\foreach \i/\j in {1/2,1/4,2/4,2/5,4/5}{\draw[blue,line width=1.5pt] (x\i)--(x\j);}
			
			\node[std node fill] (y1) at (4,-3) {};
			\node[std node fill] (y2) at (4,-2) {};
			\node[std node fill] (y3) at (5,-2.5) {};
			\node[std node fill] (y4) at (6,-3) {};
			\node[std node fill] (y5) at (6,-2) {};
			\foreach \i/\j in {1/2,1/3,2/3,3/5,3/4,4/5}{\draw[blue,line width=1.5pt] (y\i)--(y\j);}
			
		\def\dx{0.8}  
		
		
		
		
		
		\end{tikzpicture}
		\caption{\small A diamond and a bowtie.}
		\label{Figure-diamond-bowtie-tripod}
	\end{center}
\end{figure}

Furthermore, Chen et al. proposed the following problem in \cite{Chen2021}.
\begin{problem}[Chen et al. \cite{Chen2021}]\label{pro:chen}
    Is every $\iskdb$-free graph $3$-colorable?
\end{problem}

The main result of this paper provides an affirmative answer to Problem~\ref{pro:chen}. 
\begin{theorem}\label{Main-theorem-color}
	Every $\iskdb$-free graph is $3$-colorable.
\end{theorem}

To prove Theorem \ref{Main-theorem-color}, we establish a more refined structural decomposition theorem for graphs that are both $\iskfour$-free and $\db$-free. 
A graph $G$ is called \emph{sparse} if, for every edge $e$ of $G$, at most one vertex incident with $e$ has degree greater than two. 

\begin{theorem}\label{Main-theorem-structure}
Let $G$ be an $\iskdb$-free graph. Then $G$ is either series-parallel, a complete bipartite graph, the line graph of a sparse graph with maximum degree at most three, a graph with a clique cutset or a proper $2$-cutset, or a graph having a vertex of degree at most two.
Furthermore, if $G$ is prism-free, then $G$ is either series-parallel, a complete bipartite graph, a graph with a clique cutset, or a graph having a vertex of degree at most two. 
\end{theorem}

It is easy to see that Theorem \ref{Main-theorem-structure} implies Theorem \ref{thm:ISK4-triangle-structure}.
In~\cite{12}, it is mentioned that deciding in polynomial time whether a given graph is $\text{ISK}_4$-free remains an open problem of interest.  
By Theorem~\ref{Main-theorem-structure} and additional structural analysis, we give a decomposition algorithm for any $\iskdb$-free graph and construct a proper $3$-coloring of it.


\medskip

This paper is organized as follows.  
In Section~2, we introduce additional notation and terminology, and collect some useful results needed for subsequent proofs.  
In Section~3, we first prove the existence of a specific wheel in $\iskdbp$-free graphs, which is later referred to as a proper wheel, and then develop a structural theorem for $\iskdbp$-free graphs centered around this wheel.  
In Section~4, we present several results concerning the centers of a wheel, which enables us to determine whether a vertex in an $\iskdbp$-free graph is the center of a proper wheel.  
In Section~5, we derive several lemmas that help identify vertices of degree at most $2$ in $\isktk$-free graphs.  
In Section~6, we first prove a theorem guaranteeing the existence of a vertex of degree at most $2$ in $\iskdbp$-free graphs, and then prove Theorems~\ref{Main-theorem-structure} and~\ref{Main-theorem-color}.  
Finally, in Section~7, we give a polynomial-time algorithm for decomposing $\iskdb$-free graphs, and consequently a polynomial-time algorithm for coloring this class of graphs.
\section{Preliminary}\label{section-pre}

In this section, we introduce some notation and collect several useful lemmas
that will be used later.
The order of a graph is its number of vertices, and the size is its number of edges. We denote by $V(G)$ and $E(G)$ the vertex set and edge set of a graph $G,$ respectively.  For a vertex $x$ in a graph $G$, the neighborhood and the closed neighborhood of $x$
are denoted by $N_G(x)$ and $N_G[x]$, respectively, and the degree of $x$ is denoted by
$\deg_G(x)$. When $G$ is clear from the context, the subscript $G$ is omitted.
For a path, we denote by $P=p_1p_2\ldots p_k$ the path of order $k$ with $k\geq 1$. If $1\leq i\leq j\leq k$, we denote by $p_iPp_j$ the path $p_ip_{i+1}\ldots p_j$. For a path $P=p_1p_2\ldots p_k$, we call the set $V(P)\setminus \{p_1,p_k\}$ the interior of $P$, and denote it by $P^*$.

A \emph{prism} is a graph consisting of three vertex-disjoint paths
$P_i=a_i\ldots b_i$ $(i=1,2,3)$, each of length at least $1$, such that
$a_1a_2a_3a_1$ and $b_1b_2b_3b_1$ are triangles, and there are no other
edges between the paths (see Figure~\ref{fig:prism}).
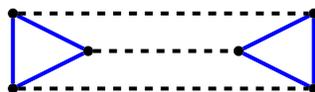
\begin{figure}[!ht]
    \centering
 \begin{tikzpicture}[scale=1,
    v/.style={circle,draw,fill=black,inner sep=1.2pt}]

\node[v] (a1) at (0,1) {};
\node[v] (a2) at (0,0) {};
\node[v] (a3) at (1,0.5) {};

\node[v] (b1) at (4,1) {};
\node[v] (b2) at (4,0) {};
\node[v] (b3) at (3,0.5) {};

\draw[blue, line width=1.5pt] (a1)--(a2)--(a3)--(a1);
\draw[blue, line width=1.5pt] (b1)--(b2)--(b3)--(b1);

\draw[dashed, line width=1.5pt] (a1)--(b1);
\draw[dashed, line width=1.5pt] (a2)--(b2);
\draw[dashed, line width=1.5pt] (a3)--(b3);

\end{tikzpicture}

    \caption{Illustration of a prism.}
    \label{fig:prism}
\end{figure}
The following   two useful lemmas were proved in \cite{Leveque2012}.
\begin{lemma}[L\'ev\^eque et al. \cite{Leveque2012}]\label{lem:ISK4-decomposition}
	Let $G$ be an $\iskfour$-free graph. Then either $G$ is a series-parallel graph, or $G$ contains a prism, a wheel, or a $K_{3,3}$. If $G$ contains a subdivision of $K_{3,3}$, then $G$ contains a $K_{3,3}$.
\end{lemma}
\begin{theorem}[L\'ev\^eque et al. \cite{Leveque2012}]\label{lemma: K33}
 Let $G$ be an $\isk$-free graph containing a $K_{3,3}$. Then $G$ is either a complete bipartite graph, a tripartite graph, or has a clique cutset of size at most $3$.
\end{theorem}
Given a hole $C$ and a vertex $v \not \in C$, $v$ is {\em linked} to $C$
if there are three paths $P_1,P_2,P_3$ such that
\begin{itemize}
\item[$(\romannumeral1)$] $P_1^* \cup P_2^* \cup P_3^* \cup \{v\}$ is disjoint from $C$;
\item[$(\romannumeral2)$] each $P_i$ has one end  $v$ and the other end in $C$, and there are 
no other edges between $P_i$ and $C$;
\item[$(\romannumeral3)$] for $i, j \in [3]$ with $i \neq j$, $V(P_i) \cap V(P_j) = \{v\}$;
\item[$(\romannumeral4)$] if $x \in P_i$ is adjacent to $y \in P_j$ then either $v \in \{x,y\}$
or $\{x,y\} \subseteq V(C)$;
\item[$(\romannumeral5)$] if $v$ has a neighbor $c \in C$, then $c \in P_i$ for some $i$.
\end{itemize}
\begin{lemma}[Chudnovsky et al. \cite{Chudnovsky2019}]\label{lem:nolink}
If $G$ is $\isk$-free, then no vertex of $G$ can be linked to a hole.
\end{lemma}

Let $S=u_1u_2u_3u_4u_1$ be a square in $G$.
A \emph{link} of $S$ is an induced path $P$ with ends $p,p'$ such that
$P^*$ is anticomplete to $V(S)$ and one of the following holds:
$p=p'$ and $N(p)\cap V(S)=V(S)$;
$N(p)\cap V(S)=\{u_1,u_2\}$ and $N(p')\cap V(S)=\{u_3,u_4\}$;
or $N(p)\cap V(S)=\{u_1,u_4\}$ and $N(p')\cap V(S)=\{u_2,u_3\}$. A graph $K$ is called a \emph{rich square} if it contains an induced square $S$
such that $V(S)$ is a vertex cut of $K$ and every component of $K-V(S)$
is a link of $S$. The square $S$ is called the \emph{central square} of $K$.

\begin{lemma}[Chen et al. \cite{Chen2021}]\label{lemma:rich-square-chen}
If $G$ is a rich square, then $G$ contains a diamond or a bowtie.
\end{lemma}

We note that if $L(H)$ is $\db$-free, then $H$ is sparse, which will be used in the next lemma.

\begin{lemma}[L\'ev\^eque et al. \cite{Leveque2012}]\label{lemma: prism0}
	Let $G$ be an $\iskfour$-free graph. If $G$ contains a prism, then either $G$ is the line graph of a graph with maximum degree three or a rich square, or $G$ has a clique cutset of size at most $3$, or $G$ has a proper $2$-cutset.
\end{lemma}

By Lemma~\ref{lemma:rich-square-chen}, a rich square contains a diamond or a bowtie. 
Together with Lemma~\ref{lemma: prism0}, this yields the following result.

\begin{lemma}\label{lemma: prism}
Let $G$ be an $\iskdb$-free graph. If $G$ contains a prism, then either $G$ is the line graph of a sparse graph with maximum degree $3$, or $G$ has a clique cutset of size at most $3$, or $G$ has a proper $2$-cutset.
\end{lemma}

A \emph{separation} of a graph $H$ is a pair $(A, B)$ of subsets of
$V(H)$ such that $A \cup B = V(H)$ and there are no edges between $A
\setminus B$ and $B \setminus A$.  It is \emph{proper} if both $A
\setminus B$ and $B\setminus A$ are non-empty.  The \emph{order} of
the separation is $|A\cap B|$.  A \emph{$k$-separation} is a
separation $(A, B)$ such that $|A \cap B| \leq k$.  A separation $(A,
B)$ is \emph{cyclic} if both $H[A]$ and $H[B]$ has cycles.  A graph
$H$ is \emph{cyclically $3$-connected} if it is $2$-connected, not a
cycle, and there is no cyclic $2$-separation. A \emph{theta} is a graph consisting of two distinct vertices joined by three internally disjoint paths, each of length at least $2$. The structure of cyclically $3$-connected graphs was characterized by L\'ev\^eque et al. \cite{Leveque2012}.

\begin{lemma}[L\'ev\^eque et al. \cite{Leveque2012}]\label{lem:cyclically-3-connected}
	A graph $H$ is cyclically $3$-connected if and only if it is either
	a theta or a subdivision of a $3$-connected graph.
\end{lemma}

A \emph{square theta} is a theta that contains a square.
An induced subdivision $H$ of $K_4$ has four vertices of degree three,
which we call the \emph{corners} of $H$.
A \emph{square subdivision of $K_4$} is a subdivision of $K_4$ whose
corners form a (possibly non-induced) square.
A graph is said to be \emph{substantial} if it is cyclically
$3$-connected and not a square theta or a square subdivision of $K_4$.

\begin{lemma}[L\'ev\^eque et al. \cite{Leveque2012}]
	\label{lem:-substantial-graph}
	Let $G$ be an $\iskfour$-free graph.  Let $H$ be a
	substantial graph such that $L(H)$ is an induced subgraph of $G$ and
	is inclusion-wise maximum with respect to that property.  Then
	either $G = L(H)$, or $G$ has a clique cutset of size at most $3$,
	or $G$ has a proper $2$-cutset.
\end{lemma}

\begin{lemma}[L\'ev\^eque et al. \cite{Leveque2012}]
	\label{lem:no-substantial-graph}
	Let $G$ be an $\iskfour$-free graph that does not contain the line graph
	of a substantial graph or a rich square as an induced subgraph.  Let
	$K$ be a prism that is an induced subgraph of $G$.  Then either
	$G=K$, or $G$ has a clique cutset of size at most $3$, or $G$ has a proper $2$-cutset.
\end{lemma}

Let $ G = (V, E) $ be a graph. A \emph{proper $ k $-edge coloring} of $ G $ is a mapping $ \varphi \colon E \to [k]$ such that $ \varphi(e) \ne \varphi(f) $ for any two adjacent edges $ e $ and $ f $. If such a coloring exists, we say that $ G $ is \emph{$ k $-edge-colorable}. The minimum integer $ k $ for which $ G $ admits a proper $ k $-edge coloring is called the \emph{edge chromatic number} of $ G $, denoted by $ \chi'(G)$.

\begin{lemma}\label{lemma:cubic-edge-color}
	Let $G$ be a graph obtained from a cubic graph by  subdividing only one edge twice and each of the remaining edges once.
	For every $k\ge 3$,
	if $G$ is $k$-edge-colorable, then $G$ has proper $k$-edge coloring $c_i$ for $i\in [2]$ such that $c_1(e_1)=c_1(e_2)$ and $c_2(e_1)\neq c_2(e_2)$, where $e_1$, $e_2$ are two non-adjacent edges obtained by subdividing a single edge twice. 
\end{lemma}
\begin{proof}
Since $G$ is $k$-edge-colorable, fix a proper edge-coloring $\varphi~\colon E(G)\to[k]$.  
Let the three consecutive edges on the twice-subdivided original edge be $e_1=xy$, $e=yz$, $e_2=zw$, where $x,w$ are vertices of degree $3$ and $y,z$ are vertices of degree $2$.

If $\varphi(e_1)=\varphi(e_2)$, then we may assume that $\varphi(e_1)=\varphi(e_2)=1$ and $\varphi(e)=2$.  
Let $H_1$ be the connected component of the subgraph induced by edges of colors $1$ and $3$ that contains $e_1$.  
Since $\varphi(e)=2$, $e\notin E(H_1)$.  
Suppose that $e_2\in E(H_1)$, then there is a path $P$ in $H_1$ between  $x$ and $w$.  
Edges of $P$ alternate colors $1$ and $3$, and since both $e_1$ and $e_2$ have color $1$, the length of $P$ is odd.  
Since $H_1$ is sparse,  the degrees of vertices on $P$ alternate $3,2,3,2,\dots$.
Thus the length of $P$ must be even. 
Therefore $e_2\notin E(H_1)$.  
Thus swapping colors $1$ and $3$ on $H_1$ yields a proper coloring $\varphi'$ with $\varphi'(e_1)\neq\varphi'(e_2)$.  
So in this case we may take $c_1=\varphi$ and $c_2=\varphi'$.

Now suppose $\varphi(e_1)\neq\varphi(e_2)$. For a contradiction, assume that in every proper $k$-edge-coloring of $G$ the edges $e_1$ and $e_2$ receive distinct colors.  
We may assume $\varphi(e_1)=1$, $\varphi(e)=2$, and $\varphi(e_2)=3$.  
Let $H_2$ be the connected component of the subgraph induced by colors $1$ and $3$ that contains $e_1$.  
If $e_2\notin E(H_2)$, then swapping colors $1$ and $3$ on $H_2$ gives a proper coloring in which $e_1$ and $e_2$ have the same color, a contradiction to the assumption.
Hence $e_2\in E(H_2)$, so there is a path $P$ in $H_2$ between $x$ and $w$ whose edges alternate colors $1$ and $3$. 
Since $x,w$ have degree $3$ and $H_2$ is sparse, the vertices along $P$ alternate degrees $3$ and $2$, so $P$ has even length.

Let $H_3$ be the connected component of the subgraph induced by colors $1$ and $2$ that contains $e_1$, and form $\varphi'$ by swapping colors $1$ and $2$ on $H_3$. We claim this swap does not affect the coloring of the path $P$. 
Indeed, if some edge $e_3\in E(P)$ of color $1$ were recolored by this swap, then $e_3\in E(H_3)$ and there would exist in $H_3$ a path $Q$ from $x$ to an end $s$ of $e_3$ whose edges alternate $1$ and $2$. Note that $s$ has degree $3$. 
Since both $e_1$ and $e_3$ have color $1$, the length of $Q$ is odd.  
Since $H_3$ is sparse, the degrees of vertices on $Q$ alternate $3,2,3,2,\dots,3$. 
Thus the length of $Q$ must be even, a contradiction. 
Thus $P$ remains with the same colors under the swap. Without loss of generality, we may assume that $P=xx_1x_2\ldots w$. 
Thus the edge adjacent to $e_1$ on $P$ is $xx_1$.

After this swap, considering the component $H_4$ of colors $2$ and $3$ containing $e_2$ in the new coloring $\varphi'$. 
We claim that $e_1\notin E(H_4)$.
For suppose not, 
since $\varphi'(e_4)=3$, there must be a path colored $2$ and $3$ from $e_2$ to $e_1$ through $e_4$.
But $x_1$ has degree of $2$ and $\varphi'(x_1x_2)=1$, a contradiction.
Thus $e_1\notin E(H_4)$, now swapping colors $2$ and $3$ on $H_4$ yields a proper coloring $\varphi''$ in which $\varphi''(e_1)=\varphi''(e_2)$,
a contradiction to the assumption.

Therefore, there exists a proper edge-coloring assigning the same color to $e_1$ and $e_2$.
Combined with the first part, this gives the two desired colorings $c_1,c_2$ and completes the proof of Lemma~\ref{lemma:cubic-edge-color}.
\end{proof}
\section{Proper wheels}
In this section, we study proper wheels and establish several structural
properties of them.
Let $W=(C,x)$ be a wheel. We call a vertex $v$ \emph{proper} for $W$ if either
$v\in V(W)$, or all neighbors of $v$ in $V(C)$ lie in one sector of $W$, and moreover,
if $v$ has more than two neighbors in $V(C)$, then $v$ is adjacent to $x$.
The wheel $W$ is said to be \emph{proper} if every vertex of $G$ is proper for $W$.
\begin{lemma}\label{lem:proper-wheel-exist}
Let $G$ be an $\iskdbp$-free graph that contains a wheel. Then there is a proper wheel in $G$.
\end{lemma}
\begin{proof}
Let $W=(C,x)$ be a wheel in $G$ with $|V(C)|$ minimum.
We assert that $W$ is a proper wheel. Suppose there exists a vertex $v\in V(G)\setminus V(W)$ that violates the definition of a proper wheel. If $v$ has at least three neighbors on the cycle $xSx$ corresponding to some
sector $S$ of $W$, then $(xSx,v)$ induces either a wheel with a rim shorter than that of $W$, or a $K_4$, a contradiction.  Therefore, 
$v$ has at most two neighbors in every sector of $W$ (and at most one if $v$ is adjacent to $x$). 
\begin{claim}\label{2spokes}
Every path of $C$ whose ends lie in $N(v)$ and whose interior is disjoint from $N(v)$ contains at most two spokes of $W$. 
\end{claim}

Otherwise, assume that there exists a subpath $P$ of $C$ with ends $v_1,v_2\in N(v)$ and with interior disjoint from $N(v)$, such that $P$ contains at least three spokes of $W$. 
Let $P'$ denote the subpath of $C$ obtained from $C$ by deleting the interior of $P$. By the minimality of $W$,  we have that the path $P'$ must have exactly three vertices, say $P'=v_1wv_2$, and we must have $wx\in E(G)$ since $v$ is not proper for $W$. Since $G$ is diamond-free, $wv\notin E(G)$ and $|N(x)\cap \{v_1,v_2\}|\leq 1$. 
By symmetry, we may assume that $xv_1\notin E(G)$. Let $x_1$ be the neighbor of $x$ in $P$ closest to $v_1$, and let $x_2$ be the neighbor of $x$ in $P$ closest to $v_2$. If $xv\in E(G)$, then $x$ can be linked to the hole $vv_1wv_2v$ by $xv$, $xw$ and $xx_2Pv_2$, a contradiction. 
Therefore, $xv\notin E(G)$. However, it follows that $w$ has exactly three neighbors  in the hole $vv_1Px_1xx_2Pv_2v$, a contradiction. This proves Claim~\ref{2spokes}.
\begin{claim}\label{claim:ends}
For each sector $S$ of $W$, either $v$ has a neighbor in $S^*$, 
or $v$ has at most one neighbor in $S$.
\end{claim}

Otherwise, assume that there exists a sector $S$ of $W$ with ends $a_1,b_1$ such
that $v$ has no neighbor in $S^*$ and $v$ has at least two neighbors in $S$.
This implies that $N(v)\cap V(S)=\{a_1,b_1\}$ and $vx\notin E(G)$. Since $G$ is
diamond-free, $S^*\neq \emptyset$. Let $Q=C-S^*$. Then $Q$ is a path with ends
$a_1$ and $b_1$. Since $v$ is not proper for $W$, $v$ has a neighbor in $Q^*$.
Let $v^*$ be the neighbor of $v$ in $Q^*$ closest to $b_1$. Let $S_1$ be a sector
of $W$ with end $b_1$ and such that $V(S_1)\subseteq V(Q)$, and let $a_2$ be the
other end of $S_1$. Let $S_2$ be a sector of $W$ with end $a_2$ and such that
$V(S_2)\subseteq V(Q)$, and let $b_2$ be the other end of $S_2$. Since the path
$b_1Qv^*$ contains at most two spokes of $W$, $v^*\in V(b_1Qb_2)\setminus
\{b_1,b_2\}$. If $v^*\in V(S_2)$, then $x$ can be linked to the hole $vb_1Qv^*v$
by $xb_1$, $xa_2$, and $xa_1v$, a contradiction. Therefore, $v^*\in S_1^*$.
This implies that $v$ can be linked to the hole $xS_1x$ by $vb_1$, $vv^*$, and
$va_1x$, a contradiction. This proves Claim~\ref{claim:ends}.

Note that $v$ is not proper for $W$, and $v$ has at most two neighbors in every
sector of $W$. Thus, there exist two sectors $S_1,S_2$ of $W$ such that $v$ has
a neighbor in $V(S_1)\setminus V(S_2)$ and a neighbor in
$V(S_2)\setminus V(S_1)$. We now consider the following two cases.
\begin{case}\label{case:no-consecutive-sectors}
There are no two consecutive sectors $S_1$ and $S_2$ of $W$ such that $v$ has a
neighbor in $V(S_1)\setminus V(S_2)$ and another neighbor in
$V(S_2)\setminus V(S_1)$.
\end{case}

Now we can choose $S_1,S_2$ and, for $i=1,2$, label the ends of $S_i$ as $a_i,b_i$
such that $b_1\neq a_2$. By Claim~\ref{2spokes}, we may assume that $b_1,a_2$ are
the ends of a sector $S_3$ of $W$, where $v$ has no neighbor in $S_3^*$.
And by Claim~\ref{claim:ends}, we have $N(v) \cap \{a_1,b_1,a_2,b_2\} = \emptyset$. Let $s$ and $t$ be the neighbors of $v$ in $S_1$ closest to $a_1$ and $b_1$, respectively. 
Similarly, let $z$ and $w$ be the neighbors of $v$ in $S_2$ closest to $a_2$ and $b_2$, respectively. By Lemma~\ref{lem:nolink}, $x$  cannot  be linked to the hole $vtS_1b_1S_3a_2S_2zv$, and it follows that $xv\notin E(G)$. If $s = t$ or $st \notin E(G)$, then $x$ can be linked to the hole 
$vtS_1b_1S_3a_2S_2zv$ by $xb_1$, $xa_2$, and either $xa_1S_1t$ (if $s=t$) or $xa_1S_1sv$ ($s \ne t$ and $st \not\in E(G)$), a contradiction. 
Hence $s \neq t$ and $st \in E(G)$. 
By a symmetric argument, $z \neq w$ and $zw \in E(G)$. This implies that $G[\{v,s,t,z,w\}]$ is a bowtie, a contradiction. This completes the proof of Case~\ref{case:no-consecutive-sectors}.

\begin{case}
There exist two consecutive sectors $S_1$ and $S_2$ of $W$ such that $v$ has a neighbor in $V(S_1)\setminus V(S_2)$ and a neighbor in $V(S_2)\setminus V(S_1)$.
\end{case}

Therefore, we may choose $S_1, S_2$ and, for $i=1,2$, label the ends of $S_i$ as
$a_i, b_i$ such that $b_1 = a_2$. Let $s$ and $t$ be the neighbors of $v$ in $S_1$
closest to $a_1$ and $b_1$, respectively, and let $z$ and $w$ be the neighbors of
$v$ in $S_2$ closest to $a_2$ and $b_2$, respectively.

Suppose first that $t=b_1$. By Claim~\ref{claim:ends}, we have that $s\in S_1^*$
and $w\in S_2^*$. Note that $v$ has at most two neighbors in every sector of $W$
(and at most one if $v$ is adjacent to $x$). Then $xv\notin E(G)$. Since $G$ is
diamond-free, either $sb_1\notin E(G)$ or $b_1w\notin E(G)$. By symmetry, we may
assume that $sb_1\notin E(G)$. Now $v$ can be linked to the hole $xS_2x$ by
$vb_1$, $vw$, and $vsS_1a_1x$, a contradiction. Therefore, $t\neq b_1$. By
symmetry, $z\neq b_1$.

Suppose that $vx\in E(G)$. Then $s=t$ and $z=w$. Note that $W$ has at least four
spokes, so $b_2\neq a_1$ and $b_2a_1\notin E(G)$. Since $G$ is diamond-free,
either $s\neq a_1$ or $w\neq b_2$. By symmetry, we may assume that $s\neq a_1$.
Now $v$ can be linked to the hole $xS_2x$ by $vx$, $vz$, and $vsS_1b_1$, a
contradiction. Therefore, $vx\notin E(G)$.

Suppose that $s\neq t$. If $z\neq b_2$ or
$b_1b_2\notin E(G)$, then $v$ can be linked to the hole $xS_1x$ by $vs$, $vt$, and $vzS_2b_1x$ (if $z\neq b_2$) or $vzx$ ($z=b_2$ and $b_1b_2\notin E(G)$), a contradiction. Hence $z=b_2$ and
$b_1b_2\in E(G)$. If $st\notin E(G)$, then $b_2$ can be linked to the hole $xS_1svtS_1b_1x$ by $b_2b_1$,
$b_2v$, and $b_2x$, a contradiction. Therefore $st\in E(G)$ and $G[\{v, x\} \cup V(S_1) \cup V(S_2)]$ is a prism, a contradiction. Thus $s=t$. By symmetry, $z=w$. Then $x$ can be linked to the hole $vtS_1b_1S_2z$ by
$xb_1$, $xa_1S_1t$, and $xb_2S_2z$, a contradiction. This completes the proof of
Lemma~\ref{lem:proper-wheel-exist}.
\end{proof}
\begin{lemma}\label{lem:triangle-wheel-4}
Let $G$ be an $\iskdbp$-free graph, and let $W = (C, x)$ be a proper wheel in $G$. 
Let $u, v \in N(x)$ with $uv \in E(G)$. 
Let $W' = (C', x)$ be a wheel in $G$ containing $u$ and $v$ with exactly four spokes.
Let $S$ be a sector of $W'$ that contains $u$ but not $v$. 
If $V(C') \setminus V(S) \subseteq V(C)$, then there exists a proper wheel in $G$ centered at $x$ with exactly four spokes.
\end{lemma}
\begin{proof}
Let $W' = (C', x)$ be a wheel in $G$ with four spokes containing vertices $u$ and $v$. Let $S$ be a sector of $W'$ containing $u$ but not $v$ such that $V(C') \setminus V(S) \subseteq V(C)$. Choose $W'$ so that $|V(C')|$ is minimum.  For convenience, let $T_1$ and $T_2$ be two sectors of $W'$ such that $V(T_1) \cap \{u, v\} = \{u\}$ and $V(T_2) \cap \{u, v\} = \{v\}$.  
Furthermore, let $w_1$ and $w_2$ denote the ends of $T_1$ and $T_2$ distinct from $u$ and $v$, respectively. That is, $S=T_1$. Let $T_3$ be the sector of $W'$ with ends $w_1$ and $w_2$. We may assume that $w_1, u, v, w_2$ appear in $C'$ in this order. We assert that $W'$ is a proper wheel. Suppose that $y\in V(G)\setminus V(W')$ violates the definition of a proper wheel. 

Since $y$ is proper for $W$, we have $N(y)\cap V(T_1)\neq \emptyset$. 
Let $y_1$ and $y_2$ be the neighbors of $y$ in $T_1$ closest to $u$ and $w_1$, respectively. 
Let $Q=C'-uv$, and let $y_1'$ be the neighbor of $y$ in $V(Q)$ closest to $v$.

Suppose $xy \in E(G)$. Since $G$ is $\db$-free, $N(x)\cap N(y)=\emptyset$. Therefore, $N(y)\cap V(C')\subseteq T_1^*\cup T_2^*\cup T_3^*$. Since $y$ is not proper for $W'$, $y_1'\in T_2^*\cup T_3^*$. If $y_1'\in T_2^*$, then $x$ has exactly three neighbors in  the hole $yy_1T_1uvT_2y_1'y$, a contradiction. Therefore,  $y$ has no neighbor in $V(T_2)$. Suppose $y_1'\in T_3^*$. Let $y_2'$ be the neighbor of $y$ in $T_3$ closest to $w_1$. If $y_1'=y_2'$, then $x$ can be linked to the hole $yy_2T_1w_1T_3y_1'y$ by $xy$, $xw_1$ and $xw_2T_3y_1'$, a contradiction. Then $y_1'\neq y_2'$. Now,  let  $D$ be a hole $uvT_2w_2T_3y_1'yy_1T_1u$, and let $S'$ be the  sector of $(D,x)$ such that $V(S')\cap \{u,v\}=\{u\}$, i.e. $S' = uT_1y_1y$. Now, $(D,x)$ is a wheel in $G$ containing $u$ and $v$ with four spokes such that $V(D)\setminus V(S')\subseteq V(C)$. Since $y_1'\neq y_2'$, $|V(D)|< |V(C')|$. However, this contradicts the choice of $W'$. Therefore, $xy\notin E(G)$.

Suppose that $y_1=w_1$. If $y_1'\in V(T_2)\setminus \{w_2\}$, then $x$  has exactly three neighbors in the hole $yw_1T_1uvT_2y_1'y$,  a contradiction. Hence, $y_1'\in V(T_3)$. Since $y$ is proper for $W$  and $yx\notin E(G)$, $|N(y)\cap ( V(T_3)\setminus \{w_1\})|\leq 2$. By Lemma~\ref{lem:nolink}, we have that $|N(y)\cap  V(T_3)|\leq 2$. This implies that $y$ is proper for $W'$, a contradiction. Therefore, $y_1\neq w_1$.

Suppose that  $y_1'\in V(T_2)$.  Since $y$ is proper for $W$, we have $|N(y) \cap V(T_2)| \leq 2$. Let $y_2'$ be the neighbor of $y$ in $T_2$ closest to $w_2$. If $y_1' = y_2'$ or $y_1'y_2'\notin E(G)$, then $x$ can be linked to the hole $yy_1T_1uvT_2y_1'y$ by $xu$, $xv$, and $xw_2T_2y_1'$ (if $y_1' = y_2'$) or $xw_2T_2y_2'y$ ($y_1'\neq  y_2'$ and $y_1'y_2'\notin E(G)$), a contradiction.  Therefore, $y_1'\neq  y_2'$ and $y_1'y_2'\in E(G)$. Since $G$ is $\db$-free, $y_1'\neq v$. However, it follows that $G[\{y,u,v\} \cup V(y_1T_1u) \cup V(T_2)]$ is a prism, a contradiction. Therefore, $y_1'\notin V(T_2)$.

Suppose that  $y_1'\in V(T_3)\setminus \{w_1\}$. Then $x$ can be linked to the hole $yy_1T_1uvT_2w_2T_3y_1'y$ by $xu$, $xv$, and $xw_2$. Therefore, $N(y) \cap V(C') \subseteq V(T_1)$. Since $yx\notin E(G)$ and $y$ is not proper for $W'$, $|N(y)\cap V(T_1)|\geq 3$. Now,  let  $D$ be a hole $uvT_2w_2T_3w_1T_1y_1'yy_1T_1u$, and let $S'$ be the  sector of $(D,x)$ such that $V(S')\cap \{u,v\}=\{u\}$. Now, $(D,x)$ is a wheel in $G$ containing $u$ and $v$ with four spokes such that $V(D)\setminus V(S')\subseteq V(C)$. Since $G$ is $\db$-free,  $|V(y_1'T_1y_1)|>|V(y_1'yy_1)|$. Then $|V(D)|< |V(C')|$. However, this contradicts the choice of $W'$. This completes the proof of Lemma~\ref{lem:triangle-wheel-4}.
\end{proof}

\begin{lemma}\label{lem:triangle-wheel-4+}
Let $G$ be an  $\iskdbp$-free graph, and let $W=(C,x)$ be a proper wheel in $G$. Let $u,v\in N(x)$ such that $uv\in E(G)$. If there exists a wheel $W'=(C',x)$ of $G$ containing $u,v$ such that $V(C') \setminus \{u,v\} \subseteq V(C)$, then there exists a proper wheel in $G$ with center $x$ and at most the same number of spokes as $W'$. 
\end{lemma}
\begin{proof}
To the contrary, assume that $W' = (C', x)$ is a wheel centered at $x$ containing $u$ and $v$, having the minimum number of spokes among those with $V(C') \setminus \{u,v\} \subseteq V(C)$.  Let $T_1$ and $T_2$ be two sectors of $W'$ such that $V(T_1) \cap \{u, v\} = \{u\}$ and $V(T_2) \cap \{u, v\} = \{v\}$. For convenience, let $w_1$ be the end of $T_1$ distinct from $u$, and let $w_2$ be the end of $T_2$ distinct from $v$.  We may assume that $w_1,u,v,w_2$ appear in $C'$ in this order. If $W'$ is proper in $G$, then $W'$ is the desired proper wheel.
Therefore, assume that there exists $y\in V(G)\setminus V(W')$ that violates the definition of a proper wheel.

Since $y$ is proper for $W$  and is not proper for $W'$, $N(y)\cap \{u,v\}\neq \emptyset$. Without loss of generality, we may assume that $yu\in E(G)$. Since $G$ is $\db$-free, $N(y)\cap \{v,x\}=\emptyset$. Since $y$ is proper for $W$, $|N(y)\cap (V(C')\setminus \{u\})|\leq 2$.  In fact,  $|N(y)\cap (V(C')\setminus \{u\})|=1$. Indeed, if $|N(y)\cap (V(C')\setminus \{u\})|=2$, then $y$ has  exactly three neighbors in $V(C')$, a contradiction.  Since $y$ is not proper for $W'$, it has a neighbor in $V(C')\setminus \{u\}$, and hence has exactly one such neighbor. Let $y'$ be the neighbor of $y$ in $V(C')\setminus \{u\}$, and let $P_{u,y}$ be the path on $C'$ from $u$ to $y$ such that $v\notin P_{u,y}^*$. Furthermore, we choose $y$ such that $|V(P_{u,y})|$ is as small as possible.

Suppose that $y'\in V(T_2)$. Then $x$ can be linked to the hole $yuvT_2y'y$ by $xv$, $xu$ and $xw_2T_2y'$, a contradiction.  Therefore, $y'\notin V(T_2)$. Let $T_3$ be a sector of $W'$ with $V(T_3) \cap V(T_1) = \{w_1\}$, and let $w_3$ (possibly $w_3 = w_2$) be the other end of $T_3$ distinct from $w_1$. Suppose $y'\in V(T_3)$. Then $x$ can be linked to the hole $yuT_1w_1T_3y'y$ by $xu$, $xw_1$ and $xw_3T_3y'$, a contradiction. Now, let $Q$ be a path on $C'$ from $w_3$ to $w_2$ such that $u\notin Q^*$. This implies that $y'\in Q^*$.

Now, let $D$ be the hole $yuT_1w_1T_3w_3Qy'$.
Then $W_1=(D,x)$ is a wheel with fewer spokes than $W'$. 
If $W_1$ is proper, then $W_1$ is the desired wheel.
Therefore,  there exists $z\in V(G)\setminus V(W_1) $ such that $z$ is not proper for $W_1$. Clearly, $z\notin V(W')\setminus V(W_1).$ 

Suppose that $zu\in E(G)$. Since $G$ is $\db$-free, $N(z)\cap \{y,x\}=\emptyset$. Similarly, since $z$ is proper for $W$, $|N(z)\cap (V(T_1)\setminus \{u\})|\leq 2$. By Lemma~\ref{lem:nolink}, we have that $|N(z)\cap V(T_1)|\leq 2$. Then $z$ has a neighbor in $V(D)\setminus V(T_1)$. This implies that $z$ is not proper for $W'$. By  the choice of $y$, we have that $z$ has  exactly  one neighbor $y'$ in $V(D)\setminus \{u\}$. This implies that  $z$ is proper for $W_1$, a contradiction. 

Therefore, $zu\notin E(G)$. This implies that $zy\in E(G)$.  Let $S_y$ be a sector of $W_1$ containing $y'$. Suppose $zv\in E(G)$. Since $G$ is $\db$-free, $zx\notin E(G)$.  Since $z$ is proper for $W$, $|N(z)\cap (V(C')\setminus \{v\})|\leq 2$.  In fact,  $|N(z)\cap (V(C')\setminus \{v\})|=1$. Indeed, if $|N(z)\cap (V(C')\setminus \{v\})|=2$, then $z$ has  exactly three neighbors in $V(C')$, a contradiction.  Now, let $z'$ be the neighbor of $z$ in  $V(C')\setminus \{v\}$. Then  $z'\in V(D)$. Now, $z$ can be linked to the hole $D$ by $zy$, $zz'$ and $zvu$, a contradiction. Therefore, $zv\notin E(G)$.

Since $z$ is proper for $W$ and $z$ is not proper for $W_1$, $z$ has no neighbor in $(V(C)\setminus V(D))\cup V(S_y)$.  If $z$ has exactly  one neighbor $z'$ in $V(D)\setminus \{y\}$, then $y$ can be linked to the hole $V(C')$ by $yu$, $yy'$ and $yzz'$, a contradiction. If $z$ has exactly two neighbors in $V(D)\setminus \{y\}$, then $z$ has exactly three neighbors in $V(D)$, a contradiction. Therefore, $z$ has two non-adjacent neighbors $z_1,z_2$ in $V(D)\setminus \{y\}$. However, 
it follows that $y$ can be linked to the hole obtained from $C'$ by rerouting $C'$ through $z$ by $yz$, $yu$ and $yy'$, a contradiction. This completes the proof of Lemma~\ref{lem:triangle-wheel-4+}.
\end{proof}
\begin{lemma}\label{lem:edge-N(x)}
Let $G$ be an $\iskdbp$-free graph, and let $W = (C, x)$ be a proper wheel in $G$ with the minimum number of spokes subject to having center $x$. 
Let $u, v \in N(x) \setminus V(C)$ with $uv \in E(G)$. 
If $W$ has at least five spokes, then at most one of the vertices $u$ and $v$ has a neighbor in $V(C)$. 
\end{lemma}
\begin{proof}
To the contrary, we may assume that both $u$ and $v$ have neighbors in $V(C)$.  Since $W$ is proper in $G$, there exist two sectors $S_1$ and $S_2$ of $W$ (possibly $S_1 = S_2$) such that $N(u) \cap V(C) \subseteq V(S_1)$ and $N(v) \cap V(C) \subseteq V(S_2)$. For $i = 1, 2$, let $a_i$ and $b_i$ be the ends of $S_i$. 
If $S_1 \ne S_2$, then we may assume that $a_1, b_1, a_2, b_2$ appear on $C$ in this order and that $a_1 \ne b_2$. 
If $S_1 = S_2$, then we may assume that $a_1 = a_2$ and $b_1 = b_2$. Since $G$ is diamond-free, $\{u,v\} $ is anticomplete to $N(x)\setminus \{u,v\}$ and $N(u)\cap N(v)=\{x\}$. 

Suppose that $S_1\neq S_2$.  Let $u_1$ be the neighbor of $u$ in $S_1$ closest to  $b_1$, and let $v_1$ the neighbor
of $v$ in $S_2$ closest to $a_2$.
Since $W $ is a proper wheel in $G$ with the minimum number of spokes subject to having center $x$ and $W$ has at least five spokes,
by Lemmas~\ref{lem:triangle-wheel-4} and \ref{lem:triangle-wheel-4+}, we have that $b_1=a_2$. Now, $x$ has exactly three neighbors in the hole $uvv_1S_2a_2S_1u_1u$, a contradiction. Therefore, $S_1=S_2$.

Since $u$ and $v$ have no common neighbor in $S_1$, there exist vertices $u', v' \in V(S_1^*)$ such that $uu', vv' \in E(G)$, and there are no edges between the interior of $u'S_1v'$ and $\{u, v\}$. We may assume that $a_1, u',v', b_1$ appear on $S_1$ in this order. If there are no edges between the interior of $a_1S_1u'$ and $\{u, v\}$, then $x$ can be linked to the hole $uu'S_1v'vu$ by $xu$, $xv$ and $xa_1S_1u'$, a contradiction. Therefore,   there are edges between the interior of $a_1S_1u'$ and $\{u, v\}$. Let $w \in (N(u) \cup N(v)) \cap V(a_1S_1u')$ be the vertex closest to $a_1$ along $S_1$. Then $w\neq u'$. If $wu' \notin E(G)$, then $x$ can be linked to the hole  $uu'S_1v'vu$ by $xu$, $xv$ and $xa_1S_1wu$ (or $xa_1S_1wv$), a contradiction. Therefore, $wu'\in E(G)$. Since $G$ is bowtie-free, $N(w)\cap \{u,v\}=\{v\}$. However, it follows that $x$ can be linked to the hole $uvwu'u$  by $xu$, $xv$ and $xa_1S_1w$,  a contradiction. This completes the proof of Lemma~\ref{lem:edge-N(x)}.
\end{proof}
\begin{lemma}\label{lem:K13}
Let $G$ be an $\iskdbp$-free graph, and let $C=v_1v_2\ldots v_nv_1$ be a hole.
Suppose that $H$ is a subdivision of $K_{1,3}$ with leaves $a,b,c$ and center $x$ such that
$V(H)\setminus \{a,b,c\}$ is anticomplete to $V(C)$, and the neighborhoods of $a,b,c$ on $C$
are arranged as depicted in Figure~\ref{Fig:abc}.
Moreover, assume that for any two distinct vertices $u,v\in\{a,b,c\}$, the following hold:
\begin{itemize}
\item[$(\romannumeral1)$] $N(u)\cap V(C)\neq N(v)\cap V(C)$, and
\item[$(\romannumeral2)$] $|N(u)\cap N(v)\cap V(C)|\leq 1$.
\end{itemize}
Then at least one of $a,b,c$ has no neighbor in $V(C)$.
\end{lemma}
\begin{figure}[ht]
\centering
\begin{tikzpicture}[scale=1.3,
    every node/.style={font=\small},
    dot/.style={circle,fill,inner sep=1.2pt}
]

\draw[dashed, thick] (0,0) circle (1.6);

\node[dot,label=right:$x$] (x) at (0,0) {};

\node[dot] (v1) at ({1.6*cos(-90)},{1.6*sin(-90)}) {};
\node[dot] (v2) at ({1.6*cos(30)},{1.6*sin(30)}) {};
\node[dot] (v3) at ({1.6*cos(150)},{1.6*sin(150)}) {};

\coordinate (v4) at ({1.6*cos(90)},{1.6*sin(90)});
\coordinate (v5) at ({1.6*cos(60)},{1.6*sin(60)});
\coordinate (v6) at ({1.6*cos(120)},{1.6*sin(120)});

\coordinate (v7) at ({1.6*cos(210)},{1.6*sin(210)});
\coordinate (v8) at ({1.6*cos(180)},{1.6*sin(180)});
\coordinate (v9) at ({1.6*cos(240)},{1.6*sin(240)});

\coordinate (v10) at ({1.6*cos(330)},{1.6*sin(330)});
\coordinate (v11) at ({1.6*cos(360)},{1.6*sin(360)});
\coordinate (v12) at ({1.6*cos(300)},{1.6*sin(300)});

\node[dot,label=right:$a$] (a) at (0,0.9) {};
\node[dot,label=below left:$b$]  (b) at ({0.9*cos(210)},{0.9*sin(210)}) {};
\node[dot,label=below right:$c$] (c) at ({0.9*cos(330)},{0.9*sin(330)}) {};

\draw[dashed] (x)--(a);
\draw[dashed] (x)--(b);
\draw[dashed] (x)--(c);

\draw[densely dotted] (a)--(v4);
\draw[densely dotted] (a)--(v5);
\draw[densely dotted] (a)--(v6);

\draw[densely dotted] (b)--(v7);
\draw[densely dotted] (b)--(v8);
\draw[densely dotted] (b)--(v9);

\draw[densely dotted] (c)--(v10);
\draw[densely dotted] (c)--(v11);
\draw[densely dotted] (c)--(v12);

\end{tikzpicture}
\caption{The neighbors of $a$, $b$, and $c$ on the hole $C$ and their locations on $C$.}
\label{Fig:abc}
\end{figure}
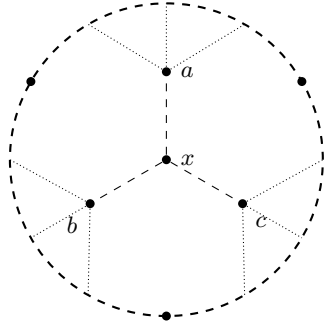
\begin{proof}
To the contrary, assume that each of $a$, $b$, and $c$ has a neighbor on $C$. Let $P_a$ be the path in $H$ with ends $a$ and $x$, 
$P_b$ the path with ends $b$ and $x$, and 
$P_c$ the path with ends $c$ and $x$. 
\begin{claim}\label{claim:u-neighbor}
For any $u\in \{a,b,c\}$. If $u$ has two neighbors in $V(C)$, then $u$ has two non-adjacent neighbors in $V(C)$.
\end{claim}

Otherwise, assume that $a$ has exactly two adjacent neighbors $a_1$ and $a_2$ in $V(C)$. 
First consider the case where $N(b)\cap V(C)\subseteq \{a_1,a_2\}$ and $N(c)\cap V(C)\subseteq \{a_1,a_2\}$. By $(\romannumeral1)$ and $(\romannumeral2)$, we may assume that $a_1b\in E(G)$ and $a_2c\in E(G)$. This implies that  $a$ can be linked to the hole $a_1a_2cP_cxP_bba_1$ by $aa_1$, $aa_2$ and $aP_ax$, a contradiction. Therefore, either $N(b)\cap V(C)\not\subseteq \{a_1,a_2\}$ or $N(c)\cap V(C)\not\subseteq \{a_1,a_2\}$. Without loss of generality, assume that $N(b)\cap V(C)\not\subseteq \{a_1,a_2\}$. That is, $b$ has a neighbor in $V(C)\setminus \{a_1,a_2\}$.  If $b$ has exactly one neighbor $b_1$ in $V(C)$ and hence $b_1\notin \{a_1,a_2\}$, then $a$ can be linked to the hole $C$ by $aa_1$, $aa_2$, and $aP_axP_bbb_1$, a contradiction.  If $b$ has two  non-adjacent neighbors in $V(C)$, then $a$ can be linked to the hole obtained from  $C$ by rerouting $C$ through $b$ by $aa_1$, $aa_2$, and $aP_axP_bb$, a contradiction. Therefore, $b$ has exactly two adjacent  neighbors  in $V(C)$. Since $G$ is bowtie-free, $N(b)\cap N(a)\cap V(C)=\emptyset$. But now, $G[V(P_a)\cup V(P_b)\cup V(C)]$ is a prism, a contradiction. This proves Claim~\ref{claim:u-neighbor}.

If $a$, $b$, and $c$ each has a unique neighbor $a'$, $b'$, and $c'$ in $C$, 
then $x$ can be linked to $C$ by $xP_a a a'$, $xP_b b b'$, and $xP_c c c'$, 
contradicting Lemma~\ref{lem:nolink}. 
If some of $a$, $b$, or $c$ has several neighbors in $C$, 
then similar linkages work for the holes obtained from $C$ by rerouting the cycle through 
$a$, $b$, and $c$, respectively. 
This completes the proof of Lemma~\ref{lem:K13}.
\end{proof}
\begin{lemma}\label{lem:neighbor-consecuitive-sector}
Let $G$ be an $\iskdbp$-free graph.
Let $W=(C,x)$ be a wheel in $G$.
Let $S_1,S_2$ be consecutive sectors of $W$,  and let $v\in N(x)\setminus V(C)$  such that $N(v)\cap V(C)\subseteq V(S_1)\cup  V(S_2)$. Suppose  $v$ has neighbors in $V(S_1)\setminus V(S_2)$ and $V(S_2)\setminus V(S_1)$.
Then, $|N(v)\cap (V(S_1)\setminus V(S_2))|\geq 2$ and $|N(v)\cap (V(S_2)\setminus V(S_1))|\geq 2$. Furthermore, if $|N(v)\cap (V(S_i)\setminus V(S_{3-i}))|=2$ for $i \in \{1,2\}$, then $V(S_1)\cap V(S_2)\subseteq N(v)$.
\end{lemma}
\begin{proof}
For convenience, we may assume that $V(S_1)\cap V(S_2)=\{y\}$. Suppose for a contradiction that $|N(v) \cap (V(S_1) \setminus V(S_2))|=1$. If $vy \in E(G)$, then $v$ has exactly three neighbors in the hole $xS_1x$, a contradiction. Therefore, $vy\notin E(G)$. Let $z$ denote the neighbor of $v$ in $V(S_1)$, and let $w$ denote the neighbor of $v$ in $V(S_2)$ closest to $y$. For $i=1,2$, denote the end of $S_i$ different from $y$ by $a_i$. If $w\neq a_2$, then $v$ can be linked to the hole $xS_1x$ by $vz$, $vx$, and $vwS_2y$, a contradiction. Therefore, $w=a_2$.  Since $G$ is diamond-free, $z\neq a_1$. However, it follows that $v$ can be linked to the hole $xS_2x$ by $va_2$, $vx$, and $vzS_1y$, contradicting Lemma~\ref{lem:nolink}.  Furthermore, if $|N(v)\cap (V(S_i)\setminus V(S_{3-i}))|=2$ and $yv\notin E(G)$, then $v$ has exactly three neighbors in the hole $xS_ix$, a contradiction. This completes the proof of Lemma~\ref{lem:neighbor-consecuitive-sector}.
\end{proof}
\begin{definition}[Chudnovsky et al. \cite{Chudnovsky2019}] 
\label{def:skip}
A vertex $v$ is a \textit{skip} for a wheel $W=(C,x)$ if there exist two consecutive sectors $S_1,S_2$ of $W$ such that 
\begin{itemize}
    \item[$(\romannumeral1)$] $vx\in E(G)$;
    \item[$(\romannumeral2)$] $v$ has neighbors in $V(S_1)\setminus V(S_2)$ and $V(S_2)\setminus V(S_1)$;
    \item[$(\romannumeral3)$] $N(v)\cap V(C)\subseteq V(S_1)\cup V(S_2)$; and
    \item[$(\romannumeral4)$] if $u\in V(G)\setminus V(W)$ is adjacent to $v$, then $N(u)\cap V(C)\subseteq V(S_1)\cup V(S_2)$.
\end{itemize}
If $V(S_1)\cap V(S_2)=\{a\}$, we also say that $v$ is an $a$-skip.
\end{definition}
\begin{definition}[Chudnovsky et al. \cite{Chudnovsky2019}]\label{def:k-almost}
Let $k$ be a positive integer.
We say that wheel $W=(C,x)$ is \emph{$k$-almost proper} if there are spokes $x_1, \dots, x_k$ of $W$ and a set $X \subseteq V(G) \setminus V(W)$ such that
\begin{itemize}
\item[$(\romannumeral1)$] no two spokes in $\{x_1, \dots, x_k\}$ are consecutive;
\item[$(\romannumeral2)$] $W$ is proper in $G\setminus X$; and
\item[$(\romannumeral3)$] for every $v$ in $X$, there exists $i\in[k]$ such that $v$ is an $x_i$-skip.
\end{itemize}
\end{definition}
\begin{lemma}\label{lem:1-almost}
Let $G$ be an $\iskdbp$-free graph, and let $W=(C,x)$ be a $1$-almost proper wheel in $G$ with $x_1$ and $X$ as in  Definition~$\ref{def:k-almost}$. Let $S_1$ and $S_2$ be the sectors of $W$ containing $x_1$.
Then, there exists a proper wheel $W'$ in $G$ with center $x$ and the same number of spokes as $W$. Moreover, either $W = W'$ or $V(W')\setminus V(W)=\{ v^*\}$, where $v^*$ is a skip for $W$, and $V(W)\setminus V(W')\subseteq V(S_1^*)\cup V(S_2^*)\cup\{x_1\}$.
\end{lemma}
\begin{proof}
We may assume that $X\neq\emptyset$, for otherwise $W$ is proper in $G$. 
For each $v\in X$, let $P(v)$ be a longest path in $G[V(S_1)\cup V(S_2)]$ whose ends are neighbors of $v$. 
Choose $v^*\in X$ such that $|V(P(v^*))|$ is maximum, and let $Y$ be the interior of $P(v^*)$. Let $C' = G[(V(C)\cup\{v^*\})\setminus Y]$. 
Then $W'=(C',x)$ is a wheel. Moreover,
\[
N(x)\cap V(C') = \biggl( (N(x)\cap V(C))\setminus\{x_1\} \biggr)\cup\{v^*\},
\]
and hence $W'$ has the same number of spokes as $W$.

If $W'$ is proper, then the conclusion follows. 
Thus we may assume that there exists a vertex 
$y\in V(G)\setminus V(W')$ that is not proper for $W'$. 
Let $S_1',S_2'$ be the sectors of $W'$ containing 
$S_1\setminus Y$ and $S_2\setminus Y$, respectively.

We first suppose that $y\in V(W)$, and hence $y\in Y$. 
By Lemma~\ref{lem:neighbor-consecuitive-sector}, 
the vertex $v^*$ has at least two neighbors in each of $S_1$ and $S_2$, 
and therefore $|V(P(v^*))|\ge 4$. 
It follows that either 
$N(y)\cap V(C')\subseteq V(S_1')$ 
or 
$N(y)\cap V(C')\subseteq V(S_2')$. 
Moreover, $N(y)\cap V(C')\subseteq N[v^*]$, and hence 
$|N(y)\cap V(C')|\le 2$. 
Thus $y$ is proper for $W'$, a contradiction. 
This shows that $y\notin V(W)$.

Next, suppose that $y\in X$, that is, $y$ is an $x_1$-skip for $W$. 
Hence $yx\in E(G)$ and 
$N(y)\cap V(C')\subseteq V(S_1')\cup V(S_2')$. 
It follows that
\[
N(y)\cap (V(S_1')\setminus V(S_2'))\neq\emptyset
\quad\text{and}\quad
N(y)\cap (V(S_2')\setminus V(S_1'))\neq\emptyset .
\]
By Lemma~\ref{lem:neighbor-consecuitive-sector}, 
$y$ has at least two neighbors in each of 
$V(S_1')\setminus V(S_2')$ and $V(S_2')\setminus V(S_1')$. 
But then $|V(P(y))|>|V(P(v^*))|$, contradicting the choice of $v^*$. 
Hence $y\notin X$.

Therefore, $y\in V(G)\setminus (X\cup V(W))$. 
If $yv^*\notin E(G)$, then 
$N(y)\cap V(C')\subseteq N(y)\cap V(C)$. 
Since $y\notin X$, it follows that $y$ is proper for $W$, 
and hence also proper for $W'$, a contradiction. 
Thus $yv^*\in E(G)$.

Since $v^*$ is a skip for $W$, we have
$N(y)\cap V(C)\subseteq V(S_1)\cup V(S_2)$. 
As $y$ is proper for $W$, we may assume by symmetry that
$N(y)\cap V(C)\subseteq V(S_1)$. 
Consequently,
$N(y)\cap V(C')\subseteq V(S_1')$.

If $yx\in E(G)$, then $y$ is proper for $W'$, a contradiction. 
Hence $yx\notin E(G)$, and thus $|N(y)\cap V(C)|\le 2$. 
It follows that $|N(y)\cap V(C')|\le 3$. 
By Lemma~\ref{lem:nolink}, $y$ cannot have exactly three neighbors in $C'$. 
Therefore, $|N(y)\cap V(C')|\le 2$, and $y$ is proper for $W'$, a contradiction. We conclude that $W'$ is proper in $G$, and hence $W'$ is a desired wheel. 
This completes the proof of Lemma~\ref{lem:1-almost}.
\end{proof}
\begin{lemma}\label{lem:2-almost}
Let $G$ be an  $\iskdbp$-free graph, and let $W=(C,x)$ be a $2$-almost proper wheel in $G$.
Then, there exists a proper wheel $W'$ in $G$ with center $x$ and at most the same number of spokes as $W$.
\end{lemma}
\begin{proof}
Let $x_1,x_2$ and $X$ be as in Definition~\ref{def:k-almost} of a $2$-almost proper wheel, 
and let $S_1,S_2$ be the sectors of $W$ containing $x_1$. 
Let $X_1$ be the set of $x_1$-skip vertices in $X$, and let $X_2 = X\setminus X_1$. 
We may assume that both $X_1$ and $X_2$ are non-empty; otherwise, the conclusion follows from Lemma~\ref{lem:1-almost}.  

Then $W$ is $1$-almost proper but not proper in $G-X_2$. 
Let $W'=(C',x)$ and $v^*$ be as in the conclusion of Lemma~\ref{lem:1-almost}. 
Thus $W'$ is a proper wheel in $G-X_2$.

If $W'$ is $1$-almost proper in $G$, then the conclusion of Lemma~\ref{lem:2-almost} follows from Lemma~\ref{lem:1-almost}. 
Hence we may assume that $W'$ is not $1$-almost proper in $G$. 
Since every vertex in $V(G)\setminus X_2$ is proper for $W'$, there exists a vertex $v\in X_2$ that is neither proper nor an $x_2$-skip for $W'$.

By the definition of $X_1$, and since $W$ is $2$-almost proper in $G$, we have
\[
N(v)\cap V(C)\subseteq V(S_3)\cup V(S_4),
\]
where $S_3$ and $S_4$ are the sectors of $W$ containing $x_2$.
Since $x_1$ and $x_2$ are not consecutive, we have $S_3,S_4\notin\{S_1,S_2\}$. 
Therefore, by Lemma~\ref{lem:1-almost}, $V(S_3)\cup V(S_4)\subseteq V(W')$. 
Consequently, $S_3$ and $S_4$ are sectors of $W'$.

Suppose that $vv^*\in E(G)$.
Since $v$ is an $x_2$-skip of $W$,
 $N(v^*)\cap V(C)\subseteq V(S_3)\cup V(S_4)$. Since $v^*$ is an $x_1$-skip of $W$,
 $N(v^*)\cap V(C)\subseteq V(S_1)\cup V(S_2)$. This implies that $$N(v^*)\cap V(C)\subseteq (V(S_1)\cup V(S_2))\cap (V(S_3)\cup V(S_4)).$$ Hence $N(v^*)\cap V(C)\subseteq N(x)$. However, it follows that $G$ contains a diamond, a contradiction. Therefore, $vv^*\notin E(G)$.

Since $v$ is an $x_2$-skip of $W$, $N(v)\cap V(C')\subseteq V(S_3)\cup V(S_4)$. Note that $v$ has a neighbor in both $V(S_3)\setminus V(S_4)$ and  $V(S_4)\setminus V(S_3)$, where $S_3,S_4$ are the sectors of $W'$ containing $x_2$. For $i \in \{3,4\}$, let $s_i$ denote the neighbor of $v$ in $S_i$ furthest from $x_2$. Among all vertices of $X_2$ that are neither proper nor $x_2$-skips for $W'$, choose $v$ so that the path $s_3Cs_4$ containing $x_2$ is maximal. Since $v$ is not an $x_2$-skip for $W'$,
there exists a vertex $u \in N(v) \setminus V(W')$ with a neighbor in $V(C') \setminus (V(S_3) \cup V(S_4))$ by Definition~\ref{def:skip}~$(\romannumeral4)$.
Since $v$ is an $x_2$-skip for $W$, $uv^*\in E(G)$. 
Now, since $v$ and $v^*$ are skips for $W$, it follows that 
$N(u) \cap V(C) \subseteq (V(S_1) \cup V(S_2)) \cap (V(S_3) \cup V(S_4))\subseteq N(x)$. Since $G$ is diamond-free, $ux\notin E(G)$, and therefore 
$u \notin X$. Consequently, $u$ is proper for $W$, and all the neighbors of 
$u$ in $C$ belong to one sector of $W$. It follows that $u$ has at most one neighbor in $V(C)$. Suppose that  $u$ has exactly one neighbor in $V(C)$. Then $u$ has three neighbors in the cycle arising from $C'$ by replacing $s_3S_3x_2S_4s_4$ by $s_3vs_4$, a contradiction. It follows that $u$ has no neighbors in $V(C)$. 

Let $P_1'$ be the subpath of $C'$ from $s_3$ to $v^*$ that does not contain $x_2$, and let $P_1 = vs_3P_1'v^*$. Similarly, let $P_2'$ be the subpath of $C'$ from $s_4$ to $v^*$ that does not contain $x_2$, and let $P_2 = vs_4P_2' v^*$. Define $D = G[V(P_1) \cup \{u\}]$.  

Since $x_1$ and $x_2$ are not consecutive, each of $P_1^*$ and $P_2^*$ contains at least one neighbor of $x$. Therefore, $W'' = (D, x)$ forms a wheel with fewer spokes than $W$. Let $S_3'$ denote the sector of $W''$ that contains $v$ but does not contain $u$.

If $W''$ is proper in $G$, then the conclusion follows. Therefore, we may assume that there is a vertex $y\in V(G) \setminus V(W'')$ that is not proper for $W''$. Since every vertex in
$V(W) \setminus V(W'')$ has at most two consecutive neighbors in $D$, it follows that $y \notin V(W)$.

Suppose that $yu\notin E(G)$. If $y \in X_2$, then $N(y) \cap V(D) \subseteq V(S_3')\cup \{v^*\}$ and $yx\in E(G)$. Clearly, $yv^*\notin E(G)$. Otherwise, $N(y)\cap V(C)\subseteq  (V(S_1) \cup V(S_2)) \cap (V(S_3) \cup V(S_4))\subseteq N(x)$, which contradicts the fact that $G$ is diamond-free.
 This implies that $N(y) \cap V(D) \subseteq V(S_3')$. So $y$ is proper for $W''$, a contradiction. Thus $y \notin X_2$, and so $y$ is proper for $W'$.  If $yv\notin E(G)$, then $N(y) \cap V(D) \subseteq N(y) \cap V(C')$, and again $y$ is proper for $W''$, a contradiction. Thus  $yv\in E(G)$. Since $v$ is an $x_2$-skip for $W$, $N(y) \cap V(C) \subseteq V(S_3) \cup V(S_4)$, and so $N(y) \cap V(D) \subseteq V(S_3')$.
Since $y$ is not proper for $W''$, $yx\notin E(G)$ and has at least three neighbors in $S_3'$.
But $y$ is proper for $W'$, and so $y$ has at most two neighbors in $S_3$; thus $y$ has exactly three neighbors in $S_3'$ and hence in $D$, a contradiction. This implies that $yu\in E(G)$. Since $G$ is diamond-free, $|N(y)\cap \{v^*,v,x\}|\leq 1$. 

Since $y$ is not proper for $W''$, it follows that $y$ has a neighbor in $P_1^*$. Suppose $yv^*\in E(G)$.  Since $v^*$ is an $x_1$-skip for $W$, $N(y)\cap V(C)\subseteq V(S_1)\cup V(S_2)$, and so $N(y)\cap V(D)\subseteq V(S_2)\cup \{v^*\}$. Let $a_1$ be the end of $S_2$ different from $x_1$. Let $y'$ denote the neighbor of $y$ in $V(S_2)\cap D$ closest to $a_1$. Since $G$ is diamond-free, $y'v^*\notin E(G)$. This implies that $y$ can be linked to the hole $xvuv^*x$ by $yv^*$, $yu$ and $yy'S_2a_1x$, a contradiction. Therefore, $yv^*\notin E(G)$. Similarly, we have that $yv\notin E(G)$.

Now we claim that $y$ has no neighbor in $P_2$. Otherwise, assume that $y$ has a neighbor in $P_2$. If $y \in X_2$, then, since $yu\in E(G)$ and $y$ has a neighbor in $P_1^*$, we deduce that $y$ is not an $x_2$-skip for $W'$. By the choice of $v$, we have $N(y)\cap V(S_3')=\{s_3\}$. Since $G$ is diamond-free, $N(s_3)\cap \{v^*,x,u\}=\emptyset$. Now, $y$ can be linked to the hole $uv^*xvu$ by three paths $yu$, $yx$, and $ys_3v$, a contradiction. Therefore, $y \notin X_2$. Consequently, $y$ is proper for $W'$. Note that $y$ has a neighbor in both $P_1^*$ and $P_2^*$, a contradiction. Therefore, $y$ has no neighbor in $P_2$.

Let $z_1$ be the neighbor of $y$ in $V(P_1)$ closest to $v$, and let $z_2$ be the neighbor of $y$ in $V(P_1)$ closest to $v^*$. Let $D'$ be the hole $v^*P_2vuv^*$. If $z_1 \neq z_2$ and $z_1z_2\notin E(G)$, we can link $y$ to $D'$ by $yz_1P_1v$, $yz_2P_1v^*$, and $yu$; if $z_1 \neq z_2$ and $z_1z_2\in E(G)$, we can link $y$ to $D$ by $yz_1$, $yz_2$, and $yu$; and if $z_1=z_2$, then we can link $z_1$ to $D'$ by $z_1P_1v$, $z_1P_1v^*$, and $z_1yu$, which leads to a contradiction in all cases. This completes the proof of Lemma~\ref{lem:2-almost}.
\end{proof}
\begin{lemma}\label{lem:unique}
 Let $G$ be an $\iskdbp$-free graph, and let $W=(C,x)$ be a proper wheel in $G$ of center $x$ with a minimum number of spokes. Let $P = p_1 \cdots p_k$ be a path with $V(P) \subseteq V(G) \setminus V(W)$, whose ends $p_1$ and $p_k$ have neighbors in $V(C)$, and whose internal vertices have no neighbors in $C$.
 Assume that no sector of $W$ contains all vertices in $(N(p_1)\cup N(p_k))\cap V(C)$. For $i\in \{1,k\}$, if $xp_i\notin E(G)$, then $p_i$ has a unique neighbor in $V(C)$.
\end{lemma}
\begin{proof}
Let $S_1,S_2$ be distinct sectors of $W$ such that $N(p_1)\cap V(C)\subseteq V(S_1)$ and $N(p_k)\cap V(C)\subseteq V(S_2)$. Without loss of generality, we may assume that $p_1x\notin E(G)$, so $p_1$ has at most two neighbors in $C$.  

Since no sector of $W$ contains $(N(p_1)\cup N(p_k))\cap V(C)$, $p_k$ has a neighbor in $V(S_2)\setminus V(S_1)$. Suppose $p_k$ has exactly one neighbor in $S_2$ or has two non-adjacent neighbors in $S_2$. Since $p_1$ cannot be linked to the hole $C$ (or the hole obtained from $C$ by rerouting $S_2$ through $p_k$) by two one-edge paths and $P$, it follows that $p_1$ has a unique neighbor in $C$. Thus, $p_k$ has two adjacent neighbors in $S_2$. Since $p_k$ cannot be linked to the hole $xS_2x$, $p_kx\notin E(G)$.  

Similarly, we have that $p_1$ has exactly two adjacent neighbors in $S_1$. Since $G$ is $\db$-free, $N(p_1)\cap N(p_k)\cap V(C)=\emptyset$. However, it follows that $G[V(C)\cup V(P)]$ is a prism, a contradiction. This completes the proof of Lemma~\ref{lem:unique}.
\end{proof}
\begin{lemma}\label{lem:proper-wheel-path-4-spokes}
Let $G$ be an   $\iskdbpk$-free graph, and let $W=(C,x)$ be a proper wheel in $G$ of center $x$ with four spokes.  Let $P=p_1\cdots p_k$ be a path with $V(P)\subseteq V(G)\setminus V(W)$ such that $x$ has at most one neighbor in $P$. 
\begin{itemize}
\item[$(\romannumeral1)$] If $P$ contains no neighbor of $x$, then there is a sector $S$ of
$W$ such that every edge from $P$ to $C$ has an end in $V(S)$.
\item[$(\romannumeral2)$] If $P$ contains exactly one neighbor of $x$, then there are two
sectors $S_1,S_2$ of $W$ such that $V(S_1) \cap V(S_2) \neq \emptyset$, and
 every edge from $P$ to $C$ has an end in $V(S_1) \cup V(S_2)$ (where possibly $S_1=S_2$).
\end{itemize}  
\end{lemma}
\begin{proof}
Let $P$ be a path violating the assertions of the lemma and assume that $P$ is chosen with $k$ minimum. 
Since $W$ is proper, it follows that $k>1$. 
For convenience, we may assume that $N(x)\cap V(C)=\{a_1,a_2,a_3,a_4\}$ and $a_1,a_2,a_3,a_4$ appear on $C$ in this order. 
Let $T_i$ be the sector of $W$ with ends $a_i$ and $a_{i+1}$, where the subscripts are taken modulo $4$.

Suppose that $P^*$ is anticomplete to $V(C)$. 
By Lemma~\ref{lem:unique}, we may assume that $p_k$ has a unique neighbor  $z$ in $V(C)$. Without loss of generality, we may assume that $N(p_k)\cap V(C)\subseteq V(T_1)$. If $p_1x\notin E(G)$, then $p_1$ has a unique neighbor  $s$ in $V(C)$ by Lemma~\ref{lem:unique}.
Note that no sector of $W$ contains all vertices in $(N(p_1)\cup N(p_k))\cap V(C)$.  This implies that $s\notin V(T_1)$. 
If $s\in V(T_4)$, then  $z\neq a_1$ and $x$ has exactly three neighbors in  the hole $p_1Pp_kzT_1a_2T_2a_3T_3a_4T_4sp_1$, a contradiction.
Thus by symmetry, we may assume  that $s\in T_3^*$ and $z\in T_1^*$. However, it follows that $x$ can be linked to the hole $p_1Pp_kzT_1a_2T_2a_3T_3sp_1$ by $xa_2$, $xa_3$ and $xa_1T_1z$, a contradiction. 
Therefore, $xp_1\in E(G)$. This implies that $N(p_1)\cap V(C)\subseteq V(T_3)$. 
In particular, $z\in T_1^*$ and $p_1$ has a neighbor in $T_3^*$.
Let $u$ be the neighbor of $p_1$ in $T_3$ closest to $a_3$. 
Then $u\neq a_4$. Now, $p_1$ can be linked to the hole $xT_2x$ by $p_1x$, $p_1Pp_kzT_1a_2$ and $p_1uT_3a_3$, a contradiction. 
Therefore, $P^*$ is not anticomplete to $V(C)$; consequently, there exists $p_i\in V(P^*)$ with a neighbor in $V(C)$. 

Suppose that $x$ has no neighbor in $V(P)$. Then, by the minimality of $k$, we may assume that $p_1$ has a neighbor in $V(T_1)\setminus V(T_2)$, 
$p_k$ has a neighbor in $V(T_2)\setminus V(T_1)$, 
and every edge between $P^*$ and $V(C)$ has an end at $a_2$. Let $s$ be the neighbor of $p_1$ in $V(T_1)$ closest to $a_1$, and let $z$ be the neighbor of $p_k$ in $V(T_2)$ closest to $a_3$. Now, $x$ has exactly three neighbors in the hole $p_1Pp_kzT_2a_3T_3a_4T_4a_1T_1sp_1$, a contradiction.  Therefore, $x$ has a neighbor $p_j$ in $V(P)$ and $p_j$ is the unique neighbor of $x$ in $V(P)$.

Suppose that $i\neq j$. By symmetry, we may assume that $j<i$. 
By the minimality of $k$, we may assume that every edge between $\{p_1,\ldots,p_{k-1}\}$ and $V(C)$ has an end in $V(T_1)\cup V(T_2)$. 
Therefore, $p_k$ has a neighbor in $(V(T_3)\cup V(T_4))\setminus (V(T_1)\cup V(T_2))$. Since $p_k$ is proper for $W$, either $N(p_k)\cap V(C)\subseteq V(T_3)$ or $N(p_k)\cap V(C)\subseteq V(T_4)$. By the minimality of $k$ and symmetry, we may assume that  every edge between $\{p_{j+1},\ldots,p_{k}\}$ and $V(C)$ has an end  in $V(T_3)$. This implies that  every edge between $\{p_{j+1},\ldots,p_{k-1}\}$ and $V(C)$ has an end at $a_3$. If every edge between $\{p_1,\ldots,p_j\}$ and $V(C)$ has an end in $V(T_2)$, 
then we can choose $S_1=T_2$ and $S_2=T_3$, and we are done. Therefore, there exists $r\in [j]$ such that $p_r$ has a neighbor in $V(T_1)\setminus V(T_2)$. Choose $r$ as large as possible. Let $t$ be the neighbor of $p_r$ in $V(T_1)$ closest to $a_1$, and let $z$ be the neighbor of $p_k$ in $V(T_3)$ closest to $a_4$. However, it follows that $x$ has exactly  three neighbors in the hole $p_rPp_kzT_3a_4T_4a_1T_1tp_r$, a contradiction. Therefore, $i=j$. This implies that $V(P)\setminus \{p_1,p_i,p_k\}$ is anticomplete to $V(C)$.

Since $p_i$ is proper for $W$, we may assume that $N(p_i)\cap V(C)\subseteq V(T_1)$. Suppose that $p_i$ has a neighbor in $T_1^*$. By the minimality of $k$ and symmetry, we may assume that $N(p_1)\cap V(C)\subseteq V(T_4)$ and $N(p_k)\cap V(C)\subseteq V(T_2)$. In particular, $p_1$ has a neighbor in $V(T_4)\setminus V(T_1)$ and $p_k$ has a neighbor in $V(T_2)\setminus V(T_1)$. Let $s$ be the neighbor of $p_1$ in $V(T_4)$ closest to $a_4$, and let $z$ be the neighbor of $p_k$ in $V(T_2)$ closest to $a_3$.  Then $x$ has exactly neighbors in the hole $p_1Pp_kzT_2a_3T_3a_4T_4sp_1$, a contradiction. Therefore, $p_i$ has no neighbor in $T_1^*$. Without loss of generality, we may assume that $p_ia_1\in E(G)$. Since $G$ is diamond-free, $p_ia_2\notin E(G)$. Now, at least one of $p_1$ and $p_k$ has a neighbor in 
$(V(T_2)\cup V(T_3))\setminus (V(T_1)\cup V(T_4))$. 
Without loss of generality, assume it is $p_1$. Furthermore, we may assume that $N(p_1)\cap V(C)\subseteq V(T_3)$ by symmetry. Let $s$ be the neighbor of $p_1$ in $V(T_3)$ closest to $a_4$. If $s\neq a_3$, then $p_i$ can be linked to the hole $xT_4x$ by $p_ix$, $p_ia_1$ and $p_iPp_1sT_3a_4$, a contradiction. Therefore, $s=a_3$. Now, assume that $N(p_k)\cap V(C)\subseteq V(T_i)$. Then we can choose $S_1=T_i$ and $S_2=T_{i-1}$ or $T_{i+1}$, 
where the subscripts are taken modulo $4$. This completes the proof of Lemma~\ref{lem:proper-wheel-path-4-spokes}.
\end{proof}

\begin{lemma}\label{lem:proper-wheel-path-x-no-neighbor}
Let $G$ be an $\iskdbpk$-free graph, and let $W=(C,x)$ be a proper wheel in $G$ with minimum number of spokes among all proper wheels centered at $x$. Let $P=p_1\cdots p_k$ be a path with $V(P)\subseteq V(G)\setminus V(W)$ such that $P$ contains no neighbor of $x$. Suppose that $W$ has at least five spokes and every path $P'$ with $V(P')\subseteq V(G)\setminus V(W)$ and $|V(P')| < k$ satisfies:
\begin{itemize}
    \item[$(\romannumeral1)$] If $P'$ contains no neighbor of $x$, then there is a sector $S$ of $W$ such that every edge from $P'$ to $C$ has an end in $V(S)$.
    \item[$(\romannumeral2)$] If $P'$ contains exactly one neighbor of $x$, then there are two sectors $S_1,S_2$ of $W$ with $V(S_1)\cap V(S_2)\neq\emptyset$ such that every edge from $P'$ to $C$ has an end in $V(S_1)\cup V(S_2)$, possibly with $S_1=S_2$.
\end{itemize}
Then there is a sector $S$ of $W$ such that every edge from $P$ to $C$ has an end in $V(S)$.
\end{lemma}
\begin{proof}
Let $P$ be a path violating the conclusion of the lemma. Since $W$ is proper, it follows that $k>1$. By assumption,  there exist two sectors $S_1,S_2$ of $W$ such that every edge from $\{p_1, \dots, p_{k-1}\}$ to $V(C)$  has an end in $V(S_1)$, and every edge from $\{p_2, \dots, p_k\}$ to $V(C)$  has an end in $V(S_2)$. It follows that $S_1 \neq S_2$.  In particular,  $p_1$ has a neighbor in $V(S_1) \setminus V(S_2)$, and $p_k$ has a neighbor in $V(S_2) \setminus V(S_1)$, and every edge from $P^*$ to $V(C)$ has an end in $V(S_1) \cap V(S_2)$. For $i=1,2$ let $a_i,b_i$  be the ends of $S_i$. We may assume  that $a_1,b_1,a_2,b_2$ appear on $C$ in this order and that $a_1 \neq b_2$.  Let $Q_1$ be the path of $C$ from $b_2$ to $a_1$ not containing $b_1$, and let $Q_2$ be the path of $C$ from $b_1$ to $a_2$ not containing $a_1$. We choose $S_1$ and $S_2$ such that $|V(Q_2)|$ is minimum (without changing $P$). Let $s$ be the neighbor of $p_1$ in $S_1$ closest to $a_1$, $t$ be the neighbor of $p_1$ in $S_1$ closest to $b_1$, $y$ be the neighbor of $p_k$ in $S_2$ closest to $a_2$, and $z$ be the neighbor of $p_k$ in $S_2$ closest to $b_2$. Then $s\ne b_1$ and $z\ne a_2$. It follows that $V(Q_2)\cap\{s,z\}=\emptyset$. Moreover, if $V(S_1) \cap V(S_2) \neq \emptyset$,
then $b_1=a_2$ and $V(Q_2)=\{b_1\}$, and in all cases 
$V(Q_1)$ is anticomplete to $P^*$. Now $D_1=sp_1Pp_kzS_2b_2Q_1a_1S_1s$ is a hole.
\begin{claim}\label{claim:W_1-wheel}
   $W_1=(D_1,x)$ is a wheel with fewer spokes than $W$.  
\end{claim}

Since $V(Q_2) \cap  V(D_1)=\emptyset$ and $V(Q_2)$ contains 
 at least a neighbor of $x$, it follows that  $x$ has fewer neighbors in $D_1$ than in $C$. It now suffices to 
show that $x$ has at least three neighbors in $Q_1$. Since $a_1,b_2 \in V(Q_1)$, we may assume that $x$ has no neighbor in $Q_1^*$, and 
$Q_1$ is a sector of $W$. By the choice of $P$, it follows that $t \neq a_1$ or $y \neq b_2$. By symmetry, we may assume that $t \neq a_1$. Since $x$ has at least five neighbors in $V(C)$,  $V(S_1) \cap V(S_2) = \emptyset$. Consequently, 
$P^*$ is anticomplete to $V(C)$. It follows from Lemma~\ref{lem:unique} that
$s=t$.
Now  we can link $s$ to the hole $a_1Q_1b_2xa_1$ by
$sS_1a_1$, $sp_1Pp_kzS_2b_2$ and $sS_1b_1x$, 
a contradiction. 
Hence, $W_1=(D_1,x)$ is a wheel with fewer spokes than $W$. This proves Claim~\ref{claim:W_1-wheel}. 

It follows from Lemma~\ref{lem:nolink} that $x$ has at least two neighbors in $Q_1^*$. By the choice of $W$, it follows  that  $W_1$ is not proper. Let $S_0$ be the sector $a_1S_1sp_1Pp_kzS_2b_2$ of $W_1$.

\begin{claim}\label{claim:type-v}
There exists $v \in  V(G) \setminus (V(W_1)\cup V(W))$ such that either 
\begin{itemize}
\item[$(\romannumeral1)$] $vx\notin E(G)$, $|N(v)\cap V(S_0)|\geq 3$ and $N(v) \cap V(D_1) \subseteq V(S_0)$, or
\item[$(\romannumeral2)$] there is a sector $S_3$ of $W$ with $V(S_3) \subseteq V(Q_1)$ such that $v$ has a neighbor in  $V(S_3) \setminus V(S_0)$ and a neighbor in $V(S_0) \setminus V(S_3)$, and $N(v) \cap V(C) \subseteq V(S_3)$.
\end{itemize}
\end{claim}

Since $W_1$ is not proper in $G$, there exists a vertex $v\in V(G)\setminus V(W_1)$ that is not proper for $W_1$. First, we show that $v\notin V(C)$. 
The only vertices of $C$ that may have more than two neighbors in $D_1$ are $b_1$ and $a_2$, and this can happen only if $b_1=a_2$. However, since $N(b_1)\cap V(D_1)\subseteq V(S_0)$ and $b_1x\in E(G)$, $b_1$ is proper for $W_1$. Therefore, $v\ne b_1$. Moreover, every other vertex of $C$ has at most two neighbors in $D_1$, which, if  they exist, are either $\{s,p_1\}$ or $\{z,p_k\}$
and are proper for $W_1$.
Hence, $v\notin V(C)$. Since $v$ is proper for $W$ but not proper for $W_1$, $v$ must have a neighbor in $V(P)$. Suppose that $(\romannumeral1)$ of Claim~\ref{claim:type-v} does not hold. Then there exists a sector $S_3$ of $W$ with $V(S_3)\subseteq V(Q_1)$ such that $v$ has a neighbor in $V(S_3)\setminus V(S_0)$. Since $v$ is proper for $W$,  $N(v)\cap V(C)\subseteq V(S_3)$. This implies that $(\romannumeral2)$ of Claim~\ref{claim:type-v} holds. This proves Claim~\ref{claim:type-v}.

Let $v$ be a vertex satisfying Claim~\ref{claim:type-v}; we will use this vertex in the following arguments to analyze the structure of $G$ relative to $W$ and $W_1$.
\begin{claim}\label{claim:v-neighbor-P-unique}
$v$ has a unique neighbor in $P$.
\end{claim}

Otherwise, suppose that $v$ has at least two neighbors in $P$. Let $v_1$ be the neighbor of $v$ in $P$ closest to $p_1$, and let $v_2$ be the neighbor of $v$ in $P$ closest to $p_k$.

First, assume that  $(\romannumeral1)$ of Claim~\ref{claim:type-v} holds for $v$.  Note that $vx\notin E(G)$. Since $W$ is a proper wheel, $v$ has at most two neighbors in $V(S_0)\setminus V(P)$. If $v$ has at least three neighbors in $P$, then $P' := p_1Pv_1vv_2Pp_k$ is a shorter path than $P$ as $G$ is diamond-free. By assumption, there is a sector $S$ of $W$ such that every edge from $P'$ to $C$ has an end in $V(S)$. Since $P'$ and $P$ have the same ends $p_1, p_k$, this leads to a contradiction. 
This implies that $v$ has exactly  two neighbors in $P$. By Lemma~\ref{lem:nolink}, we have that $v$ has  exactly two neighbors in $V(S_0)\setminus V(P)$.    Since $v$ is proper for $W$, by symmetry we may assume that $v$ has a neighbor in $V(S_0)\cap V(S_1)$. That is, $N(v)\cap V(C)\subseteq V(a_1S_1s)$.
 If $v_2\neq p_2$, then $vv_2Pp_k$ is a shorter path than $P$. By assumption,  there is a sector $S$ of $W$ such that every edge from $vv_2Pp_k$ to $C$ has an end in $V(S)$.  Note that $v$ has a neighbor in $S_1^*$. This implies that  $S_1=S_2$, a contradiction. Therefore, $v_2=p_2$ and $v_1=p_1$. Let $v_1'$ be the neighbor of $v$ in $S_1$ closest to $a_1$,  and  let $v_2'$ be the neighbor of $v$ in $S_1$ closest to $s$. Since $G$ is $\db$-free, we have $v_2'\ne s$, $v_1'v_2'\notin E(G)$, and $st\notin E(G)$. If $s=t$,  then  $v$ can be linked  to the hole $xS_1x$ by $vv_1'$, $vv_2'$ and $vp_1s$, a contradiction. Hence,  $s\neq t$. However, it follows that $v$ has exactly three neighbors  in the hole obtained from $xS_1x$ by rerouting $S_1$ through $p_1$, a contradiction.

Now, we assume that $(\romannumeral2)$ of Claim~\ref{claim:type-v} holds for $v$.
Note that no vertex of $V(Q_1)$ is contained both in a sector with end $a_1$ and in a
sector with end $b_2$.
Therefore, we may assume that $v$ has a neighbor in a sector $S_3$ where $S_3$ is the first sector from $b_2$ to $a_1$ on $Q_1$ and
does not have end $b_2$. 

Suppose that $vx\notin E(G)$. If $v_2\neq p_2$, then $vv_2Pp_k$ is a shorter path than $P$. By assumption,  there is a sector $S$ of $W$ such that every edge from $vv_2Pp_k$ to $C$ has an end in $V(S)$. This implies that either $b_2\in S$ or $S=S_2$, a contradiction. Hence,  $v_2=p_2$ and $v_1=p_1$. Moreover, since $vx\notin E(G)$, $v$ has exactly two neighbors in $S_3$, say $v_1'$ and $v_2'$; otherwise, $v$ would have only one neighbor in $S_3$ and hence exactly three neighbors in $D_1$, a contradiction.
 Since  $G$ is $\db$-free, $st\notin E(G)$ and $s\notin \{v_1',v_2'\}$. If $s=t$, then   $v$ can be linked  to the hole $C$ by $vv_1'$, $vv_2'$ and $vp_1s$, a contradiction. Hence $s\neq t$. However, it follows that $v$ can be linked to the hole obtained from $C$ by rerouting $S_1$ through $p_1$ by $vp_1$, $vv_1'$ and $vv_2'$,
 a contradiction. Thus, we have $vx\in E(G)$. 

Suppose that $v_2\neq p_2$. Let $S_4$ be the sector of $D_1$ with ends, say $\{b_2,a_3\}$ such that $V(S_4)\subseteq V(Q_1)$. Now, the path $vv_2Pp_k$ is a shorter path than $P$. By assumption, there are two sectors $S',T'$ of $W$ with $V(S')\cap V(T')\neq\emptyset$ such that every edge from $vv_2Pp_k$ to $C$ has an end in $V(S')\cup V(T')$.
Note that $S_3$ cannot be chosen so that $b_2\in V(S_3)$ and $y$ has a neighbor in $V(S_2)$. Then $V(S')\cup V(T')=V(S_4)\cup V(S_3)$ and $a_3\in V(S_3)$. 
 That is, $y=b_2$. 
 Let $v_1'$ be the neighbor of $v$ in $V(S_3)$ closest to $a_3$, and  let $b_3$ be the other end of $S_3$. If $v_1'\neq b_3$, then $x$ has exactly three neighbors in the hole $vv_2Pp_kb_2S_4a_3S_3v_1'v$, a contradiction. Therefore, $v_1'=b_3$. By Lemma~\ref{lem:triangle-wheel-4}, 
there exists a proper wheel with exactly four spokes, a contradiction.
Therefore, $v_2=p_2$ and $v_1=p_1$. Since $G$ is $\db$-free, $N(v)\cap N(x)=\emptyset$. Now, $p_1$ can be linked to the hole $p_2Pp_kzS_2b_2xvp_2$ by $p_1v_2$, $p_1v$ and $p_1sS_1a_1x$, a contradiction. This proves Claim~\ref{claim:v-neighbor-P-unique}.

By Claim~\ref{claim:v-neighbor-P-unique}, let $N(v)\cap V(P)=\{p_j\}$. Without loss of generality, we may assume that $j\neq k$.  Note that $(\romannumeral1)$ of Claim~\ref{claim:type-v} cannot happen; otherwise, $v$ has exactly three neighbors in $D_1$,  a contradiction. So $(\romannumeral2)$ of Claim~\ref{claim:type-v} holds. 
Since $v$ is proper for $W$ but not proper for $W_1$, it follows that $v$ has a neighbor in $Q_1^*$.

\begin{claim}\label{claim:P*-edge}
There are  edges between $P^*$ and $V(C)$.
\end{claim}

Otherwise, we have $s=t$ and $y=z$ by Lemma~\ref{lem:unique}. We assert that in this case $b_1 \neq a_2$. If $b_1 = a_2$, then 
$b_1$ can be linked to the hole $xa_1S_1sp_1Pp_kzS_2b_2x$
by $b_1x$, $b_1S_1s$ and $b_1S_2z$, a contradiction.  Therefore, $b_1\neq a_2$. If $v$ has exactly one neighbor $v'$ in $C$, then
$p_j$ can be linked to $C$ by $p_jPp_1s$, $p_jPp_kz$ and 
$p_jvv'$, a contradiction. So $v$ has at least two  neighbors in $C$. Recall that
$N(v) \cap V(C) \subseteq V(S_3)$. If $v$ has two non-adjacent neighbors in $V(S_3)$, then let $D_1'$ be the hole obtained from $C$ by rerouting $S_3$ through $v$.
It follows that $s,z \in V(D_1')$, and  $p_j$ can be linked to $D_1'$ by
$p_jPp_1s$, $p_jPp_kz$ and 
$p_jv$, a contradiction. Therefore, $v$ has two adjacent neighbors in $S_3$. Clearly, $vx\notin E(G)$.  If $p_j\notin \{p_1,p_k\}$,  then $v$ can be linked to the hole $C$ by two one-edge paths and $vp_jPp_1s$ (or $vp_jPp_kz$), a contradiction. Therefore, $p_j\in \{p_1,p_k\}$. By symmetry, we may assume that $p_j=p_1$. If $a_1v\notin E(G)$ or $z\neq a_1$, then $v$ can be linked to the hole $C$ by  two one-edge paths and $vp_1z$, a contradiction. Then $a_1v\in E(G)$ and $s=a_1$. Since $G$ is diamond-free, $v$ has two non-adjacent neighbors in $V(S_3)$. Now, $p_1$ can be linked to the hole  obtained  from $C$ by rerouting $S_3$  through $v$ by $p_1v$, $p_1a_1$ and $p_1Pp_kz$, a contradiction. This proves Claim~\ref{claim:P*-edge}.

It follows from Claim~\ref{claim:P*-edge} that $b_1=a_2$, $P$ has length at least $2$ and $b_1$ has neighbors in $P^*$. Let $p_i$ be the neighbor of $b_1$ in $P^*$ closest to $p_1$.

\begin{claim}\label{claim:v-1-k}
    $N(v)\cap \{p_1,p_k\}=\emptyset$.
\end{claim}

Without loss of generality, we may assume that $vp_1\in E(G)$. 
First, suppose that $vx\notin E(G)$. 
If $t\neq a_1$ or $i\neq k-1$, then by assumption we have 
$N(v)\cap V(C)\subseteq V(S_1)$, a contradiction.
 Therefore, $t=a_1$, $i=k-1$ and $a_1\in V(S_3)$. 
 If $y=z$, then $z$ can be linked to the hole $p_1Pp_{k-1}b_1S_1a_1p_1$ by $zp_kp_{k-1}$, $zS_2b_1$ and $zS_2b_2Q_1a_1$, a contradiction. Therefore, $y\neq z$. If $y\neq b_1$, then $p_k$ can be linked to the hole $xS_2x$ by $p_kz$, $p_ky$ and $p_kp_{k-1}b_1$, a contradiction. Therefore, $y=b_1$. Since $G$ is bowtie-free, $a_1b_1\notin E(G)$. This implies that $b_1$  can be linked to the hole $D_1$ by $b_1p_{k-1}$, $b_1p_k$ and $b_1S_2z$, a contradiction.

Therefore, $vx\in E(G)$.   Suppose that $va_1\in E(G)$. Since $v$ is proper for $W$ and $v$ has a neighbor in $Q_1^*$, $v$ has no neighbor in $S_1^*$. But $v$ can be linked to $xS_0x$ by $vx,va_1$ and $vp_1$, a contradiction. Similarly, $vb_2\notin E(G)$.
If $s\neq t$, then $p_1$ can be linked to the hole $xS_1x$ by $p_1s$, $p_1t$ and $p_1vx$.  Therefore, $s=t$. But now, $p_1$ can be linked to the hole $xS_1x$ by $p_1s$, $p_1Pp_ib_1$ and $p_1vx$. This proves Claim~\ref{claim:v-1-k}.
\begin{claim}\label{claim:vx-1}
$vx\in E(G)$.
\end{claim}

Otherwise, assume that $vx\notin E(G)$.  It is obvious that $|V(P)|\geq 3$. By Claim~\ref{claim:v-1-k}, there exists a path from a neighbor of $b_1$ in $P^*$ to  $v$ with interior in $P^*$. By assumption, we have that $N(v)\cap V(C)\subseteq V(S_1)\cup V(S_2)$, a contradiction. This proves Claim~\ref{claim:vx-1}.

\begin{claim}\label{claim:v-a_1-b_2}
    $N(v)\cap \{a_1,b_2\}=\emptyset$.
\end{claim}

Otherwise, assume that $va_1\in E(G)$.  Then $v$ has exactly three neighbors in  the hole $xS_0x$,   a contradiction. This proves Claim~\ref{claim:v-a_1-b_2}.

By Claim~\ref{claim:v-1-k},  $p_j\in P^*$. Now, by considering the shortest path $P'$ from a neighbor of $b_1$ in $P^*$ to $v$ with interior in $P^*$, then $P'$ is a shorter path than $P$.
 By assumption, there are two sectors $S',T'$ of $W$ with $V(S')\cap V(T')\neq\emptyset$ such that every edge from $P'$ to $C$ has an end in $V(S')\cup V(T')$. This implies that $S_3$ can be chosen such that one of $a_1,b_2$ belongs to $V(S_3)$. 
By symmetry, we may assume that $a_1\in V(S_3)$. Furthermore, let $v^*$ be the  neighbor of $v$ in $S_3$ closest to $a_1$. By Claim~\ref{claim:v-a_1-b_2}, we have that $v^*\neq a_1$.
If $v^*$ is the other end of $S_3$, then $x$ can be linked to the hole $vv^*S_3a_1S_1sp_1Pp_jv$ by  $xv^*$, $xv$ and $xa_1$, a contradiction. Thus, $v^*$ is not the other end of $S_3$. However, it follows that $x$ can be linked to the hole $vv^*S_3a_1S_1sp_1Pp_jv$ by $xv$,  $xa_1$ and $xb_2S_2zp_kPp_j$, a contradiction. This completes the proof of Lemma~\ref{lem:proper-wheel-path-x-no-neighbor}.
\end{proof}

\begin{lemma}\label{lem:proper-wheel-path-x-neighbor}
Let $G$ be an $\iskdbpk$-free graph, and let $W=(C,x)$ be a proper wheel in $G$ with minimum number of spokes among all proper wheels centered at $x$. Let $P=p_1\cdots p_k$ be a path with $V(P)\subseteq V(G)\setminus V(W)$ such that $P$ contains exactly  one neighbor of $x$. Suppose that $W$ has at least five spokes and every path $P'$ with $V(P')\subseteq V(G)\setminus V(W)$ and $|V(P')| < k$ satisfies:
\begin{itemize}
 \item[$(\romannumeral1)$] If $P'$ contains no neighbor of $x$, then there is a sector $S$ of $W$ such that every edge from $P'$ to $C$ has an end in $V(S)$.
    \item[$(\romannumeral2)$] If $P'$ contains exactly one neighbor of $x$, then there are two sectors $S_1,S_2$ of $W$ with $V(S_1)\cap V(S_2)\neq\emptyset$ such that every edge from $P'$ to $C$ has an end in $V(S_1)\cup V(S_2)$, possibly with $S_1=S_2$.
\end{itemize}
Then there are two sectors $S_1,S_2$ of $W$ with $V(S_1)\cap V(S_2)\neq\emptyset$ such that every edge from $P$ to $C$ has an end in $V(S_1)\cup V(S_2)$, possibly with $S_1=S_2$.
\end{lemma}
\begin{proof}
Let $P$ be a path violating the conclusion of the lemma. Since $W$ is proper, it follows that $k>1$. Let $xp_i\in E(G)$.
Then $p_i$ is the unique neighbor of $x$ in $V(P)$.
By assumption,  
there exist two distinct sectors $S_1,S_2$ of $W$ such that 
$p_1$ has a neighbor in $V(S_1) \setminus V(S_2)$, and $p_k$ has a neighbor in
$V(S_2) \setminus V(S_1)$. 
If $1<i<k$, then
every edge from $\{p_1, \dots, p_{i-1}\}$ to $V(C)$  has an end in $V(S_1)$,
and every  edge from $\{p_{i+1}, \dots, p_k\}$ to $V(C)$  has an end in 
$V(S_2)$; if $i=1$ then every edge from $V(P) \setminus \{p_1\}$ to $V(C)$
has an end in $V(S_2)$; and if $i=k$ then  every edge from 
$V(P) \setminus \{p_k\}$ to $V(C)$ has an end in $V(S_1)$. For $j=1,2$, let $a_j,b_j$  be the ends of $S_j$.

\begin{claim}\label{claim:S1,S2-select}
One of the following statements holds:
\begin{itemize}
\item[$(\romannumeral1)$] there is no edge between $V(C)$ and $P^*$, or
\item[$(\romannumeral2)$] we can choose $S_1,S_2$ such that   $a_1,b_1,a_2,b_2$ appear in $C$ in 
order and  there is a sector $S_3$ with ends $b_1,a_2$, and every edge between
$V(C)$ and $P^*$ is from $b_1$ to $\{p_2, \dots, p_{i-1}\}$
or from $p_i$ to $S_3$, or from $a_2 $ to $\{p_{i+1}, \dots, p_{k-1}\}$. 
\end{itemize}
\end{claim}

Otherwise, it follows that there exists an edge between $P^*$ and $V(C)$. 
Suppose that there is a sector $S_3$ of $W$ with an edge between $S_3^*$ and $P^*$. 
Without loss of generality, we may assume that some vertex $p_j\in P^*$ has a neighbor in $S_3^*$. We assert that $j=i$. Otherwise, by symmetry, assume that $i<j$. By assumption, we have that every  edge from $\{p_{i+1}, \dots, p_k\}$ to $V(C)$  has an end in 
$V(S_3)$, and there are two sectors $S_1',S_2'$ of $W$ with $V(S_1')\cap V(S_2')\neq\emptyset$ such that every edge from $\{p_1,\ldots,p_j\}$ to $V(C)$ has an end in $V(S_1')\cup V(S_2')$. Since $p_j$ has a neighbor in $S_3^*$, $S_3\in \{S_1',S_2'\}$. However, it follows that  every edge from $V(P)$ to $V(C)$ has an end in $V(S_1')\cup V(S_2')$, a contradiction. Therefore, $i=j$ and $P^*\setminus \{p_i\}$ is anticomplete to $V(C)\setminus N(x)$.

Let $b_1',a_2'$ be the ends of $S_3$. Let $S_1'$ and $S_2'$ be the sectors of $W$ satisfying
$V(S_1')\cap V(S_3)={b_1'}$ and $V(S_2')\cap V(S_3)={a_2'}$, respectively. By assumption,  there exist two sectors $T_1$, $T_2$  with $V(T_1)\cap V(T_2)\neq \emptyset$ such that  every edge from $\{p_1,\ldots,p_{k-1}\}$ to $V(C)$ has an end in $V(T_1)\cup V(T_2)$. Since $p_i\in P^*$ and $p_i$ has a neighbor in $S_3^*$, we may, by symmetry, assume that $T_1=S_1'$ and $T_2=S_3$.   On the one hand, since $P$ is a counterexample of Lemma~\ref{lem:proper-wheel-path-x-neighbor}, $p_k$ has a neighbor in $V(C)\setminus (V(S_1')\cup V(S_3))$. On the other hand, by assumption, we have that  every edge from $\{p_2,\ldots,p_{k}\}$ to $V(C)$ has an end in $V(S_3)\cup V(S_2')$. This implies that every edge from $\{p_2,\ldots,p_{k-1}\}$ to $V(C)$ has an end in $\{b_1',a_2'\}$.
Note that $p_1$ has a neighbor in $V(S_1')\setminus V(S_3)$ and $p_k$ has a neighbor in $V(S_2')\setminus V(S_3)$.
By assumption, every edge from $\{p_2,\ldots,p_{i-1}\}$ to $V(C)$ has an end in $\{b_1'\}$, and every edge from $\{p_{i+1},\ldots,p_{k-1}\}$ to $V(C)$ has an end in $\{a_2'\}$. Now we choose $S_1 = S_1'$ and $S_2 = S_2'$ such that $b_1 = b_1'$ and $a_2 = a_2'$.
Hence, $(\romannumeral2)$ of Claim~\ref{claim:S1,S2-select} holds. This implies that $P^*$ is anticomplete to $V(C)\setminus N(x)$.

Suppose that there exists an edge between $P^*\setminus \{p_i\}$ and $V(C)$. Without loss of generality, we may assume that some vertex $p_j\in P^*\setminus \{p_i\}$ has a neighbor $b_1$ in $V(C)\setminus N(x)$. By symmetry we may assume that $j<i$.
Let $S_1$ and $S_3$ be two sectors such that $V(S_1)\cap V(S_3)=\{b_1\}$. By assumption and symmetry, we may assume that every edge from $\{p_1,\ldots,p_{i-1}\}$ to $V(C)$ has an end in $V(S_1)$. Let $a_2$ be  the other end of $S_3$ distinct from $b_1$, and let $S_2$ be the sector of $W$ such that $V(S_2)\cap V(S_3)=\{a_2\}$. By assumption, there exist two sectors $T_1$, $T_2$  with $V(T_1)\cap V(T_2)\neq \emptyset$ such that  every edge from $\{p_2,\ldots,p_k\}$ to $V(C)$ has an end in $V(T_1)\cup V(T_2)$. Since $P$ is a counterexample of Lemma~\ref{lem:proper-wheel-path-x-neighbor}, $S_1\notin \{T_1,T_2\}$. Therefore, $\{T_1,T_2\}=\{S_2,S_3\}$. In particular, $p_1$ has a neighbor in $V(S_1)\setminus V(S_3)$, $p_k$ has a neighbor in $V(S_2)\setminus V(S_3)$ and $N(p_i)\cap V(C)\subseteq V(S_3)$.  This implies that every edge from $\{p_2,\ldots,p_{i-1}\}$ to $V(C)$ has an end in $\{b_1\}$, and every edge from $\{p_{i+1},\ldots,p_{k-1}\}$ to $V(C)$ has an end in $\{a_2\}$. Therefore,  $(\romannumeral2)$ of Claim~\ref{claim:S1,S2-select} holds. Therefore, $P^*\setminus \{p_i\}$ is anticomplete to $V(C)$. This implies that $p_i\in P^*$ and $p_i$ has a neighbor in $N(x)\cap V(C)$. Since $G$ is diamond-free, $p_i$ has exactly one neighbor in $N(x)\cap V(C)$, say $b_1$.

Let  $S_3$ and $S_4$ be two sector of $W$ such that $V(S_3)\cap V(S_4)=\{b_1\}$.
Let $a_1$ be the end of $S_3$ distinct from $b_1$, and let $a_2$ be the end of $S_4$ distinct from $b_1$. Let $S_3'$ be the sector of $W$ such that $V(S_3')\cap V(S_3)=\{a_1\}$, and let $S_4'$ be sector of $W$ such that $V(S_4')\cap V(S_4)=\{a_2\}$. We assert that $N(p_1)\subseteq V(S_3)\cup V(S_4)$. Otherwise, by assumption, we have that $N(p_1)\subseteq V(S_3')\cup V(S_4')$.     Since $p_1$ is proper for $W$, by symmetry, we may assume that $N(p_1)\subseteq V(S_3')$. Let $t$ be the neighbor of $p_1$ in $S_3'$ closest to $a_1$. By Lemma~\ref{lem:triangle-wheel-4} and $W$ has at least five spokes, we have that $t$ is not the end of $S_3'$ distinct from $a_1$. However, it follows that $p_i$ can be linked to the hole $xS_3x$ by $p_ib_1$, $p_ix$, and $p_iPp_1tS_3'a_1$, a contradiction. Therefore, $N(p_1)\subseteq V(S_3)\cup V(S_4)$. By symmetry, $N(p_k)\subseteq V(S_3')\cup V(S_4')$. Hence, $P$ is not a counterexample. This proves Claim~\ref{claim:S1,S2-select}.

If $ (\romannumeral2)$ of Claim~\ref{claim:S1,S2-select} holds, 
let $Q_1$ be the path of $C$ from $b_2$ to $a_1$ not containing $b_1$, 
and let $Q_2=S_3$. 
Otherwise, assume $ (\romannumeral1)$ of Claim~\ref{claim:S1,S2-select} holds.
We may assume that $a_1,b_1,a_2,b_2$  appear on $C$ in this order. Also,
$a_1,b_1,a_2,b_2$ are all distinct, since $P$ violates the 
conclusion of the lemma.
Let $Q_1$ be the path of $C$ from $b_2$ to $a_1$ not containing $b_1$, and let 
$Q_2$ be the path of $C$ from $b_1$ to $a_2$ not containing $a_1$. We may assume 
that $S_1$, $S_2$ are chosen with $|V(Q_2)|$ minimum (without changing $P$). 
We may assume that $P$ was chosen so that $V(Q_1)$ is (inclusion-wise) minimal.

Since $W$ is proper, it follows that $N(p_1) \cap V(C) \subseteq V(S_1)$ 
and $N(p_k) \cap V(C) \subseteq V(S_2)$. 
Let $s$ and $t$ be the neighbors of $p_1$ in $S_1$ closest to $a_1$ and $b_1$ respectively.
Let $y$ and $z$ be the neighbors of $p_k$ in $S_2$ closest to $a_2$ and $b_2$ respectively. Then $s \neq b_1$ and $z \neq a_2$.

Let $D_1$ be the hole $a_1S_1sp_1Pp_kzS_2b_2Q_1a_1$. Then $W_1=(D_1,x)$
is a wheel with fewer spokes than $W$. By Lemma~\ref{lem:nolink}, we have that  $x$ has a neighbor in $Q_1^*$. 
Let $S_0$ be the sector $a_1S_1sp_1Pp_i$, and let $T_0$ be the
sector $p_iPp_kzb_2$ of $W_1$.

\begin{claim}\label{S_0T_0}
No vertex $v \in V(G) \setminus V(W_1)$ has both a neighbor
in $V(S_0) \setminus V(T_0)$ and a neighbor in $V(T_0) \setminus V(S_0)$.
\end{claim}

Otherwise, suppose that there exists $v \in V(G) \setminus V(W_1)$ such that $v$ has a neighbor in $V(S_0) \setminus V(T_0)$ and a neighbor in $V(T_0) \setminus V(S_0)$. Suppose $xv\notin E(G)$. If $v$ has a neighbor in $V(a_1S_1s)$, then $v$ has no  neighbor in $V(zS_2b_2)$ since $v$ is proper for $W$. Consequently $v$ has a neighbor in $V(T_0) \setminus (V(S_2) \cup V(S_0))$.
Let $j$ be maximum such that $v$ is adjacent to $p_j$. Then $j>i$.
Since $|V(vp_jPp_k)|\le k$ and $x$ has no neighbor in $V(vp_jPp_k)$, by Lemma~\ref{lem:proper-wheel-path-x-no-neighbor},
we deduce that either $S_1=S_2$ or $V(S_1)\cap V(S_2)\neq\emptyset$,  a contradiction.
Thus we may assume that 
$N(v) \cap (V(S_0) \cup V(T_0)) \subseteq V(P)$. Let $j$ be minimum
index and $l$ maximum index such that $v$ is adjacent to $p_j$ and $p_l$. Then 
$j<i$ and $l>i$.  Since $x$ has no neighbor in $V(p_1Pp_jvp_lPp_k)$ and $|V(p_1Pp_jvp_lPp_k)|\leq k$, by Lemma~\ref{lem:proper-wheel-path-x-no-neighbor}, we deduce that either $S_1=S_2$ or $V(S_1)\cap V(S_2)\neq\emptyset$, a contradiction. Therefore, $xv\in E(G)$.

Suppose that $N(v)\cap V(D_1)\subseteq V(S_0)\cup V(T_0)$. By Lemma~\ref{lem:neighbor-consecuitive-sector}, we have that 
$v$ has at least three neighbors in  $V(T_0)$ and at least three  neighbors in $V(S_0)$. Suppose that $v$ has a neighbor  in $V(a_1S_1s)$. Then $N(v)\cap V(T_0)  \subseteq V(T_0) \setminus V(S_2)$ and $|N(v)\cap V(p_iPp_k)|\geq 3$. Let $j$ be maximum index such that $vp_j\in E(G)$. 
This implies that $vp_jPp_k$ is a path of order less than $k$, which is a contradiction to the assumption. Therefore, $N(v)\cap V(D_1)\subseteq V(P)$.  Note that $i\notin \{1,k\}$. 
But now, rerouting $P$ through $v$, we obtain a contradiction to the assumption. Therefore, there  is a sector $S_4$ of $W$ with $S_4\subseteq V(Q_1)$ such that $v$ has a neighbor in  $V(S_4) \setminus (V(S_0) \cup V(T_0))$.

Since $v$ is proper for $W$, $|N(v)\cap \{a_1,b_2\}|\leq 1$. Without loss of generality, we may assume that $vb_2\notin E(G)$. Suppose that $va_1\in E(G)$. Since $G$ is diamond-free, $vp_i\notin E(G)$.  Let $j$ be maximum such that $vp_j\in E(G)$. Then $j>i$. If $j> 2$, then $y=b_2$ and $Q_1^*$ has exactly one spoke  $a_3$   since the path $vp_jPp_k$ has order less than $k$. Note that $v$ has a neighbor in $V(S_4)\setminus \{a_1\}$. Let $v'$ be the neighbor of $v$ in $S_4$ closest to $a_3$, and let $S_5$ be a sector of $W$ with ends $a_3,b_2$. Then $x$ can be linked to the hole $vv'S_4a_3S_5b_2p_kPv_jv$ by $xv$, $xa_3$ and $xb_2$, a contradiction. Therefore, $j\leq 2$ and hence $j=2$ and $i=1$. Since $G$ is $\db$-free, $p_1v\notin E(G)$.  But now, $x$ can be linked to the hole $vp_2p_1sS_1a_1v$ by $xv,xa_1$ and $xp_1$,  a contradiction.
Therefore, $va_1\notin E(G)$. Similarly, $vb_2\notin E(G)$. Then $N(v)\cap (V(S_0)\cup V(T_0))\subseteq V(P)$. 

Let $j$ be minimum index
and $l$ be maximum index such that $v$ is adjacent to $p_j$ and $p_l$. Then 
$j<i$ and $l>i$.  Now, we consider the path $p_1Pp_jvp_lPp_k$. By assumption, we have that $j=i-1$ and $l=i+1$. Now, $x$ can be linked to the hole $vp_{i-1}p_ip_{i+1}v$ by $xv,xp_i$ and $xa_1S_1sp_1Pp_{i-1}$, a contradiction. This proves Claim~\ref{S_0T_0}.

\begin{claim}\label{claim:skip}
Every skip for $W_1$ is 
either $a_1$-skip or $b_2$-skip.
\end{claim}
Let $v$ be a skip for $W_1$. Since $W$ is proper, it follows 
that $N(v) \cap V(C)$ is included in a unique sector of $W$.  Consequently,
$v$ is $a_1$-skip, $b_2$-skip or $p_i$-skip. However, Claim~\ref{S_0T_0} implies that $v$ is not
$p_i$-skip. This proves Claim~\ref{claim:skip}.

Let $X$ be the set of all skips for $W_1$. It follows from
Lemma~\ref{lem:2-almost} that $W_1$ is not proper in $V(G) \setminus X$.

\begin{claim}\label{claim:cases}
There exists $v \in V(G) \setminus (V(W_1) \cup X)$ such that 
one of the following holds:
\end{claim}
\begin{itemize}
   \item[$(\romannumeral1)$] $v$ is non-adjacent to $x$, and $v$ has at least three neighbors in $S_0$, and $N(v) \cap V(D_1) \subseteq V(S_0)$;
\item[$(\romannumeral2)$] $v$ is non-adjacent to $x$, and $v$ has at least three neighbors in $T_0$, and $N(v) \cap V(D_1) \subseteq V(T_0)$;
\item[$(\romannumeral3)$] 
There is a sector  $S_4$ of $W$ with $V(S_4) \subseteq V(Q_1)$ 
 such that $v$ has a neighbor in  $V(S_4) \setminus (V(S_0) \cup V(T_0))$, 
$v$ has a neighbor in  $V(S_0) \setminus V(S_4)$,
$v$ does not have a neighbor in $V(T_0) \setminus (V(S_0) \cup V(S_4))$, 
and $N(v) \cap V(C) \subseteq V(S_4)$. (possibly with the roles of $S_0$ and $T_0$ exchanged) 
\end{itemize}
We may assume that $(\romannumeral1)-(\romannumeral2)$ of Claim~\ref{claim:cases} do not hold.  
Since $W$ is proper and $W_1$ is not proper, there exists $v \in V(G) \setminus V(W_1)$ and a sector  $S_4$ of 
$W$ with  $V(S_4) \subseteq V(Q_1)$, such 
that $v$ has a neighbor in  $V(S_4) \setminus (V(S_0)\cup V(T_0))$ and $v$ has a neighbor in 
$V(P)$. 
But now Claim~\ref{S_0T_0} implies that $(\romannumeral3)$ of Claim~\ref{claim:cases} holds. This proves Claim~\ref{claim:cases}. 

Let $v \in V(G)$ be as in Claim~\ref{claim:cases}. Next we show that:
\begin{claim}\label{claim:uniquenbr}
$v$ has a unique neighbor in $V(P)$.
\end{claim}

Otherwise, we may assume that $v$ has at least two neighbors in $P$. Clearly, if $vx \notin E(G)$, then $|N(v)\cap V(P)| \le 2$. Indeed, if $|N(v)\cap V(P)| \ge 3$, then rerouting $P$ through $v$ would contradict the assumption.

First suppose that $(\romannumeral1)$ of Claim~\ref{claim:cases} holds. 
Since $vx\notin E(G)$, $v$ has exactly two neighbors in $P$. This implies that $v$  has a neighbor in $V(sS_1a_1)$. 
Let $j$ be maximum index such that $vp_j\in E(G)$. 
Since $N(v)\cap V(D_1)\subseteq V(S_0)$, we have $j\leq i$.  If $j \ge 3$, then the path $vp_jPp_k$ contradicts the assumption. 
Therefore, $j=2$. 
This implies that $N(v)\cap V(P)=\{p_1,p_2\}$. 
Since $v$ is proper for $W$, by Lemma~\ref{lem:nolink}, $v$ has exactly two neighbors in $V(a_1S_1s)$, denoted by $v_1$ and $v_2$; Otherwise $v$ can be linked to $xS_0x$.
Since $G$ is $\db$-free,  $\{v_1,v_2\}\cap\{s,t\}=\emptyset$  and $v_1v_2, st\notin E(G)$. If $s=t$, then $v$ can be linked to the hole $xS_1x$ by $vv_1$, $vv_2$ and $vp_1s$, a contradiction. 
Therefore, $s\neq t$. But now, $v$ can be linked to the hole obtained from $xS_1x$ by rerouting $S_1$ through $p_1$, by $vp_1$, $vv_1$ and $vv_2$, again a contradiction. By symmetry, we may assume that $(\romannumeral3)$ of Claim~\ref{claim:cases} holds. 

We may assume that  $N(v)\cap V(P)\subseteq V(S_0)$. Let $j$ be minimum index such that $vp_j\in E(G)$, and let $l$ be maximum index such that $vp_l\in E(G)$. Since $N(v)\cap V(D_1)\subseteq V(S_0)$, we have $j<l\leq i$. 

Suppose that $xv\notin E(G)$. Then $v$ has exactly two neighbors in $V(P)$.
Since $v$ has a neighbor in 
$V(S_4) \setminus V(S_0)$,  by Lemma~\ref{lem:proper-wheel-path-x-no-neighbor} and assumption,
$t=a_1$ and $a_1 \in V(S_4)$. Let $v'$ be the  neighbor of $v$ in $V(S_4)$ closest to $a_1$. If $v'$ is not the other end of $S_4$ distinct from $a_1$, then $p_l$ can be linked to the hole $p_1Pp_jvv'S_4a_1p_1$ by $p_lPp_j$, $p_lv$ and $p_lPp_ixa_1$, a contradiction. Therefore, $v'$ is the end of $S_4$ distinct from $a_1$. If $p_jp_l\notin E(G)$, then $x$ can be linked to the hole $p_1Pp_jvv'S_4a_1p_1$ by $xa_1$, $xv'$ and $xp_iPp_lv$, a contradiction. Therefore, $p_jp_l\in E(G)$. Since $G$ is prism-free, $a_1v'\notin E(G)$. Now, $v$ can be linked to the hole $p_1Pp_ixa_1p_1$ by $vp_j$, $vp_l$ and $vv'x$, a contradiction. Therefore, $vx\in E(G)$.

Now suppose that $a_1 \notin V(S_4)$.
Let $S_5$ be the sector of $W$ 
with end $a_1$ such that $V(S_5) \subseteq V(Q_1)$, and let $b_3$ be the 
other end of $S_5$. Since $G$ is diamond-free, $j\leq i-2\leq k-2$. By assumption
$V(S_4) \cap V(S_5) = \{b_3\}$ and $t=a_1$. Denote the other end of $S_4$, distinct from $b_3$, as $a_3$. Let $v_1$ be a neighbor  of $v$ in $V(S_4)$ closest to $b_3$. By Lemma~\ref{lem:triangle-wheel-4} and $W$ has at least five spokes, we have that $v_1\neq a_3$. Now, $x$  has exactly three neighbors in the hole $vp_jPp_1a_1S_5b_3S_4v_1v$, a contradiction.  This proves that $a_1\in V(S_4)$. Let $b_3$ be the other end of $S_4$ distinct from $a_1$. If $v$ has no neighbors in $S_4^*$, then $vb_3\in E(G)$. 
Hence $v$ can be linked to the hole $xS_4x$ by $vx$, $vb_3$, 
and either $va_1$ or $vp_jPp_1sa_1$, a contradiction. Therefore, $v$ has a neighbor in $S_4^*$.  let $S_5$ be 
the  sector of $W$ such that $V(S_5)\cap V(S_4)=\{b_3\}$, and let $a_3$ be the other 
end of $S_5$. Since $v \not \in X$, it follows that $v$ has a neighbor 
$u \in V(G) \setminus V(W_1)$ such that $u$ has a neighbor in 
$V(D_1) \setminus (V(S_4) \cup V(S_0))$.  Suppose that $ux\in E(G)$.  By Lemma~\ref{lem:edge-N(x)} and $W$ has at least five spokes, we have that $N(u)\cap V(D_1)\subseteq V(p_iPp_k)$.  Let $r$ be minimum such that $p_ru\in E(G)$. Then $x$ can be linked to the hole $p_lPp_ruvp_l$ by $xu,xv$ and $xp_i$, a contradiction. Therefore, $ux\notin E(G)$.

Suppose that $u$ has a neighbor in $V(Q_1) \setminus V(S_4)$. Since $G$ is diamond-free  and $v$ has at least  two neighbors in $V(P)$, it follows that 
$i \geq 3$, and therefore $k \geq 3$.  Consequently,  the path $uv$ is 
shorter than $P$,
and by assumption
$N(u) \cap V(C) \subseteq V(S_5)$. Let $v'$ be the neighbor of $v$ in $S_4$ closest to $b_3$. Since $ux\notin E(G)$ and $u$ is proper for $W$, $u$ has at most two neighbors in $S_5$. If $u$ has exactly one neighbor $u'$ in $S_5$, then $v$ can be linked to the hole $xS_5x$ by $vx$, $vv'S_4b_3$ and $vuu'$, a contradiction. 
Therefore, $u$ has exactly two neighbors in $S_5$. If $u$ has two non-adjacent neighbors in $S_5$, then $v$ can be linked to the hole obtained from $xS_5x$  by rerouting $S_5$ through $u$ by $vx$, $vu$ and $vv'S_4b_3$, a contradiction. Therefore,  the two neighbors of $u$ in $S_5$ are adjacent. Since $G$ is $\db$-free,  $N(u)\cap N(v)=\emptyset $. If $v'= b_3$, then $G[\{u,v,x\}\cup V(S_5)]$ is a prism, a contradiction. Therefore, $v'\neq b_3$. But now, $u$ can be linked to the hole $xS_5x$ by two one-edge paths and $vux$, a contradiction. Therefore, $u$ has no neighbor in $V(Q_1) \setminus V(S_4)$. This implies that $u$ has a neighbor in $V(p_iPp_k)$. Since $ux\notin E(G)$, $u$ has at most two neighbors in $V(p_iPp_k)$. Let $u_1$ be the neighbor of $u$ in $V(p_iPp_k)$ closest to $p_i$, and let $u_2$ be the neighbor of $u$ in $V(p_iPp_k)$ closest to $p_k$. If $u_1\neq u_2$, then $u$ can be linked to the hole $xT_0x$ by $uu_1$, $uu_2$ and $uvx$, a contradiction. Therefore, $u_1=u_2$. However, it follows that $x$ can be linked to the hole $vuu_1Pp_lv$ by $xv$, $xp_i$ and $xb_2S_2zp_kPu_1$, a contradiction. This proves Claim~\ref{claim:uniquenbr}.

In view of Claim~\ref{claim:uniquenbr}, let $p_j$ be the  unique neighbor of $v$ in $V(P)$.

\begin{claim}\label{claim:case2}
  $(\romannumeral3)$ of Claim$~\ref{claim:cases}$ holds.
\end{claim}

Suppose that $(\romannumeral1)$ of Claim~\ref{claim:cases} holds. 
Then, by Lemma~\ref{lem:nolink}, $v$ has at least four neighbors in the hole $xS_0x$. 
Hence, by Claim~\ref{claim:uniquenbr}, $x$ has at least three neighbors in the path $a_1S_1s$, 
contradicting the fact that $W$ is proper. 
Similarly, $(\romannumeral2)$ of Claim~\ref{claim:cases} does not hold.
This proves Claim~\ref{claim:case2}.

In the next claim, we further determine the structure of $P$. 

\begin{claim}\label{structure}
One of the following statements holds:

\begin{itemize}
\item[$(\romannumeral1)$] there are  edges between $P^*$ and $V(C)$, or
\item[$(\romannumeral2)$] $j=1$ and we can choose $S_4$ so that $a_1 \in V(S_4)$, or
\item[$(\romannumeral3)$] $j=k$ and we can choose $S_4$ so that $b_2 \in V(S_4)$.
\end{itemize}
\end{claim}

Suppose that Claim~\ref{structure} is false. 
By Lemma~\ref{lem:K13}, we have that $j\in \{1,k\}$. By symmetry, we may assume that $j=1$.   Then $S_4$ cannot be chosen so that 
$a_1 \not \in V(S_4)$, since otherwise Claim~\ref{structure} holds. By Lemma~\ref{lem:edge-N(x)} and $W$ has at least five spokes, we have that $|N(x)\cap \{p_1,v\}|\leq 1$. 
Let $S_5$ be 
the sector of $W$ with $V(S_5)\cap V(S_1)=\{a_1\}$, and let $b_3$ be the other end of $S_5$ distinct from $a_1$. If $k\geq 3$, then by assumption $t=a_1$, $|N(x)\cap \{p_1,v\}|=1$ and $b_3\in V(S_4)$. 
Let $a_3$ be the other end of $S_4$ distinct from $b_3$, and let $v'$ be the neighbor of $v$ in $S_4$ closest to $b_3$. By Lemma~\ref{lem:triangle-wheel-4} and $W$ has at least five spokes, we have that $v'\neq a_3$. Now, $x$ has exactly three neighbors in the hole $p_1a_1S_5b_3S_4v'vp_1$, a contradiction. Therefore, $k=2$.   By Claim~\ref{S_0T_0}, we have $b_2v \notin E(G)$. Let $R$ be the path from $v$ to $a_1$ with $R^* \subseteq V(Q_1)\setminus\{b_2\}$, and let $Q_1'$ be the subpath of $R$ from an end of $S_4$ to $a_1$.  Then  $V(Q_1') \subsetneq V(Q_1)$. This implies that $p_1v$ is a counterexample of Lemma~\ref{lem:proper-wheel-path-x-neighbor}.
 However, this  contradicts the choice of $P$. This proves Claim~\ref{structure}.

\begin{claim}\label{claim:i=j}
    $i=j$.
\end{claim}

Otherwise, we may assume that $j<i$. First suppose that $xv\notin E(G)$. Then  Lemma~\ref{lem:proper-wheel-path-x-no-neighbor} holds for the path $p_1Pp_jv$; therefore, $a_1=t$ and $S_4$ can be chosen so that $a_1\in V(S_4)$.  Since $v$ is proper for $W$, it follows that $v$ has at most two neighbors in $V(S_4)$. Suppose $v$ has two neighbors in $V(S_4)$. If $va_1 \notin E(G)$ or $j \neq 1$, then $v$ can be linked to the hole $xS_4x$ by two one-edge paths and 
$vp_jPp_i x$ (if $j \neq 1$), or by $vp_1a_1$ (if $j=1$ and $va_1 \notin E(G)$), a contradiction.
 Therefore, $a_1v\in E(G)$ and $j=1$.  Let $v'$ be the neighbor of $v$ in $V(S_4)$ distinct from $a_1$. Since $G$ is diamond-free, $v'a_1\notin E(G)$. Now, $p_1$ can be linked to the hole obtained from $xS_4x$ by rerouting $S_4$ through $v$ by $p_1a_1$, $p_1v$ and $p_1Pp_ix$, a contradiction. Therefore, $v$ has a unique neighbor $v'$ in $S_4\setminus \{a_1\}$. Now, it follows from Claim~\ref{claim:S1,S2-select} that $p_j$ can be linked to the hole $xS_4x$ by $p_jvv'$, $p_jPp_1a_1$ and $p_jPp_ix$, a contradiction. 

Therefore, $vx\in E(G)$.
Suppose that $a_1\in V(S_4)$. Let $v'$ be the neighbor of $v$ in $V(S_4)$ closest to $a_1$. Then $x$ can be linked to the hole $p_1Pp_jvv'S_4a_1S_1sp_1$ by $xv$, $xa_1$ and $xp_iPp_j$ (if $xv'\notin E(G)$), or  by $xv'$ (if $xv'\in E(G)$), a contradiction.  Therefore, $a_1\notin V(S_4)$. Let $S_5$ be 
the sector of $W$ with $V(S_5)\cap V(S_1)=\{a_1\}$, and let $b_3$ be the other end of $S_5$ distinct from $a_1$. 

Let $v'$ be the neighbor of $v$ in $Q_1$ closest to $a_1$. 
Suppose that $V(S_4)\cap V(S_5)\neq\emptyset$; 
that is, $b_3\in V(S_4)$. 
By Lemma~\ref{lem:triangle-wheel-4} and $W$ has at least five spokes, $v'$ is not the end of $S_4$ distinct from $b_3$. However, it follows that $x$ has exactly three neighbors in the hole $p_1Pp_jvv'S_4b_3S_5a_1S_1sp_1$, a contradiction. Therefore,
$V(S_4)\cap V(S_5)=\emptyset$. This implies that there are no consecutive sectors $S_1',S_2'$ of $W$ such that $(N(p_1)\cup N(v))\cap V(C)\subseteq V(S_1')\cup V(S_2')$.
If $j\le k-2$, then the path $p_1Pp_jv$ has order at most $k-1$, contradicting the assumption.
 Thus $j=k-1$ and $i=k$. By Claim~\ref{claim:S1,S2-select}, we have that $\{a_2\}$ is anticomplete to  $V(P)\setminus \{p_k\}$. Since $j\neq k$ and $a_1\notin V(S_4)$, it follows from Claim~\ref{structure} that there are edges between $P^*$ and $V(C)$. Now by Claim~\ref{claim:S1,S2-select},
there is a sector $S_3$ of $W$ with 
ends $a_2,b_1$,  and $b_1$ has a neighbor in $P^*$. Let $p_{j'}$ be the neighbor of $b_1$ in $P^*$ closest to  $p_{k-1}$. By Claim~\ref{S_0T_0}, we have that $vb_2\notin E(G)$. In particular, $N(v)\cap V(C)\subseteq Q_1^*$. Note that $|V(p_{j'}Pp_jv)|\leq k-1$. However, it follows that $p_{j'}Pp_jv$ is a path violates the assumption of lemma, a contradiction. This proves Claim~\ref{claim:i=j}.

\begin{claim}\label{claim:v-nism-x}
    $vx\notin E(G)$.
\end{claim}

Otherwise, assume that $vx\in E(G)$. By Lemma~\ref{lem:edge-N(x)} and $W$ has at least five spokes, we have that $p_i$ is anticomplete to $V(C)$, and so $p_i\in P^*$.  Since $G$ is diamond-free, $N(v)\cap N(x)\cap V(C)=\emptyset$. Let $j_1\in [1,i]$ be maximum index such that $p_{j_1}$ has a neighbor in $V(C)$, and let $j_2\in [i,k]$ be minimum index such that $p_{j_2}$ has a neighbor in $V(C)$. Then $i\leq j_1<i<j_2\leq k$. By Lemma~\ref{lem:K13}, we have that $v$ has no neighbor in $V(C)$, a contradiction. This proves Claim~\ref{claim:v-nism-x}.

\begin{claim}\label{ends}
     $i\leq 2$ and $i\geq k-1$.
\end{claim}

To the contrary, we may assume that $k-i>1$. Consequently $k>2$. Suppose   that $S_4$ can be chosen so that $a_1 \in V(S_4)$. If $v$ has a unique
neighbor $v'$  in $V(S_4)$, then, since $s \neq b_1$,  $p_i$ can be linked to $xS_4x$ by $p_ivv'$, $p_ix$ and $p_iPp_1sS_1a_1$, a contradiction. Thus  $v$
has at least two neighbors in $V(S_4)$.  Since $v$ is proper for $W$, $v$ has exactly neighbors in $V(S_4)$. However, it follows that  $v$ can be linked to the hole $xS_4x$ by two one-edge paths and $vp_ix$, a contradiction.  Thus $S_4$ cannot be chosen so that $a_1 \in V(S_4)$.

Let $S_5$ be the sector of $W$ with end $a_1$ such that $V(S_5) \subseteq V(Q_1)$. Since $i \leq k-2$, by assumption
 $t=a_1$ and 
$V(S_4) \cap V(S_5) \neq \emptyset$.  
Let $b_3$ be the other end of $S_5$ distinct from $a_1$, and let $v'$ be the neighbor of $v$ in $V(S_4)$ closest to $b_3$.  If $v'$ is not the end of $S_4$ distinct from $b_3$, then $x$ has exactly three neighbors in the hole $p_1Pp_ivv'S_4b_3S_5a_1p_1$, a contradiction. Therefore, $v'$ is the end of $S_4$ distinct from $b_3$.

By Lemma~\ref{lem:triangle-wheel-4} and $W$ has at least five spokes,  we have that $i\neq 1$. It follows from Claim~\ref{structure} that there
is some edge between $P^*$ and $V(C)$, and by Claim~\ref{claim:S1,S2-select}, there
is a sector $S_3$ of $W$ with ends $b_1,a_2$ and every edge from 
$p_i$ to $V(C)$ have an end in $V(S_3)$. Note that $p_i$ is anticomplete to $V(C)$. Otherwise, by assumption (using the path
$p_iv$),  $p_i$ has  exactly one neighbor $a_2$ in $V(C)$ and $v$ has exactly one neighbor  $v'$ in $V(C)$.  This implies  that $S_4$ can be chosen so that $b_2 \in V(S_4)$. By Lemma~\ref{lem:triangle-wheel-4}, we have that $W$ has   exactly four spokes, a contradiction. Therefore, $p_i$ is anticomplete to $V(C)$. Since $v$ is not proper for $W_1$, $v$ has a neighbor in $Q_1^*$. Let $j_1\in [1,i]$ be maximum such that $p_{j_1}$ has a neighbor in $V(C)$, and let $j_2\in [i,k]$ be minimum such that $p_{j_2}$ has a neighbor in $V(C)$. 
Then $1\leq j_1<i<j_2\leq k$. By Lemma~\ref{lem:K13}, we have that $v$ has no neighbor in $V(C)$, a contradiction. This proves  Claim~\ref{ends}.

It follows  from Claims~\ref{claim:i=j} and~\ref{claim:ends} that either $k=3$ and $i=j=2$, or $k=2$. 

Suppose $k=3$ and $i=j=2$.  Note that $v$ has a neighbor in $Q_1^*$. 
By Claim~\ref{structure}, there is some edge between $P^*$ and $V(C)$. So by Claim~\ref{claim:S1,S2-select} there exists a sector $S_3$ with ends $a_2$ and $b_1$ 
such that $p_2$ has neighbors in $V(S_3)$. 
By assumption $N(p_2)\cap V(C)\subseteq \{b_1,a_2\}$, 
and either $a_1\in V(S_4)$ or $b_2\in V(S_4)$. Since $G$ is diamond-free, $|N(p_2)\cap  \{b_1,a_2\}|\leq1$.  By symmetry, we have that  $a_1\in V(S_4)$ and $N(p_2)\cap V(C)=\{b_1\}$. Let $v'$ be the neighbor of $v$ in $V(S_4)$ closest to $a_1$. By Lemma~\ref{lem:triangle-wheel-4} and $W$ has at least five spokes, 
$v'$ is not the end of $S_4$ distinct from $a_1$. Now, $p_2$ can be linked to the hole $xS_1x$ by $p_2x$, $p_2b_1$ and $p_2vv'S_4a_1$, a contradiction. Therefore,  $k=2$. By symmetry, we may assume that $i=1$. By Claim~\ref{structure},  we 
can choose $S_4$ with $a_1 \in V(S_4)$.  Since $W$ is proper, by Claim~\ref{claim:v-nism-x} it follows that $v$ has
at most two neighbors in $S_4$. If $v$ has a unique neighbor
$v'$ in $V(S_4)$, then $p_1$ can be linked to the hole $xS_4x$ by
$p_1vv'$, $p_1x$ and $p_1sS_1a_1$,  a contradiction. Therefore, $v$ has two neighbors in $V(S_4)$.  If $v$ has two non-adjacent neighbors in $V(S_4)$, then $p_1$ can be linked to the hole obtained from $xS_4x$ by rerouting $S_4$ through $v$ by
$p_1v$, $p_1x$ and $p_1sS_1a_1$, a contradiction. Therefore, $v$ has  exactly two adjacent neighbors in $V(S_4)$. If $s\neq a_1$, then $v$ can be linked to the hole $xS_4x$ by two one-edge paths and $vp_1x$,  a contradiction. Then $a_1=s$. Since $G$ is diamond-free, $a_1v\notin E(G)$. However, it follows that $G[\{p_1,x\}\cup V(S_4)]$ is a prism, a contradiction. This completes the proof of Lemma~\ref{lem:proper-wheel-path-x-neighbor}.
\end{proof}

Combining Lemmas~\ref{lem:proper-wheel-path-4-spokes}, 
\ref{lem:proper-wheel-path-x-no-neighbor} and 
\ref{lem:proper-wheel-path-x-neighbor}, we obtain the following theorem immediately.

\begin{theorem}\label{thm:proper-wheel-path}
Let $G$ be an   $\iskdbpk$-free graph, and let $W=(C,x)$ be a proper wheel in $G$ of center $x$ with minimum number of spokes.  Let $P=p_1\cdots p_k$ be a path with $V(P)\subseteq V(G)\setminus V(W)$ such that $x$ has at  most one neighbor in $P$. 
\begin{itemize}
\item[$(\romannumeral1)$] If $P$ contains no neighbor of $x$, then there is a sector $S$ of
$W$ such that every edge from $P$ to $C$ has an end in $V(S)$.
\item[$(\romannumeral2)$] If $P$ contains exactly one neighbor of $x$, then there are two
sectors $S_1,S_2$ of $W$ such that $V(S_1) \cap V(S_2) \neq \emptyset$, and
 every edge from $P$ to $C$ has an end in $V(S_1) \cup V(S_2)$ (where possibly $S_1=S_2$).
\end{itemize}
\end{theorem}
\begin{theorem}\label{thm:proper-wheel-property}
	Let $G$ be an $\iskdbpk$-free graph, and let $x$ be the center of a proper wheel in $G$. If $W=(C,x)$
    is a proper wheel in $G$ of center $x$ with minimum number of spokes, then 
	\begin{itemize}
		\item[$(\romannumeral1)$] every component of $G\setminus N[x]$  contains the interior of at most one sector of $W$, and
		
		\item[$(\romannumeral2)$] for every $u\in N(x)$, the component  $D$  of $G\setminus (N[x]\setminus \{u\})$ such that $u\in V(D)$ contains the interiors of at most two sectors of $W$, and if $S_1,S_2$ are sectors with $S_i^*\subseteq V(D)$ for $i=1,2$, then $V(S_1)\cap V(S_2)\neq \emptyset$.
	\end{itemize}
\end{theorem}
\begin{proof}
To prove $(\romannumeral1)$ of Theorem~\ref{thm:proper-wheel-property}, we observe that if some  component of 
$V(G) \setminus N(x)$ contains the interiors of two sectors of $W$, then this 
component contains a path violating $(\romannumeral1)$ of Theorem~\ref{thm:proper-wheel-path}. 
For $(\romannumeral2)$ of Theorem~\ref{thm:proper-wheel-property}, suppose $D$ contains the interiors of two disjoint 
sectors $S_1,S_2$ of
$W$. Since $|D \cap N(x)|=1$, there exists a path in $D$ violating $(\romannumeral2)$ of Theorem~\ref{thm:proper-wheel-path}.
This completes the proof of Theorem~\ref{thm:proper-wheel-property}. 
\end{proof}

As an immediate corollary of Theorem~\ref{thm:proper-wheel-property}, we obtain the following.
\begin{corollary}\label{cor:component-number}
Let $G$ be an $\iskdbpk$-free graph and $x$ be the center of a proper wheel in $G$. Then $G- N[x]$ has at least three components.
\end{corollary}

\section{Proper wheel centers}
In this section, we perform certain transformations on $\iskdbpk$-free graphs, ensuring that these operations preserve the $\iskdbpk$-free property and do not  create  any new centers of proper wheels.
\begin{lemma}\label{lem:component-center}
Let $G$ be an $\iskdbpk$-free graph with $ s \in V(G) $. Let $K $ be a component of $ G-N[s] $, and let $ N $ be the set of vertices in $N(s)$ that have a neighbor in $ K $. Define $ H = G[V(K) \cup N \cup \{s\}] $. Then $s$ is not the center of a proper wheel in $H$. Furthermore,  if $v$ is the center of a proper wheel in $ H $, then $v $ is also the center of a proper wheel in $G$.
\end{lemma}
\begin{proof}
Since $H-N[s]$ is connected, it follows that $s$ is not the center of a proper wheel in $H$ by Corollary~\ref{cor:component-number}. Let $v\in V(H)\setminus \{s\}$ be the center of a proper wheel $W=(C,\,v)$ in $H$. For all $w\in V(G)\setminus V(H)$, $N(w)\cap V(C)\subseteq N[s]$.
Suppose that $W$ is not proper in $G$. Then there exists a vertex $w$ such that either $w$ has more than two neighbors in some sector of $W$ but $wv\notin E(G)$, or $w$ has neighbors in at least two distinct sectors of $W$. It follows that $w$ has more than one neighbor in $V(C)$.  First, assume that $w\in N[s]$. Since $G$ is diamond-free, $w$ either has at most one neighbor in $V(C)$ or has exactly two consecutive neighbors on $C$ (only when $s \in V(C)$). Hence $w$ is proper for $W$, a contradiction.
This implies that $w\in V(G)\setminus (V(H)\cup N[s])$. Since $G$ is diamond-free, $N(w)\cap N(s)$ is an independent set.
Since $w$ is not proper for $W$, $|N(w)\cap N(s)|\geq 2$. 
Suppose that $w$ has three distinct neighbors $a,b,c$ in $V(C)\cap N(s)$. Let $P$ be a shortest path connecting two of $a,b,c$, say $a$ and $b$, with interior in $K$. Note that $s$ is anticomplete to $P^*$.  If $c$ is anticomplete to $P$, then $G[V(P)\cup \{w,s,a,b,c\}]$ is an $\iskfour$, a contradiction.
By the choice of $P$, we have that $P^*$ consists of a single vertex $x$ and $cx\in E(G)$. However, it follows that $G[\{w,s,x,a,b,c\}]$ is a $K_{3,3}$, a contradiction. 
So $w$ has exactly two neighbors $a$ and $b$ in different sectors of $W$, say $S,S'$. 
Since $a,b\in N(s)$ and $s$ is proper for $W$, $s\in V(S)\cap V(S')$.
Then, $v$ can be linked to the cycle $wasbw$ by $vs$ and the two paths with interiors in $S-\{s\}$ and $S'-\{s\}$, a contradiction.
This completes the proof of Lemma~\ref{lem:component-center}.
\end{proof}
\begin{lemma}\label{lem:series-parallel-stable}
Let $G$ be an $\iskdbpk$-free graph without clique cutsets, and let $s$ be the center of a proper wheel in $G$. Let $K$ be a component of $G - N[s]$, and let $N \subseteq N(s)$ be the set of neighbors of $s$ that have a neighbor in $K$. If the induced subgraph $G[V(K) \cup N \cup \{s\}]$ is series-parallel, then $N$ is an independent set in $G$.
\end{lemma}
\begin{proof}
Since $G$ is \{diamond, bowtie\}-free, $e(G[N]) \le 1$. Suppose to the contrary that $G[N]$ has an edge $ab$. 
Since $G $ has no clique cutset, $|N|\geq 3$.  
Let $c\in N\setminus\{a,b\}$, and let $P$ be a path connecting $a$ and $b$ with interior in $K$.  
Since $c$ has a neighbor in $K$, there exists a path that starts at $c$, terminates in $P^*$, and has its interior contained in $K$.
This implies that $G[\{s,a,b,c\}\cup V(P)\cup V(P')]$ is a subdivision of $K_4$, and hence $G[V(K)\cup N\cup\{s\}]$ is not series-parallel, a contradiction. This completes the proof of Lemma~\ref{lem:series-parallel-stable}.
\end{proof}
A {\it tripod graph} is a triangle with an induced path attached to each vertex (see Figure~\ref{Figure-tripod}).
\begin{figure}[ht]
	\begin{center}
		\begin{tikzpicture}
			
			\tikzset{std node fill/.style={draw=black, circle,fill=black, line width=1pt, inner sep=2pt}}

		\def\dx{0.8}  
		
		\node[std node fill] (t1) at ({8+\dx},-3) {};
		\node[std node fill] (t2) at ({9+\dx},-3) {};
		\node[std node fill] (t3) at ({8.5+\dx},-2) {};
		
		\foreach \i/\j in {t1/t2,t2/t3,t3/t1}{\draw[blue,line width=1.5pt] (\i)--(\j);}
		
		\node[std node fill] (ta1) at ({7.6+\dx},-3.4) {};
		\node[std node fill] (ta2) at ({7.2+\dx},-3.8) {};
		\draw[blue,line width=1.5pt] (t1) -- (ta1) -- (ta2);
		
		\node[std node fill] (tb1) at ({9.4+\dx},-3.4) {};
		\node[std node fill] (tb2) at ({9.8+\dx},-3.8) {};
		\draw[blue,line width=1.5pt] (t2) -- (tb1) -- (tb2);
		
		\node[std node fill] (tc1) at ({8.5+\dx},-1.5) {};
		\node[std node fill] (tc2) at ({8.5+\dx},-1) {};
		\draw[blue,line width=1.5pt] (t3) -- (tc1) -- (tc2);
		\end{tikzpicture}
\caption{\small Illustrations of the configuration: a tripod graph.}

		\label{Figure-tripod}
	\end{center}
\end{figure}
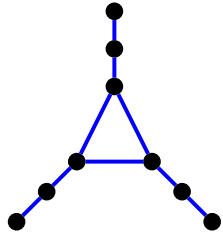
 \begin{lemma}[Chudnovsky et al. \cite{Chudnovsky2019}]\label{lem:degree-1}
	Let $G$ be a connected graph, let $a,b,c\in V(G)$ with $\deg(a)=\deg(b)=\deg(c)=1$, and let $H$ be a connected induced subgraph of $G$ containing $a,b,c$ with $V(H)$  minimum. Then, either $H$ is a subdivision of $K_{1,3}$ with $a,b,c$ being the vertices of degree one, or $H$ is a tripod graph with $a,b,c$ being the vertices of degree one.
\end{lemma}
\begin{lemma}\label{lem:series-parallel-contracting}
Let $G$ be an $\iskdbpk$-free graph without clique cutsets, and let $s$ be the center of a proper wheel in $G$. Let $K$ be a component of $G - N[s]$, and let $N$ denote the set of vertices in $N(s)$ having a neighbor in $K$. Let $G'$ be the graph obtained from $G$ by contracting $V(K)$ into a new vertex $z$. If $G[V(K)\cup N\cup \{s\}]$ is series-parallel, then $G'$ is $\iskdbpk$-free graph without clique cutsets.
\end{lemma}
\begin{proof}
By Lemma~\ref{lem:series-parallel-stable}, we have that  $N$ is an independent set,  $z$ is not contained in any triangle in $G'$, and thus $G'$ is \{diamond, bowtie\}-free.  Suppose that $H$ is an induced subgraph of $G'$, which is a $K_{3,3}$,  or a prism, or a subdivision of $K_4$. Then $z\in V(H)$. If $z$ has degree two in $H$, let $a,b$ be its neighbors in $H$. 
Then we can replace the path $azb$ by a path $P$ connecting $a$ and $b$ with interior in $K$, obtaining a subdivision of $H$.
This yields an induced $\iskfour$ in $G$ or a prism, a contradiction. 
Therefore, $z$ has degree three in $H$; let $a,b,c$ denote the neighbors of $z$ in $H$. Let $P$ be a shortest path connected $a$ and $b$  with interior in $K$. Then, $c$ has at most one neighbor on $P$; for otherwise $G[V(P)\cup \{a,b,c,s\}]$ is a wheel, a contradiction to the fact that  $G[V(K)\cup N\cup\{s\}]$ is series-parallel. 
Let $Q$ be a shortest path that starts at $c$, terminates in $P^*$, and has its interior contained in $K$; then, by symmetry, each of $a, b, c$ has a unique neighbor in $V(Q) \cup V(P)$. Let $H'$ be a minimum connected induced subgraph of  $G[V(P)\cup V(Q)]$ containing $a,b,c$.  
By Lemma~\ref{lem:degree-1}, we have that either $H'$ is a subdivision of $K_{1,3}$ with $a,b,c$ being the vertices of degree one, or $H'$ is a tripod graph  with $a,b,c$ being the vertices of degree one. 
If $H'$ is a tripod graph  with $a,b,c$ as the vertices of degree one, then $G[V(H')\cup \{s,a,b,c\}]$  is an $\iskfour$, a contradiction. 
Therefore, $H'$ is a subdivision of $K_{1,3}$ with $a,b,c$ as the vertices of degree one. This implies that $G[V(H')\cup V(H-z)]$ is a subdivision of $H$.
Recall that $z$ is not contained in any triangle. So $H$ is not a prism.
However, it follows that $G$ contains an $\iskfour$ or by Lemma~\ref{lem:ISK4-decomposition} a $K_{3,3}$, a contradiction. This completes the proof of Lemma~\ref{lem:series-parallel-contracting}.
\end{proof}

\begin{lemma} \label{lem:non-centers} Let $G$ be an $\iskdbpk$-free graph without clique cutsets and let $s$ be the center of a proper wheel in $G$. Let $K$ be a component of $G \setminus N[s]$, and let $N$ be the set of vertices in $N(s)$ with a neighbor in $K$.
Suppose $H = G[(V(K) \cup N \cup \{s\}]$ is series-parallel (so $N$ is an independent set). Let $G'$ be obtained from $G$ by contracting $V(K)$ to a new vertex $z$. Then $z$ is not the center of a proper wheel in $G'$, and for $v \in V(G') \setminus \{s,z\}$, if $v$ is the center of a proper wheel in $G'$, then $v$ is the center of a proper wheel in $G$.
\end{lemma}
\begin{proof}
Since $N_{G'}(z) \subseteq N_{G'}(s)$, it follows from Corollary~\ref{cor:component-number} that $z$ is not the center of a proper wheel in $G'$. Let $v \in V(G') \setminus \{s,z\}$ be the center of a proper wheel in $G'$, and let $W = (C, v)$ be such a wheel with minimum number of spokes. Since $G'$ is $\iskdbpk$-free by Lemma~\ref{lem:series-parallel-contracting}, it follows that $W$ satisfies the hypotheses of  Theorem~\ref{thm:proper-wheel-property}. We complete the proof of Lemma~\ref{lem:non-centers} by analyzing the following cases.
\begin{case}\label{case: z in C}
	$z\in V(C)$. 
\end{case}

Let $a, b$ denote the neighbors of $z$ in $V(C)$. Let $P$ be a shortest path connecting $a$ and $b$ with interior in $K$. Then every vertex in $V(K)$ has at most two neighbors in $V(P)$ since $G$ is diamond-free. Let $W' = (C',v)$ be the wheel in $G$ obtained from $W$ by replacing the subpath $azb$ of $C$ by $P$ to obtain $C'$. 

It remains to show that $W'$ is a proper wheel in $G$. Since  $s$ is proper for $W$ in $G'$, $zv\notin E(G)$. Therefore, $azb$ is a subpath of some sector of $W$. This  implies that $P$ is a subpath  of some sector of $W'$. 

Now suppose that there is a vertex $x$ which is not proper for $W'$ in $G$.  Clearly, $x\in V(H)$ and $N(x)\cap P^*\neq \emptyset$.  If $x \in K$, then $N_G(x) \cap V(C') \subseteq V(P)$, and so $|N_G(x) \cap V(C')| \leq 2$ by the minimality of $|V(P)|$, a contradiction; it follows that $x \notin K$. Therefore, $x\in N$. Suppose that  $x $ has two or more neighbors in $P^*$. Then $H[\{a,b,x,s\} \cup V(P)]$ induces a wheel in $H$, contradicting the fact that $H$ is series-parallel. Consequently, $x$ has exactly one neighbor in $P^*$. Note that $x$  is proper for $W$ and $xz\in E(G')$. This implies that  $x$ is proper for $W'$ in $G$, a contradiction. This concludes the proof of Case~\ref{case: z in C}.
\begin{case}\label{case:z notin C}
	$z\notin V(C)$
\end{case}

Clearly, $W$ is a wheel in $G$. Since $W$ is proper in $G'$, there exists a sector $S$ of $W$ that contains all neighbors of $z$ on $C$. Hence, for every $x \in K$, we have 
\begin{flalign}\label{align:x-neighour-1}
    N_G(x) \cap V(C) \subseteq N_{G'}(z) \cap V(C) \subseteq V(S).
\end{flalign}

Suppose that $zv\notin E(G)$. If $x \in G \setminus (V(C) \cup V(K))$, then $x$ is proper for $W$ is in $G$ as $x$ is proper for $W$ in $G'$. Now consider a vertex $x \in V(K)$. Since  $W$ is proper in $G'$, it follows that $|N_{G'}(z) \cap V(C)| \leq 2$. Then by~\eqref{align:x-neighour-1}, $|N_G(x) \cap V(C)| \leq 2$, and hence $x$ is proper for $W$ in $G$. 
Therefore, we may assume that $zv\in E(G)$,
and thus $v\in N$.
Let $a$ and $b$ be the ends of $S$. 
Let $z_1,z_2$ be the neighbors of $z$ such that $z_1$ is closest to $a$ and $z _2 $ is closest to $b$.
If $W$ is not proper in $G$, then there exists $y\in V(K)$ such that $y$ is not adjacent to $v$ and has at least three neighbors in $S$. Let $y_1$ (resp. $y_2$) be the neighbor of $y$ closest to $a$ (resp.  to $b$) on $S$. Furthermore, assume that $y$ is chosen such that  $|V(y_1Sy_2)|$ is as large as possible. 
Let $W'=(C',v)$ be the wheel obtained from $W$ by replacing the subpath $y_1Sy_2$ by $y_1yy_2$.  We show that $W'$ is a proper wheel in $G$. Now suppose that there is a vertex $u$ which is not proper for $W'$ in $G$. If $u\in (y_1Sy_2)^*$, then $uv\notin E(G)$ and $N_G(u)\cap V(C')\subseteq \{y_1,y,y_2\}$. 
Since $G$ is diamond-free, $|N_G(u)\cap V(C')|\leq 2$. Then $u$ is proper for $W'$, a contradiction. Therefore, $u\in V(G)\setminus V(W)$.

If $uv\in E(G)$, first we suppose that $u\in K$, then $u$ is proper for $W'$, a contradiction.
So $u\in V(G)\setminus (V(W)\cup V(K))$.
Since $v\in N$ and $N$ is an independent set,  $uz\notin E(G')$ and hence $uy\notin E(G)$.
Since $u$ is proper for $W$, 
 $u$ is proper for $W'$, a contradiction.
 So $uv\notin E(G)$ and $|N(u)\cap V(C)|\le 2$.

  If $N_{G'}(u)\cap V(C)\subseteq  V(S) $, then $N_{G}(u)\cap V(C')\subseteq  V(aSy_1yy_2Sb)$. By Lemma~\ref{lem:nolink}, we have that $|N_{G}(u)\cap V(C')|\leq 2$. This implies that $u$ is proper for $W'$, a contradiction. Therefore, there exists a sector $T$ of $W$ distinct from $S$  such that $N_{G'}(u)\cap V(C)\subseteq V(T)$. 
 
Since $u$ is not proper for $W'$, $uy\in E(G)$ and hence $uz\in E(G')$.
 By Theorem~\ref{thm:proper-wheel-path}, we have that  $V(S)\cap V(T)\neq \emptyset$. Without loss of generality, we may assume that $V(T)\cap V(S)=\{b\}$. 
Let $u_1$ be the neighbor of $u$ in $T$ closest to $b$.

If $u$ has exactly only neighbor in $V(T)$. That is, $u_1\neq b$. But now, 
$z$ can be linked to $vTv$ by $zv,zz_2Sb$ and $zuu_1$.
So $z $ has exactly two neighbors in $T$ say $u_1,u_2$.
If $u_1u_2\notin E(G)$, $z$ can be linked to  the hole obtained from $vTv$  by rerouting $T$ through $u$ by $zv$, $zu$ and $zz_2Sb$.
Therefore, $u_1u_2\in E(G)$, if $b=z_2$, then $b\neq u_1$ since $G'$ is diamond-free, but now $\{z,u,v,b,u_1,u_2\}\cup V(T)$ induces a prism, a contradiction.
Therefore, $b\neq z_2$, but now $u$ can be linked to the hole $vTv$ by $uu_1$, $uu_2$ and $uzv$, a contradiction.
This completes the proof of Lemma~\ref{lem:non-centers}.
\end{proof}

\section{\texorpdfstring{Properties of $\isktk$-free graphs}{Properties of $\isktk$-free graphs}}
In this section, we present techniques for our main theorem that serve to find either a vertex of degree one or a cycle with almost all vertices of degree two. A \emph{branch-vertex} in a graph $G$ is a vertex of degree at least $3$. A branch is a path of $G$ of length at least one whose ends are branch-vertices and whose internal vertices are not (so they all have degree $2$ in $G$). 
\begin{lemma}\label{lem:path}
Let $T$ be a tree of order at least two. Then one of the following three conditions holds:
\begin{itemize}
    \item[$(\romannumeral1)$] $T$ is a path;
    \item[$(\romannumeral2)$] there exist two vertex-disjoint paths $P_1$ and $P_2$ in $T$ whose ends are distinct leaves of $T$, and each of which contains one vertex of degree at least three in $T$;
    \item[$(\romannumeral3)$] $T$ is a subdivision of a star with at least three leaves.
\end{itemize}
\end{lemma}
\begin{proof}
We proceed by induction on the order of $T$. If $T$ has exactly two leaves, then $(\romannumeral1)$ holds. Hence, we may assume that $T$ has at least three leaves. 
Let $v$ be a leaf of $T$. Now we consider the tree $T-v$. 

If $(\romannumeral1)$ holds for $T-v$, then $(\romannumeral3)$ holds for $T$.

Suppose that $(\romannumeral2)$ holds for $T-v$. Then there exist two vertex-disjoint paths $P_1$ and $P_2$ in $T-v$ whose ends are distinct leaves of $T-v$, and each of which contains one vertex of degree at least three in $T$. If $v$ has no neighbor of $V(P_1)\cup V(P_2)$, then $(\romannumeral2)$ holds for $T$. Therefore, $v$ has exactly one neighbor in $V(P_1)\cup V(P_2)$. Without loss of generality, we may assume that $v$ has exactly one neighbor in $V(P_1)$. This implies that $(\romannumeral2)$ holds for $T$ (see Figure~\ref{fig:two-path}).

\begin{figure}[htbp]
\centering

\begin{minipage}{0.48\textwidth}
\centering
\begin{tikzpicture}[
  scale=0.55,
  transform shape,
  vertex/.style={circle,draw,fill=black,minimum size=1.1em,inner sep=1.2pt},
  edge/.style={thick},
  dedge/.style={edge,dashed},
  highlight/.style={red, line width=3pt, opacity=0.45}
]

\node[vertex] (0) at (-5,5) {};
\node[vertex] (1) at (-5,3) {};
\node[vertex] (2) at (-5,1) {};
\node[vertex] (3) at (1,5) {};
\node[vertex] (4) at (1,3) {};
\node[vertex] (5) at (1,1) {};
\node[vertex,label=left:$v$] (6) at (-6,6) {};

\draw[highlight] (6)--(0);
\draw[highlight] (2)--(0);
\draw[highlight] (3)--(5);

\path
(0) edge (6)
(0) edge[dedge] (1)
(1) edge[dedge] (2)
(3) edge[dedge] (4)
(4) edge[dedge] (5)
(1) edge[dedge] (4);

\end{tikzpicture}
\end{minipage}
\hfill
\begin{minipage}{0.48\textwidth}
\centering
\begin{tikzpicture}[
  scale=0.55,
  transform shape,
  vertex/.style={circle,draw,fill=black,minimum size=1.1em,inner sep=1.2pt},
  edge/.style={thick},
  dedge/.style={edge,dashed},
  highlight/.style={red, line width=3pt, opacity=0.45}
]

\node[vertex] (0) at (-5,5) {};
\node[vertex] (1) at (-5,3) {};
\node[vertex] (2) at (-5,1) {};
\node[vertex] (3) at (1,5) {};
\node[vertex] (4) at (1,3) {};
\node[vertex] (5) at (1,1) {};
\node[vertex,label=left:$v$] (6) at (-6,6) {};
\node[vertex] (7) at (-5,4) {};

\draw[highlight] (6)--(7);
\draw[highlight] (2)--(7);
\draw[highlight] (3)--(5);

\path
(7) edge (6)
(0) edge[dedge] (7)
(1) edge[dedge] (7)
(1) edge[dedge] (2)
(3) edge[dedge] (4)
(4) edge[dedge] (5)
(1) edge[dedge] (4);

\end{tikzpicture}
\end{minipage}

\caption{Illustrations of the possible scenario of two vertex-disjoint paths $P_1$ and $P_2$ in $T$, highlighted in red.}
\label{fig:two-path}
\end{figure}
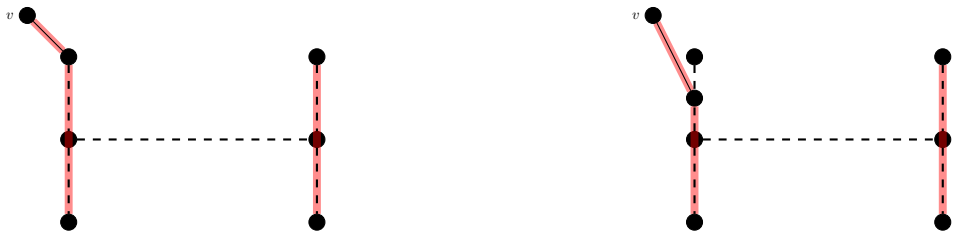

Therefore, $(\romannumeral3)$ holds for $T-v$. 
If $v$ is adjacent to the unique branch vertex or to a leaf of $T-v$, then $(\romannumeral3)$ also holds for $T$. 
Otherwise, $(\romannumeral2)$ holds for $T-v$ (see Figure~\ref{fig:star}).

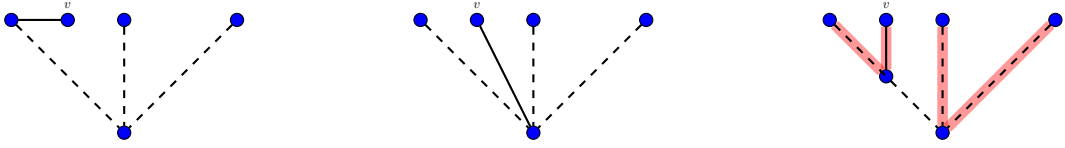
\begin{figure}[htbp]
\centering

\begin{minipage}{0.32\textwidth}
\centering
\begin{tikzpicture}[
  scale=0.5,
  transform shape,
  main_node/.style={circle,draw,fill=blue,minimum size=1em,inner sep=3pt}
]
\node[main_node] (0) at (0,0) {};
\node[main_node] (1) at (-3,3) {};
\node[main_node] (2) at (0,3) {};
\node[main_node] (3) at (3,3) {};
\node[main_node,label=above:$v$] (4) at (-1.5,3) {};

\path[draw, thick] (1) edge (4); 
\path[draw, thick, dashed] 
  (3) edge (0)
  (2) edge (0)
  
  (1) edge (0);
\end{tikzpicture}
\end{minipage}
\hfill
\begin{minipage}{0.32\textwidth}
\centering
\begin{tikzpicture}[
  scale=0.5,
  transform shape,
  main_node/.style={circle,draw,fill=blue,minimum size=1em,inner sep=3pt}
]
\node[main_node] (0) at (0,0) {};
\node[main_node] (1) at (-3,3) {};
\node[main_node] (2) at (0,3) {};
\node[main_node] (3) at (3,3) {};
\node[main_node,label=above:$v$] (4) at (-1.5,3) {};

\path[draw, thick] (0) edge (4); 
\path[draw, thick, dashed] 
  (3) edge (0)
  (2) edge (0)
  (1) edge (0);
\end{tikzpicture}
\end{minipage}
\hfill
\begin{minipage}{0.32\textwidth}
\centering
\begin{tikzpicture}[
  scale=0.5,
  transform shape,
  main_node/.style={circle,draw,fill=blue,minimum size=1em,inner sep=3pt}
]
\node[main_node] (0) at (0,0) {};
\node[main_node] (1) at (-3,3) {};
\node[main_node] (2) at (0,3) {};
\node[main_node] (3) at (3,3) {};
\node[main_node] (4) at (-1.5,1.5) {};
\node[main_node,label=above:$v$] (5) at (-1.5,3) {};
 \draw[opacity=0.4,red, line width=4pt] (5) to (4);
  \draw[opacity=0.4,red, line width=4pt] (4) to (1);
  \draw[opacity=0.4,red, line width=4pt] (2) to (0);
  \draw[opacity=0.4,red, line width=4pt] (3) to (0);
\path[draw, thick] (5) edge (4); 
\path[draw, thick, dashed] 
  (3) edge (0)
  (2) edge (0)
  (1) edge (0);
\end{tikzpicture}
\end{minipage}

\caption{Illustration of three cases of $(\romannumeral3)$ for $T-v$.}
\label{fig:star}
\end{figure}
This completes the proof of Lemma~\ref{lem:path}.
\end{proof}
\begin{lemma}\label{lem:tree}
Let $G$ be a graph and let $x \in V(G)$ be such that $G-x$ is a forest. Then one of the following five conditions holds:
\begin{itemize}
 \item[$(\romannumeral1)$] $V(G) = N[x]$, and $N(x)$ is an independent set;
\item[$(\romannumeral2)$] $V(G) \setminus N[x]$ contains a vertex of degree at most one in $G$;
 \item[$(\romannumeral3)$] $G$ contains an induced cycle $C$ passing through $x$, such that $V(C) \setminus \{x\}$ contains at most one vertex of degree greater than two in $G$. Moreover, if such a vertex exists, it is a neighbor of $x$ on $C$;
 \item[$(\romannumeral4)$] $G$ contains two induced cycles $C_1$ and $C_2$ such that
$V(C_1)\cap V(C_2)=\{x\}$, and for each $i\in\{1,2\}$,
the set $V(C_i)\setminus\{x\}$ contains at most one vertex of degree greater than two in $G$. In particular, there exists a path in $G-x$ whose ends lie in $C_1\setminus\{x\}$ and $C_2\setminus\{x\}$, respectively.

\item[(\romannumeral5)] $G$ contains an induced theta subgraph $\theta$ such that $x$ is a branch vertex of $\theta$, and every vertex of $\theta$ other than the branch vertices has degree two in $G$.

\end{itemize}
\end{lemma}

\begin{proof}
If every component of $G-x$ has exactly one vertex, then either $V(G) = N[x]$, in which $(\romannumeral1)$ holds, or $V(G) \setminus N[x]$ contains a vertex of degree zero, in which $(\romannumeral2)$ holds. If $G-x$ has a leaf that is not adjacent to $x$, then $(\romannumeral2)$ holds.
Hence, we may assume that $G-x$ has a component $T$ with at least two vertices, and that every leaf of $G-x$ is adjacent to $x$. If $(\romannumeral1)$ of Lemma~\ref{lem:path}
 holds for $T$, then $(\romannumeral3)$ of 
 holds for $G$ and $x$.

Suppose that $(\romannumeral2)$ of Lemma~\ref{lem:path}
 holds for $T$. Therefore, there exist two vertex-disjoint paths $P_1$ and $P_2$ in $T$ whose ends are distinct leaves of $T$, each containing  one vertex of degree at least three in $T$. Since $T$ is a tree, there exists a path in $G-x$ whose ends lie in $P_1^*$ and $P_2^*$, respectively. If $x$ has no neighbor in $P_1^*\cup P_2^*$,  then $(\romannumeral4)$ 
 holds for $G$ and $x$.  Therefore, Without loss of generality, we may assume that $x$ has a neighbor in $P_1^*$. For convenience, we assume that the ends of $P_1$ are $v$ and $v'$. Since $T$ is connected, $P_1$ contains a vertex of degree greater than two in $T$, which we denote by $a$. Let $v_1$ be the neighbor of $x$ on $P-v$ closest to $v$, and let $v_2$ be the neighbor of $x$ on $P-v'$ closest to $v'$. Then set $C = x v P v_1 x$ or $C = x v_2 P v' x$, and observe that every vertex of $C$, except possibly $x$ and $v_1$ (or $v_2$), has degree two in $G$. This implies that  then $(\romannumeral3)$ 
 holds for $G$ and $x$.

Suppose that $(\romannumeral3)$ of Lemma~\ref{lem:path} holds for $T$. That is, $T$ is a subdivision of a star with at least three leaves.  
 For convenience, we may assume that $y$ is the branch vertex of $T$ and that $v_1,\ldots,v_t$ are leaves of $T$. 
For each $i$, let $P_{v_i}$ denote the branch in $T$ with ends $x$ and $v_i$. If $x$ has no neighbor in $V(T)\setminus \{v_1,\ldots,v_t\}$,  then $(\romannumeral5)$ 
 holds for $G$ and $x$. Therefore, Without loss of generality, we may assume that $x$ has a neighbor in $V(P_{v_1})\setminus \{v_1\}$. Let $x_1$ be the neighbor of $x$ in $V(P_{v_1})\setminus \{v_1\}$ closest to $v_1$. Now, $C = xx_1 P_{v_1} v_1 x$ is a cycle such that every vertex in $V(C) \setminus \{x,x_1\}$ has degree two in $G$.  Hence, $(\romannumeral3)$ 
 holds for $G$ and $x$. This completes the proof of Lemma~\ref{lem:tree}.
\end{proof}
\begin{lemma}[Chudnovsky et al. \cite{Chudnovsky2019}]\label{lem:farcycle}
	Let $G$  be a series-parallel graph, and let $x,y\in V(G)$ with $x=y$ or $xy\in E(G)$. 
	If  $G-\{x,y\}$ contains a cycle, then either
	\begin{itemize}
		\item[(\romannumeral1)] there is an induced cycle $C$ in $G$ such that $V(C)\cap \{x,y\}=\emptyset$ and all but at most two vertices of $C$ have degree two in $G$, or
		\item[(\romannumeral2)] $V(G)\setminus (N[x]\cup N[y])$ contains a vertex of degree at most one in $G$.
	\end{itemize}
\end{lemma}
As a direct consequence, we obtain the following stronger statement for cliques:

\begin{lemma}~\label{lem:series-parallel-clique-cycle}
	Let $G$ be a series-parallel graph, and let $S\subseteq V(G)$ be a clique. 
	If $G-S$ contains a cycle, then either
	\begin{itemize}
		\item[(\romannumeral1)] there exists an induced cycle $C$ in $G$ such that $V(C)\cap S=\emptyset$ and all but at most two vertices of $C$ have degree two in $G$, or
		\item[(\romannumeral2)] $V(G)\setminus N[S]$ contains a vertex of degree at most one in $G$.
	\end{itemize}
\end{lemma}
The following structural theorem for $\isktk$-free graphs is mainly based on Theorem 23 in \cite{Chudnovsky2019}, 
with more refined analysis. The proof is similar to that of Theorem 23 in \cite{Chudnovsky2019}; for completeness, we provide it here.
\begin{theorem}\label{main-theorem:triangle}
Let $G$ be an $\isktk$-free graph, and let $x,y\in V(G)$ with $x=y$ or $xy\in E(G)$. Then one of the following holds:
\begin{itemize}
\item[$(\romannumeral1)$] $V(G) = N[x] \cup N[y]$;

\item[$(\romannumeral2)$] there exists a vertex in $V(G) \setminus (N[x] \cup N[y])$ of degree at most one in $G$;

\item[$(\romannumeral3)$] there exists an induced cycle $C$ containing at least one of $x$ or $y$ such that every vertex in $V(C)\setminus (N[x]\cup N[y])$ has degree two in $G$;

\item[$(\romannumeral4)$] there exists an induced cycle $C$ in $G$ with $V(C) \cap \{x,y\} = \emptyset$, and all but at most two vertices of $C$ have degree two in $G$;

\item[$(\romannumeral5)$] there exists an induced cycle $C$ in $G-\{x,y\}$ and a vertex $r \in V(C)$ such that every vertex in $V(C)\setminus N[r]$ has degree two in $G$;

\item[$(\romannumeral6)$] there exist two induced cycles $C_1$ and $C_2$ such that 
$V(C_1)\cap V(C_2)$ contains exactly one vertex from $\{x,y\}$, and for each $i\in\{1,2\}$, 
the set $V(C_i)\setminus(N[x]\cup N[y])$ contains at most one vertex of degree greater than two in $G$;

\item[$(\romannumeral7)$] there exist two vertex-disjoint induced cycles $C_1$ and $C_2$ such that $x \in V(C_1)$ and $V(C_1)\setminus\{x\}$ contains at most one vertex of degree greater than two in $G$, and $y \in V(C_2)$ and $V(C_2)\setminus\{y\}$ contains at most one vertex of degree greater than two in $G$;

\item[$(\romannumeral8)$] there exists an induced theta subgraph $\theta$ containing $x$ or $y$ as a branch vertex such that  $V(\theta) \setminus (N[x] \cup N[y])$ contains a unique vertex of degree greater than three in $G$.

\item[$(\romannumeral9)$] there exists an induced theta subgraph $\theta$ in $G-\{x,y\}$ and a vertex $r \in V(\theta)$  is a branch vertex of $\theta$ such that $V(\theta) \setminus N[r]$ contains a unique vertex of degree greater than three in $G$.

\item[$(\romannumeral10)$] there exist two induced cycles $C_1$ and $C_2$ in $G-\{x,y\}$ and a vertex $r $ such that $V(C_1)\cap V(C_2)=\{r\}$  and for each $i\in\{1,2\}$, 
the set $V(C_i)\setminus N_{C_i}[r]$ contains at most one vertex of degree greater than two in $G$.
\end{itemize}
\end{theorem}
\begin{proof}
We proceed by induction on the order of $G$. The case $|V(G)|\le 3$ is trivial, so we may assume that $|V(G)|\ge 4$. We complete the proof of Lemma~\ref{main-theorem:triangle} by analyzing the following cases.
\begin{case}
$G$ is series-parallel. 
\end{case}

If $xy \in E(G)$, then  define $H = G / \{x,y\}$ and $v$ be the vertex obtained by contracting $xy$; otherwise, define $H=G$ and $v=x$. Then $H$ is also series-parallel. Suppose that $H-v$ is a forest. We apply Lemma~\ref{lem:tree} for $H$ and $v$. 
If Lemma~\ref{lem:tree} $(\romannumeral1)$ holds, then $V(H) = N_H[v]$, and so $V(G) = N_G[x] \cup N_G[y]$. Then $(\romannumeral1)$ holds for $G$ and $x,y$. If Lemma~\ref{lem:tree} $(\romannumeral2)$ holds, then $V(H) \setminus N_H[v]$ contains a vertex of degree at most one in $H$, and thus $V(G) \setminus (N_G[x] \cup N_G[y])$ contains a vertex of degree at most one in $G$. Then $(\romannumeral2)$ holds for $G$ and $x,y$.
If Lemma~\ref{lem:tree} $(\romannumeral3)$ holds, then $H$ has an induced cycle $C$ that contains $v$, and $V(C)\setminus \{v\}$ has at most one vertex with degree greater than  two in $H$;
Moreover, if such vertex exists, it is a neighbor of $v$ on $C$. Note that $G$ is triangle-free. So there is an induced cycle $C'$ in $G$ containing at least one of $x, y$, and $V(C')\setminus \{x,y\}$ has at most one vertex with degree greater than  two in $G$; Moreover, if such vertex exists, it is a neighbor of $x$ or $y$ on $C$. Then $(\romannumeral3)$ holds for $G$ and $x,y$.

Suppose that  Lemma~\ref{lem:tree} $(\romannumeral4)$ holds. Then $H$ contains two induced cycles $C_1$ and $C_2$ such that
$V(C_1)\cap V(C_2)=\{v\}$, and for each $i\in\{1,2\}$,
the set $V(C_i)\setminus\{v\}$ contains at most one vertex of degree greater than two in $G$. In particular, there exists a path in $H-v$ whose ends lie in $C_1\setminus\{v\}$ and $C_2\setminus\{v\}$, respectively. For convenience, we may assume that $C_1=vv_1,\ldots,v_sv$ and $C_2=vu_1\ldots u_tv$. If $x=y$,  then $(\romannumeral4)$ holds for $G$ and $x,y$. Therefore, $x\neq y$. Suppose that $\{v_1,v_s\}\subseteq N_G(x)$. If $u_1,u_t\subseteq N_G(y)$, then $(\romannumeral7)$ holds for $G$ and $x,y$.  Otherwise, $(\romannumeral6)$ holds for $G$ and $x,y$. Therefore, $\{v_1,v_s\}\nsubseteq N_G(x)$. By symmetry, we may assume that $\{v_1,u_1\}\subseteq N_G(x)$ and $\{u_s,u_t\}\subseteq N_G(y)$. Note that $H$ is a series-parallel graph and there exists a path in $H-v$ whose ends lie in $C_1\setminus\{v\}$ and $C_2\setminus\{v\}$. This implies that $G$ contain a subdivision of $K_4$ as a subgraph, a contradiction.

Suppose that Lemma~\ref{lem:tree} $(\romannumeral5)$ holds. Then $H$ contains an induced theta subgraph $\theta$ such that $v$ is a branch vertex of $\theta$, and every vertex of $\theta$ other than the branch vertices has degree two in $H$. For convenience, we may assume that $w$ is a branch vertex of $\theta$ distinct from $v$, and that the three paths in $\theta$ with ends $v$ and $w$ are $w \ldots v_1 v$, $w \ldots v_2 v$, and $w \ldots v_3 v$. By the Pigeonhole Principle, we may assume that $\{v_1,v_2\} \subseteq N_G(x)$. Now, $(\romannumeral8)$ holds for $G$ and $x,y$.

Therefore, $H-v$ is not a forest. Now $H-v$ contains a cycle, and thus $G -\{x,y\}$ contains a cycle. By Lemma~\ref{lem:farcycle}, either $V(G) \setminus (N[x] \cup N[y])$ contains a vertex of degree at most one in $G$, or $G$ contains a cycle $C$ with $V(C) \cap \{x,y\} = \emptyset$ and such that all but at most two vertices in $V(C)$ have degree two in $G$. In the former case, $(\romannumeral2)$ holds for $G$ and $x,y$; in the latter case, $(\romannumeral4)$ holds for $G$ and $x,y$. 
\begin{case}
    $G$ is not series-parallel.
\end{case}

By Lemma~\ref{lem:ISK4-decomposition}, $G$ contains a wheel, and by Lemma~\ref{lem:proper-wheel-exist}, it contains a proper wheel $W$. Let $W=(C,s)$ be a proper wheel in $G$ of center $s$ (possibly $s \in \{x,y\}$) with minimum number of spokes. Let $Z = \{x,y\} \cap N(s)$. Since $x=y$ or $xy \in E(G)$, $x$ and $y$ are in the same component of $G \setminus (N[s] \setminus Z)$. Since $N(s)$ is stable, it follows that $|Z| \leq 1$. By Theorem~\ref{thm:proper-wheel-property}, the component of 
$G - (N[s]\setminus Z)$ containing $\{x,y\}$ includes the interiors of at most 
two sectors of $W$. Moreover, the interior of every other sector of $W$ lies in 
a distinct component of $G - (N[s]\setminus Z)$. Since $W$ has at least four sectors by Lemma~\ref{lem:nolink}, there is a component $K$ of $G - (N[s] \setminus Z)$ that does not intersect $N[x] \cup N[y]$. 
Let $N$ be the set of $N(s)$ with a neighbor in $K$ (see Figure~\ref{Figure-construction}).

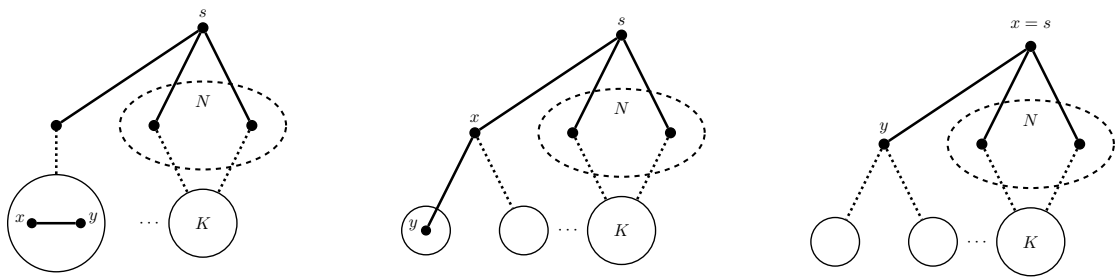
\begin{figure}[!ht]
\centering

\begin{minipage}{0.32\textwidth}
\centering
\scalebox{0.65}{
\begin{tikzpicture}[
    thick,
    every node/.style={circle, draw=black, fill=black, inner sep=2pt}
]

    \node[label={[above, yshift=-0.5mm]:{}}] (z) at (-3,5) {};
    \node[label={[above, yshift=-0.5mm]:{$s$}}] (s) at (0,7) {};
    \node (y1) at (-1,5) {};
    \node (w3) at (1,5) {};

    \node[draw=black, fill=white, inner sep=20pt] (H1) at (-3,3) {};
    \node[draw=black, fill=white, inner sep=10pt] (H3) at (0,3) {$K$};

    \draw[dashed, line width=1.2pt]
        (0,5) ellipse (1.7cm and 0.9cm);
    \node[draw=none, fill=none] at (0,5.5) {$N$};

    \node[circle, draw=black, fill=black, inner sep=1.8pt,
          label={[left]:{$x$}}] (x) at (-3.5,3) {};
          \node[circle, draw=black, fill=black, inner sep=1.8pt,
          label={[right]:{$y$}}] (y) at (-2.5,3) {};

    \node[draw=none, fill=none] at (-1.1,3) {$\cdots$};

    \draw[line width=1.5pt] (s) -- (z);
    \draw[line width=1.5pt] (x) -- (y);
   \draw[dotted,line width=1.5pt] (H1) -- (z);
    \draw[line width=1.5pt] (s) -- (w3);
     \draw[line width=1.5pt] (s) -- (y1);
    \draw[dotted,line width=1.5pt] (y1) -- (H3);
    \draw[dotted,line width=1.5pt] (w3) -- (H3);

\end{tikzpicture}}
\end{minipage}
\hfill
\begin{minipage}{0.32\textwidth}
\centering
\scalebox{0.65}{\begin{tikzpicture}[thick,every node/.style={circle,draw=black,fill=black,inner sep=2pt}]
\node[label={[above, yshift=-0.5mm]:{$x$}}] (x) at (-3,5) {};
\node[label={[above, yshift=-0.5mm]:{$s$}}] (s) at (0,7) {};
\node (y1) at (-1,5) {};
\node (w3) at (1,5) {};
\node[draw=black, fill=white, inner sep=10pt] (H1) at (-4,3) {};
\node[draw=black, fill=white, inner sep=10pt] (H2) at (-2,3) {};
\node[draw=black, fill=white, inner sep=10pt] (H3) at (0,3) {$K$};
\draw[dashed, line width=1.2pt] (0,5) ellipse (1.7cm and 0.9cm);
\node[draw=none, fill=none] at (0,5.5) {$N$};
\node[circle, draw=black, fill=black, inner sep=1.8pt,label={[left]:{$y$}}] (y) at (-4,3) {};
\node[draw=none, fill=none] at (-1.1,3) {$\cdots$};
\draw[line width=1.5pt] (s) -- (x);
\draw[line width=1.5pt] (x) -- (y);
\draw[line width=1.5pt] (s) -- (y1);
\draw[line width=1.5pt] (s) -- (w3);
\draw[dotted,line width=1.5pt] (y1) -- (H3);
\draw[dotted,line width=1.5pt] (w3) -- (H3);
\draw[dotted, line width=1.5pt] (x) -- (H2);
\end{tikzpicture}}
\end{minipage}
\hfill
\begin{minipage}{0.32\textwidth}
\centering
\scalebox{0.65}{\begin{tikzpicture}[thick,every node/.style={circle,draw=black,fill=black,inner sep=2pt}]
\node[label={[above, yshift=-0.5mm]:{$y$}}] (x) at (-3,5) {};
\node[label={[above, yshift=-1.8mm]:{$x=s$}}] (s) at (0,7) {};
\node (y1) at (-1,5) {};
\node (w3) at (1,5) {};
\node[draw=black, fill=white, inner sep=10pt] (H1) at (-4,3) {};
\node[draw=black, fill=white, inner sep=10pt] (H2) at (-2,3) {};
\node[draw=black, fill=white, inner sep=10pt] (H3) at (0,3) {$K$};
\draw[dashed, line width=1.2pt] (0,5) ellipse (1.7cm and 0.9cm);
\node[draw=none, fill=none] at (0,5.5) {$N$};
\node[draw=none, fill=none] at (-1.1,3) {$\cdots$};
\draw[line width=1.5pt] (s) -- (x);
\draw[dotted, line width=1.5pt] (x) -- (H1);
\draw[line width=1.5pt] (s) -- (y1);
\draw[line width=1.5pt] (s) -- (w3);
\draw[dotted, line width=1.5pt] (y1) -- (H3);
\draw[dotted, line width=1.5pt] (w3) -- (H3);
\draw[dotted, line width=1.5pt] (x) -- (H2);
\end{tikzpicture}}
\end{minipage}

\caption{Illustrations of possible positions of $x$ and $y$.}

\label{Figure-construction}
\end{figure}

 Now, we apply induction to $H=G[K \cup N \cup \{s\}]$ and $s$. 
By the choice of $H$ and $s$, $(\romannumeral 1)$ and $(\romannumeral 7)$ do not hold. Clearly, if $(\romannumeral2)$ holds for $H$ and $s$, then it holds for $G$ and $x,y$ as well. Note that  $(N[x]\cup N[y])\cap V(K)=\emptyset$.
If $(\romannumeral3)$ holds for $H$ and $s$, then $(\romannumeral3)$ (when $s\in \{x,y\}$) or $(\romannumeral5)$ (when $s\notin \{x,y\}$) holds for $G$ and $x,y$ (see Figure~\ref{Figure-H-z-3}).

\begin{figure}[!ht]
\centering

\resizebox{0.6\linewidth}{!}{%
\begin{tikzpicture}[scale=0.65,
  line cap=round,
  line join=round,
  edge/.style={line width=1.5, draw=black},
  bedge/.style={edge, draw=blue!85!black},
  dottededge/.style={edge, dotted},
  Ndash/.style={draw=black, line width=0.75, dash pattern=on 4.5pt off 4.5pt},
  Kcircle/.style={edge},
  bdotarc/.style={line width=1.5, draw=blue!85!black, dash pattern=on 1.69pt off 2.76pt},
  redpt/.style={draw=red, fill=red},
  blkpt/.style={draw=black, fill=black},
  hollow/.style={draw=black, fill=white, line width=1.5pt},
  lab/.style={} 
]

\def\PointR{4}

\begin{scope}
  \coordinate (S)  at (0,4.5);
  \coordinate (Y)  at (-5,2);
  \coordinate (Y1) at (-7,-2);
  \coordinate (Y2) at (-3,-2);

  \coordinate (Nc) at (0,2);
  \coordinate (Kc) at (0,-2);

  \coordinate (Lr) at (-1,2);
  \coordinate (Rr) at ( 1,2);

  \coordinate (Lb) at (-1,-1);
  \coordinate (Rb) at ( 1,-1);

  \draw[edge] (Y) -- (S);

  \draw[dottededge] (Y) -- (Y1);
  \draw[dottededge] (Y) -- (Y2);

  \path[hollow] (Y1) circle[radius=\PointR pt];
  \path[hollow] (Y2) circle[radius=\PointR pt];

  \draw[Ndash] (Nc) ellipse (3.0 and 1.05);

  \draw[Kcircle] (Kc) circle (2);

  \draw[bedge] (S) -- (Lr) -- (Lb);
  \draw[bedge] (S) -- (Rr) -- (Rb);

  \coordinate (M) at (0,-3);
  \draw[bdotarc]
    (Lb) .. controls (-1,0) and (-1,-3) .. (M)
    (M)  .. controls ( 1,-3) and ( 1,0) .. (Rb);

  \path[redpt] (S)  circle[radius=\PointR pt];
  \path[blkpt] (Y)  circle[radius=\PointR pt];
  \path[redpt] (Lr) circle[radius=\PointR pt];
  \path[redpt] (Rr) circle[radius=\PointR pt];
  \path[blkpt] (Lb) circle[radius=\PointR pt];
  \path[blkpt] (Rb) circle[radius=\PointR pt];

  \node[lab, above] at (S) {$x=s$};
  \node[lab, left]  at (Y) {$y$};
  \node[lab, right] at (1,2.45) {$N$};
  \node[lab] at (0,-2.6) {$K$};
\end{scope}

\begin{scope}[xshift=13cm]
  \coordinate (S)  at (0,4.5);
  \coordinate (X)  at (-5,2);
  \coordinate (Y)  at (-6,-2);
  \coordinate (H)  at (-3,-2);

  \coordinate (Nc) at (0,2);
  \coordinate (Kc) at (0,-2);

  \coordinate (Lr) at (-1,2);
  \coordinate (Rr) at ( 1,2);
  \coordinate (Lb) at (-1,-1);
  \coordinate (Rb) at ( 1,-1);

  \draw[edge] (X) -- (S);
  \draw[edge] (X) -- (Y);
  \draw[dottededge] (X) -- (H);

  \path[hollow] (H) circle[radius=\PointR pt];

  \draw[Kcircle] (Y) circle (1.1);

  \draw[Ndash] (Nc) ellipse (3.0 and 1.05);

  \draw[Kcircle] (Kc) circle (2);

  \draw[bedge] (S) -- (Lr) -- (Lb);
  \draw[bedge] (S) -- (Rr) -- (Rb);

  \coordinate (M) at (0,-3);
  \draw[bdotarc]
    (Lb) .. controls (-1,0) and (-1,-3) .. (M)
    (M)  .. controls ( 1,-3) and ( 1,0) .. (Rb);

  \path[redpt] (S)  circle[radius=\PointR pt];
  \path[blkpt] (X)  circle[radius=\PointR pt];
  \path[blkpt] (Y)  circle[radius=\PointR pt];
  \path[redpt] (Lr) circle[radius=\PointR pt];
  \path[redpt] (Rr) circle[radius=\PointR pt];
  \path[blkpt] (Lb) circle[radius=\PointR pt];
  \path[blkpt] (Rb) circle[radius=\PointR pt];

  \node[lab, above] at (S) {$s$};
  \node[lab, left]  at (X) {$x$};
  \node[lab, left]  at (Y) {$y$};
  \node[lab, right] at (1,2.45) {$N$};
  \node[lab] at (0,-2.6) {$K$};
\end{scope}

\end{tikzpicture}%
}
\caption{Illustrations of the case when $(\romannumeral 3)$ holds for $H$ and $s$.}
\label{Figure-H-z-3}
\end{figure}

If $(\romannumeral4)$ holds for $H$ and $s$, then $(\romannumeral4)$ holds for $G$ and $x,y$ (see Figure~\ref{Figure-H-z-4}). 

\begin{figure}[!ht]
\centering
\begingroup

\begin{tikzpicture}[x=0.75pt,y=0.75pt,yscale=-1,xscale=1,scale=0.7]
\tikzset{
  figY/edge/.style     ={line width=1.5, draw=black},
  figY/blueedge/.style ={figY/edge, draw=blue!85!black},
  figY/dash/.style     ={line width=1.5, draw=black, dash pattern=on 1.69pt off 2.76pt},
  figY/N/.style        ={draw=black, line width=0.75, dash pattern=on 4.5pt off 4.5pt},
  figY/arcblue/.style  ={line width=1.5, draw=blue!85!black, dash pattern=on 1.69pt off 2.76pt},
  figY/v/.style        ={draw=black, fill=black},
  figY/vred/.style     ={draw=red,  fill=red},
}

\def\PointR{4}
\def\KRad{60}

\begin{scope}
  \coordinate (s1)  at (111,46);
  \coordinate (p1)  at (95,87);
  \coordinate (u1)  at (63,166);
  \coordinate (v1)  at (156,162);

  \coordinate (a1)  at (85,223);
  \coordinate (b1)  at (136,224);

  \coordinate (c1)  at (78,263);
  \coordinate (d1)  at (149,264);

  \coordinate (N1) at (109,158);
  \coordinate (K1) at (113,257);

  \draw[figY/edge] (s1)--(p1)--(u1);
  \draw[figY/edge] (s1)--(v1);

  \draw[figY/N] (N1) ellipse (77 and 33);

  \draw[figY/edge] (K1) circle (\KRad);

  \draw[figY/blueedge] (a1)--(u1);
  \draw[figY/blueedge] (b1)--(v1);
  \draw[figY/blueedge] (c1)--(u1);
  \draw[figY/blueedge] (d1)--(v1);

  \draw[figY/arcblue]
    (136,224) .. controls (134,236) and (124,246) .. (111,246)
               .. controls (98,246)  and (88,236) .. (85,223);

  \draw[figY/arcblue]
    (149,264) .. controls (144,276) and (130,284) .. (113,284)
               .. controls (96,284)  and (82,275) .. (78,263);

  \path[figY/v] (s1) circle[radius=\PointR];
  \path[figY/v] (a1) circle[radius=\PointR];
  \path[figY/v] (b1) circle[radius=\PointR];
  \path[figY/v] (c1) circle[radius=\PointR];
  \path[figY/v] (d1) circle[radius=\PointR];

  \path[figY/vred] (u1) circle[radius=\PointR];
  \path[figY/vred] (v1) circle[radius=\PointR];

  \node at (126,164) {$N$};

  \node at (113,300) {$K$};

  \node[inner sep=0pt, label={[above, yshift=0.3mm]:{$s$}}] at (s1) {};
\end{scope}

\begin{scope}
  \coordinate (s2)  at (353,47);

  \coordinate (N2) at (351,159);
  \coordinate (K2) at (355,258);

  \coordinate (u2) at (310,157);
  \coordinate (v2) at (391,155);

  \coordinate (r2) at (351,159);
  \coordinate (L2) at (325,228);
  \coordinate (R2) at (386,228);
  \coordinate (t2) at (356,284);

  \draw[figY/N] (N2) ellipse (77 and 33);

  \draw[figY/edge] (K2) circle (\KRad);

  \draw[figY/dash] (s2)--(u2);
  \draw[figY/dash] (s2)--(v2);

  \draw[figY/edge] (s2)--(r2);

  \draw[figY/blueedge] (L2)--(r2);
  \draw[figY/blueedge] (R2)--(r2);

  \draw[figY/arcblue]
    (387,231) .. controls (390,236) and (391,242) .. (391,248)
               .. controls (391,268) and (375,284) .. (355,284)
               .. controls (336,284) and (320,268) .. (320,248)
               .. controls (320,241) and (321,235) .. (324,230);

  \path[figY/v] (s2) circle[radius=\PointR];
  \path[figY/v] (L2) circle[radius=\PointR];
  \path[figY/v] (R2) circle[radius=\PointR];

  \path[figY/vred] (r2) circle[radius=\PointR];
  \path[figY/vred] (t2) circle[radius=\PointR];

  \node at (373,163) {$N$};

  \node at (355,300) {$K$};

  \node[inner sep=0pt, label={[above, yshift=0.3mm]:{$s$}}] at (s2) {};
\end{scope}

\begin{scope}
  \coordinate (s3)  at (603,50);

  \coordinate (N3) at (601,162);
  \coordinate (K3) at (605,261);

  \coordinate (u3) at (601,162);
  \coordinate (l3) at (561,160);
  \coordinate (r3) at (641,159);

  \coordinate (E3)  at (607,257);
  \coordinate (eL)  at (572,257);
  \coordinate (eR)  at (642,257);

  \draw[figY/N] (N3) ellipse (77 and 33);

  \draw[figY/edge] (K3) circle (\KRad);

  \draw[figY/dash] (s3)--(u3);
  \draw[figY/dash] (s3)--(l3);
  \draw[figY/dash] (s3)--(r3);

  \draw[figY/arcblue] (E3) ellipse (35 and 20);

  \path[figY/v] (s3) circle[radius=\PointR];
  \path[figY/vred] (eL) circle[radius=\PointR];
  \path[figY/vred] (eR) circle[radius=\PointR];

  \node at (622,166) {$N$};

  \node at (605,300) {$K$};

  \node[inner sep=0pt, label={[above, yshift=0.3mm]:{$s$}}] at (s3) {};
\end{scope}

\end{tikzpicture}

\endgroup
\caption{Illustrations of the case when $(\romannumeral4)$ holds for $H$ and $s$.}
\label{Figure-H-z-4}
\end{figure}

If $(\romannumeral5)$ holds for $H$ and $s$, then $(\romannumeral5)$ holds for $G$ and $x,y$ (see Figure~\ref{Figure-H-z-5}).

\begin{figure}[!ht]
\centering
\begingroup

\begin{tikzpicture}[x=0.75pt,y=0.75pt,yscale=-1,xscale=1,scale=0.75]
\tikzset{
  figY/edge/.style     ={line width=1.5, draw=black},
  figY/blueedge/.style ={figY/edge, draw=blue!85!black},
  figY/dash/.style     ={line width=1.5, draw=black, dash pattern=on 1.69pt off 2.76pt},
  figY/N/.style        ={draw=black, line width=0.75, dash pattern=on 4.5pt off 4.5pt},
  figY/arcblue/.style  ={line width=1.5, draw=blue!85!black, dash pattern=on 1.69pt off 2.76pt},
  figY/v/.style        ={draw=black, fill=black},
  figY/vred/.style     ={draw=red,  fill=red},
}

\def\PointR{4}

\begin{scope}
  \coordinate (s1) at (111,46);

  \coordinate (p1) at (95,87);
  \coordinate (u1) at (63,166);     
  \coordinate (v1) at (157,162);    

  \coordinate (r1) at (112,227);    

  \coordinate (k1a) at (77,261);    
  \coordinate (k1b) at (149,264);   

  \coordinate (N1) at (109,158);
  \coordinate (K1) at (113,257);

  \draw[figY/edge] (s1)--(p1)--(u1);
  \draw[figY/edge] (s1)--(v1);

  \draw[figY/dash] (111,49)--(109,162);

  \draw[figY/N] (N1) ellipse (77 and 33);

  \draw[figY/edge] (K1) circle (51);

  \draw[figY/blueedge] (r1)--(u1);
  \draw[figY/blueedge] (r1)--(v1);
  \draw[figY/blueedge] (k1a)--(u1);
  \draw[figY/blueedge] (k1b)--(v1);

  \draw[figY/arcblue]
    (149,264) .. controls (144,276) and (130,284) .. (113,284)
               .. controls ( 96,284) and ( 82,275) .. ( 78,263);

  \path[figY/v]    (s1)  circle[radius=\PointR];
  \path[figY/vred] (u1)  circle[radius=\PointR];
  \path[figY/vred] (v1)  circle[radius=\PointR];
  \path[figY/vred] (r1)  circle[radius=\PointR];
  \path[figY/v]    (k1a) circle[radius=\PointR];
  \path[figY/v]    (k1b) circle[radius=\PointR];

  \node at (126,164) {$N$};

  \node at (113,295) {$K$};

  \node[inner sep=0pt, label={[right,xshift=0.4mm,yshift=-0.4mm]:{$r$}}] at (r1) {};

  \node[inner sep=0pt, label={[above, yshift=0.3mm]:{$s$}}] at (s1) {};
\end{scope}

\begin{scope}
  \coordinate (s2) at (277,46);

  \coordinate (r2) at (275,158);   
  \coordinate (a2) at (248,227);   
  \coordinate (b2) at (310,227);   
  \coordinate (c2) at (314,252);   

  \coordinate (N2) at (275,158);
  \coordinate (K2) at (279,258);

  \draw[figY/dash] (277,46)--(234,157);
  \draw[figY/dash] (279,47)--(315,154);

  \draw[figY/edge] (s2)--(r2);

  \draw[figY/N] (N2) ellipse (77 and 33);
  \draw[figY/edge] (K2) circle (51);

  \draw[figY/blueedge] (b2)--(r2);
  \draw[figY/blueedge] (a2)--(r2);

  \draw[figY/blueedge] (c2)--(b2);

  \draw[figY/arcblue]
    (314,254) .. controls (311,271) and (297,283) .. (279,283)
               .. controls (260,283) and (244,267) .. (244,247)
               .. controls (244,241) and (245,235) .. (248,229);

  \path[figY/v]    (s2) circle[radius=\PointR];
  \path[figY/vred] (r2) circle[radius=\PointR];
  \path[figY/v]    (a2) circle[radius=\PointR];
  \path[figY/vred] (b2) circle[radius=\PointR];
  \path[figY/vred] (c2) circle[radius=\PointR];

  \node at (297,163) {$N$};
  \node at (279,295) {$K$};
  \node[inner sep=0pt, label={[right,xshift=-4.0mm,yshift=-0.4mm]:{$r$}}] at (b2) {};
  \node[inner sep=0pt, label={[above, yshift=0.3mm]:{$s$}}] at (s2) {};
\end{scope}

\begin{scope}
  \coordinate (s3) at (446,48);

  \coordinate (r3) at (445,160);   
  \coordinate (a3) at (418,228);   
  \coordinate (b3) at (480,229);   

  \coordinate (N3) at (445,160);
  \coordinate (K3) at (448,259);

  \draw[figY/dash] (446,48)--(404,158);
  \draw[figY/dash] (448,49)--(484,156);

  \draw[figY/edge] (s3)--(r3);

  \draw[figY/N] (N3) ellipse (77 and 33);
  \draw[figY/edge] (K3) circle (51);

  \draw[figY/blueedge] (b3)--(r3);
  \draw[figY/blueedge] (a3)--(r3);

  \draw[figY/arcblue]
    (481,233) .. controls (483,238) and (484,243) .. (484,249)
               .. controls (484,269) and (468,285) .. (449,285)
               .. controls (429,285) and (413,269) .. (413,249)
               .. controls (413,242) and (415,236) .. (418,231);

  \path[figY/v]    (s3) circle[radius=\PointR];
  \path[figY/vred] (r3) circle[radius=\PointR];
  \path[figY/vred]    (a3) circle[radius=\PointR];
  \path[figY/vred] (b3) circle[radius=\PointR];

  \node at (473,163) {$N$};
  \node at (448,305) {$K$};
  \node at (446.5,180) {$r$};
  \node[inner sep=0pt, label={[above, yshift=0.3mm]:{$s$}}] at (s3) {};
\end{scope}

\begin{scope}
  \coordinate (s4) at (615,48);

  \coordinate (r4) at (615,216);   
  \coordinate (a4) at (581,239);   
  \coordinate (b4) at (652,239);   

  \coordinate (N4) at (613,160);
  \coordinate (K4) at (617,259);

  \draw[figY/N] (N4) ellipse (77 and 33);
  \draw[figY/edge] (K4) circle (51);

  \draw[figY/dash] (615,48)--(613,160);
  \draw[figY/dash] (615,48)--(573,158);
  \draw[figY/dash] (615,51)--(653,157);

  \draw[figY/blueedge] (a4)--(r4);
  \draw[figY/blueedge] (b4)--(r4);

  \draw[figY/arcblue]
    (652,240) .. controls (653,244) and (653,247) .. (653,251)
               .. controls (653,271) and (637,287) .. (616,287)
               .. controls (596,287) and (579,271) .. (579,251)
               .. controls (579,247) and (580,243) .. (581,239);

  \path[figY/v]    (s4) circle[radius=\PointR];
  \path[figY/vred] (r4) circle[radius=\PointR];
  \path[figY/vred] (a4) circle[radius=\PointR];
  \path[figY/vred] (b4) circle[radius=\PointR];

  \node at (644,164) {$N$};
  \node at (617,306) {$K$};
  \node[inner sep=0pt, label={[right,xshift=-2.0mm,yshift=-2.8mm]:{$r$}}] at (r4) {};
  \node[inner sep=0pt, label={[above, yshift=0.3mm]:{$s$}}] at (s4) {};
\end{scope}

\end{tikzpicture}

\endgroup
\caption{Illustrations of the case when $(\romannumeral5)$ holds for $H$ and $s$.}
\label{Figure-H-z-5}
\end{figure}

If $(\romannumeral6)$ holds for $H$ and $s$, then $(\romannumeral6)$ (when $s\in \{x,y\}$) holds for $G$ and $x,y$ or $(\romannumeral10)$ (when $s\notin \{x,y\}$) holds for $G$ and $x,y$ (see Figure~\ref{triangle-6}). 
\begin{figure}[!ht]
\centering
\begingroup

\begin{tikzpicture}[x=0.5pt,y=0.5pt,yscale=-1,xscale=1,scale=0.65]
\tikzset{
  fig/edge/.style     ={line width=1.5, draw=black},
  fig/blueedge/.style ={fig/edge, draw={rgb,255:red,0;green,23;blue,255}},
  fig/dashed/.style   ={fig/edge, dash pattern=on 1.69pt off 2.76pt},
  fig/Ndash/.style    ={line width=0.75, dash pattern=on 4.5pt off 4.5pt, draw=black},
  fig/redv/.style     ={draw={rgb,255:red,208;green,2;blue,27}, fill={rgb,255:red,208;green,2;blue,27}},
  fig/blackv/.style   ={draw=black, fill=black},
}

\def\PointR{3.6}

\draw[fig/edge] (195.33,34.58) -- (280.56,147.44);
\draw[fig/edge] (195.33,35.92) -- (91.89,150.99);
\draw[fig/edge] (195.33,35.92) -- (141.22,148.78);
\draw[fig/edge] (195.33,35.92) -- (186.56,148.78);
\draw[fig/edge] (195.33,34.58) -- (235.89,147.44);

\draw[fig/Ndash]
  (63.89,152.88) .. controls (63.89,134.44) and (118.76,119.5) .. (186.44,119.5)
  .. controls (254.13,119.5) and (309,134.44) .. (309,152.88)
  .. controls (309,171.31) and (254.13,186.25) .. (186.44,186.25)
  .. controls (118.76,186.25) and (63.89,171.31) .. (63.89,152.88) -- cycle;

\draw[fig/edge]
  (85.83,314.25) .. controls (85.83,255.39) and (133.55,207.67) .. (192.42,207.67)
  .. controls (251.28,207.67) and (299,255.39) .. (299,314.25)
  .. controls (299,373.11) and (251.28,420.83) .. (192.42,420.83)
  .. controls (133.55,420.83) and (85.83,373.11) .. (85.83,314.25) -- cycle;

\draw[fig/blueedge] (119,269.5) -- (91.89,150.99);
\draw[fig/blueedge] (163,255.5) -- (141.22,148.78);
\draw[fig/blueedge] (213,330.5) -- (186.56,148.78);
\draw[fig/blueedge] (286,336.5) -- (280.56,147.44);

\draw[fig/dashed] (235.89,147.44) -- (286,336.5);
\draw[fig/dashed] (213,330.5) -- (223.65,245.33) -- (235.89,147.44);
\draw[fig/dashed] (119,269.5) -- (154,336.5);
\draw[fig/dashed] (163,255.5) -- (154,336.5);

\path[fig/redv] (91.89,150.99) circle[radius=\PointR];
\path[fig/redv] (141.22,148.78) circle[radius=\PointR];
\path[fig/redv] (186.44,149.08) circle[radius=\PointR];
\path[fig/redv] (235.89,147.44) circle[radius=\PointR];
\path[fig/redv] (195.33,35.92) circle[radius=\PointR];
\path[fig/redv] (280.56,147.44) circle[radius=\PointR];
\path[fig/redv] (154,332.71) circle[radius=\PointR];

\path[fig/blackv] (119,269.5) circle[radius=\PointR];
\path[fig/blackv] (163,255.5) circle[radius=\PointR];
\path[fig/blackv] (213,330.5) circle[radius=\PointR];
\path[fig/blackv] (286,336.5) circle[radius=\PointR];

\node[inner sep=0pt, label={[above]:{$s=x$}}] at (195.33,35.92) {};
\node at (255,158.8) {$N$};
\node at (255,365) {$K$};

\draw[fig/edge] (513.33,32.75) -- (598.56,145.61);
\draw[fig/edge] (513.33,34.08) -- (409.89,149.15);
\draw[fig/edge] (513.33,34.08) -- (459.22,146.94);
\draw[fig/edge] (513.33,34.08) -- (504.56,146.94);
\draw[fig/edge] (513.33,32.75) -- (553.89,145.61);

\draw[fig/Ndash]
  (381.89,151.04) .. controls (381.89,132.61) and (436.76,117.67) .. (504.44,117.67)
  .. controls (572.13,117.67) and (627,132.61) .. (627,151.04)
  .. controls (627,169.47) and (572.13,184.42) .. (504.44,184.42)
  .. controls (436.76,184.42) and (381.89,169.47) .. (381.89,151.04) -- cycle;

\draw[fig/edge]
  (403.83,312.42) .. controls (403.83,253.55) and (451.55,205.83) .. (510.42,205.83)
  .. controls (569.28,205.83) and (617,253.55) .. (617,312.42)
  .. controls (617,371.28) and (569.28,419) .. (510.42,419)
  .. controls (451.55,419) and (403.83,371.28) .. (403.83,312.42) -- cycle;

\draw[fig/blueedge] (437,267.67) -- (409.89,149.15);
\draw[fig/blueedge] (481,253.67) -- (459.22,146.94);
\draw[fig/blueedge] (531,328.67) -- (504.56,146.94);
\draw[fig/blueedge] (604,334.67) -- (598.56,145.61);

\draw[fig/dashed] (553.89,145.61) -- (604,334.67);
\draw[fig/dashed] (531,328.67) -- (553.89,145.61);
\draw[fig/dashed] (437,267.67) -- (472,334.67);
\draw[fig/dashed] (481,253.67) -- (472,334.67);

\path[fig/redv] (409.89,149.15) circle[radius=\PointR];
\path[fig/redv] (459.22,146.94) circle[radius=\PointR];
\path[fig/redv] (504.44,147.25) circle[radius=\PointR];
\path[fig/redv] (553.89,145.61) circle[radius=\PointR];
\path[fig/redv] (513.33,34.08) circle[radius=\PointR];
\path[fig/redv] (598.56,145.61) circle[radius=\PointR];
\path[fig/redv] (472,330.87) circle[radius=\PointR];

\path[fig/blackv] (437,267.67) circle[radius=\PointR];
\path[fig/blackv] (481,253.67) circle[radius=\PointR];
\path[fig/blackv] (531,324.87) circle[radius=\PointR];
\path[fig/blackv] (604,334.67) circle[radius=\PointR];

\node[inner sep=0pt, label={[above]:{$s$}}] at (513.33,34.08) {};
\node at (575,156.97) {$N$};
\node at (575,365) {$K$};

\end{tikzpicture}

\endgroup
\caption{Illustrations of the case when $(\romannumeral6)$ holds for $H$ and $s$.}
\label{triangle-6}
\end{figure}

If $(\romannumeral8)$ holds for $H$ and $s$, then $(\romannumeral8)$ (when $s\in \{x,y\}$) holds for $G$ and $x,y$ or $(\romannumeral9)$ (when $s\notin \{x,y\}$) holds for $G$ and $x,y$ (see Figure~\ref{triangle-8}). 

\begin{figure}[!ht]
\centering

\begin{tikzpicture}[x=0.5pt,y=0.5pt,yscale=-1,xscale=1,scale=0.65]
\tikzset{
  fig/edge/.style   ={draw=black, line width=1.5},
  fig/bedge/.style  ={draw=blue!85!black, line width=1.5},
  fig/dash/.style   ={draw=black, line width=1.5, dash pattern=on 1.69pt off 2.76pt},
  fig/Ndash/.style  ={draw=black, line width=0.75, dash pattern=on 4.5pt off 4.5pt},
  fig/Kcircle/.style={draw=black, line width=1.5},
  fig/v/.style      ={circle, draw=black, fill=black, inner sep=1.6pt},
  fig/vred/.style   ={circle, draw=red,   fill=red,   inner sep=2.0pt},
  fig/lab/.style    ={font=\small},
}

\coordinate (sL)   at (185,38);
\coordinate (aL)   at (91.89,150.99);
\coordinate (bL)   at (185,147.33);
\coordinate (cL)   at (280.56,147.44);

\coordinate (pL)   at (119,269.5);
\coordinate (qL)   at (185,269.33);
\coordinate (rL)   at (255.67,274);

\coordinate (tL)   at (185,358.5); 

\coordinate (NcL)  at (186.44,152.88);
\coordinate (KcL)  at (192.42,314.25);

\coordinate (NL) at (249.4,158.8);
\coordinate (KL) at (243.5,345);

\draw[fig/Ndash]   (NcL) ellipse (122.56 and 33.38);
\draw[fig/Kcircle] (KcL) circle (106.59);

\draw[fig/edge] (sL)--(aL);
\draw[fig/edge] (sL)--(bL);
\draw[fig/edge] (sL)--(cL);

\draw[fig/bedge] (pL)--(aL);
\draw[fig/bedge] (qL)--(bL);
\draw[fig/bedge] (rL)--(cL);

\draw[fig/dash] (pL)--(tL);
\draw[fig/dash] (qL)--(tL);
\draw[fig/dash] (rL)--(tL);

\node[fig/vred] at (sL) {};
\node[fig/vred] at (aL) {};
\node[fig/vred] at (bL) {};
\node[fig/vred] at (cL) {};
\node[fig/vred] at (tL) {};

\node[fig/v]    at (pL) {};
\node[fig/v]    at (qL) {};
\node[fig/v]    at (rL) {};

\node[fig/lab] at (NL) {$N$};
\node[fig/lab] at (KL) {$K$};
\node[fig/lab, above] at (sL) {$s=x$};

\coordinate (sR)   at (489.67,36.83);
\coordinate (aR)   at (396.56,149.82);
\coordinate (bR)   at (489.67,146.17);
\coordinate (cR)   at (585.22,146.28);

\coordinate (pR)   at (423.67,268.33);
\coordinate (qR)   at (489.67,268.17);
\coordinate (rR)   at (560.33,272.83);

\coordinate (tR)   at (489.67,357.33); 

\coordinate (NcR)  at (491.11,151.71);
\coordinate (KcR)  at (497.08,313.08);

\coordinate (NR) at (554.07,157.63);
\coordinate (KR) at (553.17,337.83);

\draw[fig/Ndash]   (NcR) ellipse (122.56 and 33.38);
\draw[fig/Kcircle] (KcR) circle (106.59);

\draw[fig/edge] (sR)--(aR);
\draw[fig/edge] (sR)--(bR);
\draw[fig/edge] (sR)--(cR);

\draw[fig/bedge] (pR)--(aR);
\draw[fig/bedge] (qR)--(bR);
\draw[fig/bedge] (rR)--(cR);

\draw[fig/dash] (pR)--(tR);
\draw[fig/dash] (qR)--(tR);
\draw[fig/dash] (rR)--(tR);

\node[fig/vred] at (sR) {};
\node[fig/vred] at (aR) {};
\node[fig/vred] at (bR) {};
\node[fig/vred] at (cR) {};
\node[fig/vred] at (tR) {};

\node[fig/v]    at (pR) {};
\node[fig/v]    at (qR) {};
\node[fig/v]    at (rR) {};

\node[fig/lab] at (NR) {$N$};
\node[fig/lab] at (KR) {$K$};
\node[fig/lab, above] at (sR) {$s$};

\end{tikzpicture}

\caption{Illustrations of the case when $(\romannumeral8)$ holds for $H$ and $s$.}
\label{triangle-8}
\end{figure}

If $(\romannumeral9)$ holds for $H$ and $s$, then $(\romannumeral9)$ holds for $G$ and $x,y$ (see Figure~\ref{triangle-9}). 
\begin{figure}[!ht]
\centering

\begin{tikzpicture}[x=0.5pt,y=0.5pt,yscale=-1,xscale=1,scale=0.75]
\tikzset{
  fig/edge/.style   ={draw=black, line width=1.5},
  fig/bedge/.style  ={draw=blue!85!black, line width=1.5},
  fig/dash/.style   ={draw=black, line width=1.5, dash pattern=on 1.69pt off 2.76pt},
  fig/Ndash/.style  ={draw=black, line width=0.75, dash pattern=on 4.5pt off 4.5pt},
  fig/Kcircle/.style={draw=black, line width=1.5},
  fig/v/.style      ={circle, draw=black, fill=black, inner sep=1.6pt},
  fig/vred/.style   ={circle, draw=red,   fill=red,   inner sep=2.0pt},
  fig/lab/.style    ={font=\small},
}

\def\Krad{70}

\begin{scope}[xshift=0pt]
\coordinate (sL)   at (111,45.58);
\coordinate (NcL)  at (109.38,157.88);
\coordinate (KcL)  at (112.75,287.25); 

\coordinate (rL)   at (109.38,161.67);
\coordinate (rLb)  at (109.38,274.75); 
\coordinate (uL)   at (76.5,266.75);   
\coordinate (vL)   at (144,266.25);    
\coordinate (tL)   at (109.38,308.25); 

\coordinate (sLtoNL_left)  at (68.2,156.2);
\coordinate (sLtoNL_right) at (147,152.6);

\draw[fig/Ndash] (NcL) ellipse (76.63 and 33.38);
\draw[fig/Kcircle] (KcL) circle (\Krad);

\draw[fig/dash] (sL) -- (sLtoNL_left);
\draw[fig/dash] (sL) -- (NcL);
\draw[fig/dash] (sL) -- (sLtoNL_right);

\draw[fig/edge] (sL) -- (rL);

\draw[fig/bedge] (rLb) -- (rL);
\draw[fig/bedge] (uL)  -- (rL);
\draw[fig/bedge] (vL)  -- (rL);

\draw[fig/dash] (uL)  -- (tL);
\draw[fig/dash] (rLb) -- (tL);
\draw[fig/dash] (vL)  -- (tL);

\node[fig/v]    at (sL)  {};
\node[fig/vred] at (rL)  {};
\node[fig/vred] at (rLb) {};
\node[fig/vred] at (uL)  {};
\node[fig/vred] at (vL)  {};
\node[fig/v]    at (tL)  {};

\node[fig/lab] at (168,164) {$N$};
\node[fig/lab] at (125,330) {$K$};
\node[fig/lab] at (110,27) {$s$};
\node[fig/lab] at (90,163) {$r$};
\end{scope}

\begin{scope}[xshift=10pt] 
\coordinate (sM)   at (277,45.92);
\coordinate (NcM)  at (275.38,158.21);
\coordinate (KcM)  at (278.75,287.58); 

\coordinate (aM)   at (234.2,156.53);
\coordinate (bM)   at (314.6,154.13);
\coordinate (rM)   at (247.5,279.25);  
\coordinate (tM)   at (305,277.75);    
\coordinate (qM)   at (277,309.25);    

\draw[fig/Ndash] (NcM) ellipse (76.63 and 33.38);
\draw[fig/Kcircle] (KcM) circle (\Krad);

\draw[fig/edge] (sM) -- (aM);
\draw[fig/edge] (sM) -- (bM);

\draw[fig/dash] (sM) -- (NcM);

\draw[fig/bedge] (rM) -- (aM);
\draw[fig/bedge] (rM) -- (bM);
\draw[fig/bedge] (qM) -- (rM);

\draw[fig/dash] (aM) -- (tM);
\draw[fig/dash] (bM) -- (tM);
\draw[fig/dash] (qM) -- (tM);

\node[fig/v]    at (sM) {};
\node[fig/vred] at (aM) {};
\node[fig/vred] at (bM) {};
\node[fig/vred] at (rM) {};
\node[fig/vred] at (qM) {};
\node[fig/vred] at (tM) {};

\node[fig/lab] at (336,162.53) {$N$};
\node[fig/lab] at (292,330) {$K$};
\node[fig/lab] at (274,27) {$s$};
\node[fig/lab] at (230,282.5) {$r$}; 
\end{scope}

\begin{scope}[xshift=20pt]
\coordinate (sRA)   at (443.5,47.08);
\coordinate (NcRA)  at (441.88,159.38);
\coordinate (KcRA)  at (445.25,288.75); 

\coordinate (rRA)   at (441.88,163.17);
\coordinate (pRA)   at (441.88,276.25); 
\coordinate (uRA)   at (409,268.25);    
\coordinate (vRA)   at (476.5,267.75);  
\coordinate (tRA)   at (441.88,309.75); 

\coordinate (sRAtoN_left)  at (400.7,157.7);
\coordinate (sRAtoN_right) at (479.5,154.1);

\draw[fig/Ndash] (NcRA) ellipse (76.63 and 33.38);
\draw[fig/Kcircle] (KcRA) circle (\Krad);

\draw[fig/dash] (sRA) -- (sRAtoN_left);
\draw[fig/dash] (sRA) -- (NcRA);
\draw[fig/dash] (sRA) -- (sRAtoN_right);

\draw[fig/edge] (sRA) -- (rRA);

\draw[fig/dash] (pRA) -- (rRA);
\draw[fig/dash] (uRA) -- (rRA);
\draw[fig/dash] (vRA) -- (rRA);

\draw[fig/bedge] (uRA) -- (tRA);
\draw[fig/bedge] (pRA) -- (tRA);
\draw[fig/bedge] (vRA) -- (tRA);

\node[fig/v]    at (sRA) {};
\node[fig/vred] at (rRA) {};
\node[fig/vred] at (pRA) {};
\node[fig/vred] at (uRA) {};
\node[fig/vred] at (vRA) {};
\node[fig/vred] at (tRA) {};

\node[fig/lab] at (500,162.87) {$N$};
\node[fig/lab] at (458,330) {$K$};
\node[fig/lab] at (440.5,27) {$s$};
\node[fig/lab] at (422,310) {$r$}; 
\end{scope}

\begin{scope}[xshift=30pt]
\coordinate (sRB)   at (613.5,48.58);
\coordinate (NcRB)  at (611.88,160.88);
\coordinate (KcRB)  at (615.25,290.25); 

\coordinate (pRB)   at (617,253.42);    
\coordinate (uRB)   at (590.5,279.92);  
\coordinate (vRB)   at (619,279.92);    
\coordinate (wRB)   at (645,279.92);    
\coordinate (tRB)   at (617.33,306.75); 

\coordinate (sRBtoN_left)  at (570.7,159.2);
\coordinate (sRBtoN_right) at (651.1,156.8);

\draw[fig/Ndash] (NcRB) ellipse (76.63 and 33.38);
\draw[fig/Kcircle] (KcRB) circle (\Krad);

\draw[fig/dash] (sRB) -- (sRBtoN_left);
\draw[fig/dash] (sRB) -- (NcRB);
\draw[fig/dash] (sRB) -- (sRBtoN_right);

\draw[fig/bedge] (uRB) -- (pRB);
\draw[fig/bedge] (uRB) -- (vRB);
\draw[fig/bedge] (tRB) -- (uRB);

\draw[fig/dash] (vRB) -- (wRB);
\draw[fig/dash] (wRB) -- (tRB);
\draw[fig/dash] (wRB) -- (pRB);

\node[fig/v]    at (sRB) {};
\node[fig/vred] at (pRB) {};
\node[fig/vred] at (uRB) {};
\node[fig/vred] at (vRB) {};
\node[fig/vred] at (wRB) {};
\node[fig/vred] at (tRB) {};

\node[fig/lab] at (670,164.2) {$N$};
\node[fig/lab] at (628,330) {$K$};
\node[fig/lab] at (610.5,27) {$s$};
\node[fig/lab] at (575.04,280) {$r$}; 
\end{scope}

\end{tikzpicture}

\caption{Illustrations of the case when $(\romannumeral9)$ holds for $H$ and $s$.}
\label{triangle-9}
\end{figure}

If $(\romannumeral10)$ holds for $H$ and $s$, then $(\romannumeral10)$ holds for $G$ and $x,y$ (see Figure~\ref{triangle-x}).  This completes the proof of Theorem~\ref{main-theorem:triangle}. 
\begin{figure}[!ht]
\centering

\tikzset{every picture/.style={line width=0.75pt}}

\begin{tikzpicture}[x=0.5pt,y=0.5pt,yscale=-1,xscale=1,scale=0.85]
\tikzset{
  fig/edge/.style   ={draw=black, line width=1.5},
  fig/bedge/.style  ={draw=blue!85!black, line width=1.5},
  fig/dash/.style   ={draw=black, line width=1.5, dash pattern=on 1.69pt off 2.76pt},
  fig/Ndash/.style  ={draw=black, line width=0.75, dash pattern=on 4.5pt off 4.5pt},
  fig/v/.style      ={circle, draw=black, fill=black, inner sep=1.6pt},
  fig/vred/.style   ={circle, draw=red,   fill=red,   inner sep=2.0pt},
  fig/lab/.style    ={font=\small},
}

\def\Krad{84}   

\begin{scope}[xshift=0pt]
\coordinate (s1)   at (111,45.58);
\coordinate (Nc1)  at (109.38,157.88);
\coordinate (Kc1)  at (111.67,303);      

\coordinate (r1)   at (109.38,161.67);

\coordinate (a1)   at (59.33,260);
\coordinate (b1)   at (80.67,261.33);
\coordinate (c1)   at (140.67,258.67);
\coordinate (d1)   at (160.67,260);

\coordinate (p1)   at (78,311.33);
\coordinate (q1)   at (142,310.67);

\coordinate (s1L)  at (71.33,160);     
\coordinate (s1R)  at (152,163.33);    

\draw[fig/Ndash] (Nc1) ellipse (76.63 and 33.38);
\draw[fig/edge]  (Kc1) circle (\Krad);

\draw[fig/dash] (s1) -- (s1L);
\draw[fig/dash] (s1) -- (Nc1);
\draw[fig/dash] (s1) -- (s1R);

\draw[fig/edge] (s1) -- (r1);

\draw[fig/bedge] (a1) -- (r1);
\draw[fig/bedge] (b1) -- (r1);
\draw[fig/bedge] (c1) -- (r1);
\draw[fig/bedge] (d1) -- (r1);

\draw[fig/dash] (a1) -- (p1);
\draw[fig/dash] (b1) -- (p1);
\draw[fig/dash] (d1) -- (q1);
\draw[fig/dash] (c1) -- (q1);

\node[fig/v]    at (s1) {};
\node[fig/vred] at (r1) {};
\node[fig/vred] at (a1) {};
\node[fig/vred] at (b1) {};
\node[fig/vred] at (c1) {};
\node[fig/vred] at (d1) {};
\node[fig/vred] at (p1) {};
\node[fig/vred] at (q1) {};

\node[fig/lab] at (170,164) {$N$};
\node[fig/lab] at (125,352) {$K$};
\node[fig/lab] at (110,29) {$s$};
\node[fig/lab] at (93,160) {$r$};
\end{scope}

\begin{scope}[xshift=40pt] 
\coordinate (s2)   at (292.33,44.92);
\coordinate (Nc2)  at (290.71,157.21);
\coordinate (Kc2)  at (293,302.33);

\coordinate (u2)   at (252.67,159.33);
\coordinate (v2)   at (274,159.33);
\coordinate (r2)   at (281.33,239.33);
\coordinate (w2)   at (342,259.33);
\coordinate (x2)   at (322,270);

\coordinate (p2)   at (259.33,310.67);
\coordinate (q2)   at (323.33,310);

\coordinate (s2D)  at (333.33,162.67); 

\draw[fig/Ndash] (Nc2) ellipse (76.63 and 33.38);
\draw[fig/edge]  (Kc2) circle (\Krad);

\draw[fig/dash] (s2) -- (Nc2);
\draw[fig/dash] (s2) -- (s2D);

\draw[fig/edge] (u2) -- (s2);

\draw[fig/edge] (s2) -- (v2);

\draw[fig/bedge] (r2) -- (v2);
\draw[fig/bedge] (r2) -- (u2);
\draw[fig/bedge] (r2) -- (w2);
\draw[fig/bedge] (r2) -- (x2);

\draw[fig/dash] (u2) -- (p2);
\draw[fig/dash] (v2) -- (p2);
\draw[fig/dash] (w2) -- (q2);
\draw[fig/dash] (x2) -- (q2);

\node[fig/v]    at (s2) {};
\node[fig/vred] at (u2) {};
\node[fig/vred] at (v2) {};
\node[fig/vred] at (r2) {};
\node[fig/vred] at (w2) {};
\node[fig/vred] at (x2) {};
\node[fig/vred] at (p2) {};
\node[fig/vred] at (q2) {};

\node[fig/lab] at (350,163) {$N$};
\node[fig/lab] at (306,351.5) {$K$};
\node[fig/lab] at (289.33,29) {$s$};
\node[fig/lab] at (283,255) {$r$};
\end{scope}

\begin{scope}[xshift=80pt]
\coordinate (s3)   at (469.67,44.92);
\coordinate (Nc3)  at (468.04,157.21);
\coordinate (Kc3)  at (470.33,302.33);

\coordinate (u3)   at (430,159.33);
\coordinate (v3)   at (469.67,164);
\coordinate (w3)   at (510.67,162.67);
\coordinate (r3)   at (469.67,236);
\coordinate (p3)   at (436.67,310.67);
\coordinate (q3)   at (502.67,309.33);
\coordinate (t3)   at (532,310.67);

\draw[fig/Ndash] (Nc3) ellipse (76.63 and 33.38);
\draw[fig/edge]  (Kc3) circle (\Krad);

\coordinate (s3L) at (430,159.33);
\coordinate (s3R) at (510.67,162.67);
\draw[fig/dash] (s3) -- (s3L);
\draw[fig/dash] (s3) -- (Nc3);
\draw[fig/dash] (s3) -- (s3R);

\draw[fig/edge] (s3) -- (v3);

\draw[fig/edge] (s3) -- (w3);

\draw[fig/edge] (u3) -- (s3);

\draw[fig/bedge] (r3) -- (u3);
\draw[fig/bedge] (r3) -- (v3);
\draw[fig/bedge] (r3) -- (q3);
\draw[fig/bedge] (t3) -- (r3);

\draw[fig/dash] (u3) -- (p3);
\draw[fig/dash] (v3) -- (p3);
\draw[fig/dash] (w3) -- (t3);
\draw[fig/dash] (w3) -- (q3);

\node[fig/v]    at (s3) {};
\node[fig/vred] at (u3) {};
\node[fig/vred] at (v3) {};
\node[fig/vred] at (w3) {};
\node[fig/vred] at (r3) {};
\node[fig/vred] at (p3) {};
\node[fig/vred] at (q3) {};
\node[fig/vred] at (t3) {};

\node[fig/lab] at (530,163) {$N$};
\node[fig/lab] at (482,351.5) {$K$};
\node[fig/lab] at (466.67,29) {$s$};
\node[fig/lab] at (466,250) {$r$};
\end{scope}

\begin{scope}[xshift=120pt]
\coordinate (s4)   at (651.67,42.25);
\coordinate (Nc4)  at (650.04,154.54);
\coordinate (Kc4)  at (652.33,299.67);

\coordinate (r4)   at (651.67,233.33);
\coordinate (u4)   at (596.67,286.67);
\coordinate (v4)   at (628,287.33);
\coordinate (w4)   at (684,284.67);
\coordinate (x4)   at (716,285.33);
\coordinate (p4)   at (617.33,326);
\coordinate (q4)   at (702,323.33);

\draw[fig/Ndash] (Nc4) ellipse (76.63 and 33.38);
\draw[fig/edge]  (Kc4) circle (\Krad);

\coordinate (s4L) at (612,156.67);
\coordinate (s4R) at (692.67,160);
\draw[fig/dash] (s4) -- (s4L);
\draw[fig/dash] (s4) -- (Nc4);
\draw[fig/dash] (s4) -- (s4R);

\draw[fig/dash] (s4) -- (651.67,161.33);

\draw[fig/bedge] (r4) -- (u4);
\draw[fig/bedge] (r4) -- (v4);
\draw[fig/bedge] (r4) -- (w4);
\draw[fig/bedge] (x4) -- (r4);

\draw[fig/dash] (p4) -- (u4);
\draw[fig/dash] (p4) -- (v4);
\draw[fig/dash] (w4) -- (q4);
\draw[fig/dash] (x4) -- (q4);

\node[fig/v]    at (s4) {};
\node[fig/vred] at (r4) {};
\node[fig/vred] at (u4) {};
\node[fig/vred] at (v4) {};
\node[fig/vred] at (w4) {};
\node[fig/vred] at (x4) {};
\node[fig/vred] at (p4) {};
\node[fig/vred] at (q4) {};

\node[fig/lab] at (710,160.5) {$N$};
\node[fig/lab] at (665.5,349.5) {$K$};
\node[fig/lab] at (650,28.5) {$s$};
\node[fig/lab] at (652,250) {$r$};
\end{scope}

\end{tikzpicture}

\caption{Illustrations of the case when $(\romannumeral10)$ holds for $H$ and $s$.}
\label{triangle-x}
\end{figure}
\end{proof}

\section{Main result}
Let $G$ be a graph and let $x,y,z \in V(G)$ such that $\{x,y,z\}$ is a clique (possibly $x,y,z$ can be the same).
We say that $(x,y,z)$ is a \emph{noncentral triple} of $G$ if none of $x,y,z$ is the center of 
some proper wheel in $G$. 
We say that $(G,x,y,z)$ has property $\mathcal{P}$ if 
$V(G)\setminus (N[x]\cup N[y]\cup N[z])$ contains a vertex of degree at most two in $G$.
We prove the following.
\begin{theorem}\label{thm:property-p}
Let $G$ be an $\iskdbpk$-free graph without a clique cutset. Suppose that $G$ is not series-parallel, and let $(x,y,z)$ be a noncentral triple of $G$. Then $(G,x,y,z)$ has property $\mathcal{P}$.
\end{theorem}
\begin{proof}
Let $G$ be a counterexample to Theorem~\ref{thm:property-p} with $|V(G)|$ minimum.
Then $G$ is connected and every vertex in $V(G) \setminus (N[x] \cup N[y]\cup N[z])$ has degree at least three in $G$. 
Since $G$ is not series-parallel, and $G$ is $\iskdbpk$-free, it follows from Lemma~\ref{lem:ISK4-decomposition} that $G$ contains a wheel and hence by Lemma~\ref{lem:proper-wheel-exist} $G$ contains a proper wheel $W = (C,s)$. 
Let $C_1, \dots, C_k$ denote the components of $V(G) \setminus N[s]$. For $i\in[k]$, let $N_i = \{v \in N(s): N(v) \cap C_i \neq \emptyset \}$, and let  $G_i$ denote the induced subgraph of $G$ with vertex set $V(C_i) \cup N_i \cup \{s\}$.
\begin{claim}\label{claim:clique cutset}
 For $i\in [k]$,   $G_i$ has no clique cutset.
\end{claim}

Otherwise, let $i\in [k]$ such that $G_i$ has a clique cutset $S$. Since $G_i-(N_i\cup \{s\})$ is connected, $s\notin S$. However, it follows that $S$ is a clique cutset of $G$, a contradiction.  This proves Claim~\ref{claim:clique cutset}.
\begin{claim}\label{claim:series-parallel-1}
For $i\in [k],$ if $G_i$ is series-parallel and $G_i-(N[s]\cap \{x,y,z,s\})$ contains a cycle, then $\{x,y,z\}\cap V(C_i)\neq \emptyset$.
\end{claim}

Otherwise, let $i \in [k]$ such that $G_i$ is series-parallel, and suppose that $G_i- (N[s] \cap \{x,y,z,s\})$ contains a cycle. By Lemma~\ref{lem:series-parallel-stable}, we have that $N_i$ is an independent set.  This implies  that $1 \leq |N[s] \cap \{x,y,z,s\}| \leq 2$. Suppose that $V(C_i)\cap \{x,y,z\}=\emptyset$.  By Lemma~\ref{lem:series-parallel-clique-cycle} applied to $G_i$ and the vertices in $N[s] \cap \{x,y,z,s\}$, it follows that either there is a vertex in $V(G_i) \setminus N[s]$ of degree at most one in $G_i$ anticomplete to $\{x,y,z\} \cap N(s)$, or $G_i - (N[s] \cap \{x,y,z,s\})$ contains an induced cycle $C'$ with at least one vertex of degree two in $G_i$. In both cases, there is a vertex $t$ in $V(G_i) \setminus N[s]$ of degree at most two in $G_i$ and $t$ is anticomplete to $N[s] \cap \{x,y,z,s\}$, and hence its degree in $G$ is also at most two, a contradiction. This proves Claim~\ref{claim:series-parallel-1}.
\begin{claim}\label{claim:series-parallel-2}
   For $i \in [k]$, if $G_i$ is series-parallel, then $|V(C_i)| = 1$. 
\end{claim}

Otherwise, let $i \in [k]$ be such that $G_i$ is series-parallel and $|V(C_i)| >1$. Let $G'$ be the graph obtained from $G$ by contracting $V(C_i)$ to a new vertex $r$. We let $x' = r$ if $x \in V(C_i)$ and $x' = x$ otherwise;  we let $y' = r$ if $y \in V(C_i)$ and $y' = y$ otherwise; and we let $z' = r$ if $z \in V(C_i)$ and $z' = z$ otherwise.  By Lemma~\ref{lem:series-parallel-contracting}, $G'$ is $\iskdbpk$-free without clique cutsets. By Lemma~\ref{lem:non-centers}, $(x',y',z')$ is a noncentral triple for $G'$. By the minimality of $|V(G)|$, it follows that $(G', x', y',z')$ has the property $\mathcal{P}$. Let $v \in V(G') \setminus (N[x'] \cup N[y']\cup N[z'])$ be a vertex of degree at most two in $G'$. By the definition of $x'$,$y'$ and $z'$, it follows that $v \not\in N[x] \cup N[y]\cup N[z]$. It follows that either $v=r$, or  $v \in N(r)$.

First suppose that $v = r$. Then $r \not\in N[x'] \cup N[y']\cup N[z']$, and so $V(G_i) \cap \{x, y,z\} = \emptyset$. By Claim~\ref{claim:series-parallel-1}, it follows that $G_i- s$ is a tree. Since $v$ has degree at most two in $G'$, it follows that $|N_i| \leq 2$, and since $G_i- N[s]$ is connected, it follows that every vertex of $N_i$ is a leaf of $G_i - s$.  Thus, either $V(C_i)$ contains a leaf of $G_i - s$, or $G_i - s$ is a path with ends in $N_i$. In both cases $V(C_i)$ contains a vertex of degree at most two in $G$, a contradiction. 

Thus $v \neq r$ and $v \in N(r)$. That is, $v\in N(s)$. On the one hand, since $v\notin N[x'] \cup N[y']\cup N[z']$, $\{x,y,z\}\cap V(C_i)=\emptyset$. 
On the other hand, since $\deg_G(v) > 2$ and $\deg_{G'}(v) \leq 2$, $v$ has more than one neighbor in $V(C_i)$. Let $P$ be a path in $C_i$ between two neighbors of $v$, then $vPv$ is a cycle in $G_i-(N[s] \cap \{x,y,z,s\})$. 
By Claim~\ref{claim:series-parallel-1}, it follows that $V(C_i) \cap \{x, y,z\} \neq \emptyset$, a contradiction. This proves Claim~\ref{claim:series-parallel-2}.

\begin{claim}\label{claim:wheel-G_i}
For $i\in [k]$, if $G_i$ contains a wheel, then $ V(C_i)\cap \{x,y,z\}\neq \emptyset$.
\end{claim}

Otherwise, let $i \in [k]$ be such that $G_i$ contains a wheel and $V(C_i) \cap \{x,y,z\} = \emptyset$. Then $G_i$ is not a series-parallel graph.  Since $G$ is $K_4$-free, it follows that $|N_i \cap \{x,y,z\}| \leq 2$, and by symmetry, we may assume that $z\not\in N_i$. Let $z' = s$,  let $x' = x$ if $x \in N_i$ and $x' = s$ otherwise; and let $y' = y$ if $y \in N_i$ and $y' = s$ otherwise. By Lemma~\ref{lem:component-center}, $(x',y',z')$ is a noncentral triple for $G_i$. Since $G_i$ is an induced subgraph of $G$, it follows from Claim~\ref{claim:clique cutset} that $G_i$ is $\iskdbpk$-free without clique cutsets. Since $G$ is a minimum counterexample, it follows that $(G_i, x', y',z')$ has the property $\mathcal{P}$. Let $v$ be a vertex of degree at most two in $G_i$ with $v \not\in N[x'] \cup N[y']\cup N[z']$. Since $v \not\in N[s]$, it follows that $\deg_G(v) = \deg_{G_i}(v)$. This implies that $(G,x,y,z)$ has property $\mathcal{P}$, a contradiction. This proves Claim~\ref{claim:wheel-G_i}.

By Lemma~\ref{lem:ISK4-decomposition}, Claims~\ref{claim:series-parallel-2} and~\ref{claim:wheel-G_i}, there is at most one $i \in [k]$ with $|V(C_i)| > 1$. We may assume $|V(C_i)| = 1$ for all $i \in [k-1]$. If $|V(C_k)| > 1$, let $G' = G - (V(C_k) \cup \{s\})$; otherwise let $G' = G-s$.

\begin{claim}\label{claim:girth-7}
$G'$ has girth at least seven. 
\end{claim}
 
Observe that $V(G')\setminus N(s)$ is an independent set and $e(G'[N(s)]) \le 1$. Since $G$ is diamond-free, $G'$ is triangle-free.
We may assume that $z\notin V(G')$. 
Suppose that  $C=abcda$ is a cycle of length four in $G'$.  Without loss of generality, we may assume that $N(s)\cap V(C)=\{a,c\}$. 
If $\deg_G(b) \neq 2$, let $e$ be a neighbor of $b$ which is not $a,c$. Note that $e \in N(s)$. 
Then $\{a,b,c,d,e,s\}$ induces an $\iskfour$ or a $K_{3,3}$, a contradiction. 
It follows that $\deg_G(b) = \deg_G(d) = 2$, and moreover $\{x,y,z\} \cap \{a,c\} \neq \emptyset$. 
So, by symmetry, say $x = a$. Furthermore, if $y\in\{ b,d\}$, we may assume that $y=b$.
Since $N_G(d)=N_G(b)$ and $G$ contains a wheel, $G-d$  contains a wheel and hence $G-d$ is not series-parallel. Since $\{a,c\}\subseteq N(s)\cap N(b)$ and $sb\notin E(G)$, $G-d$ has no clique cutset. By the minimality of $|V(G)|$, it follows that there exists $v \in V(G) \setminus (N[x] \cup N[y]\cup N[z])$ with $\deg_{G-d}(v) \leq 2$. Since $\deg_G(v') = \deg_{G-d}(v')$ for all $v' \in V(G) \setminus \{a, d, c\}$, it follows that $v=c$, and so $N_G(c) = \{b,d,s\}$, and so $\{s, a\}$ is a  clique cutset in $G$, a contradiction.  This proves that $G'$ contains no $4$-cycle. 

Suppose $G'$ contains a $5$-cycle  or a $6$-cycle $C$. 
Since exactly three vertices in $V(C)$ are neighbors of $s$, $G[(V(C) \cup \{s\}]$ is an $\iskfour$, a contradiction. It follows that $G'$ has girth at least $7$. This proves  Claim~\ref{claim:girth-7}. 

\begin{claim}\label{claim:CK=1}
$|V(C_k)|> 1$.
\end{claim}

Suppose not, then $G' = G-s$ and $G'$ is triangle-free. So we may assume that $z=y$. Then $G'$ satisfies the hypotheses of Theorem~\ref{main-theorem:triangle}. Since $s$ is the center of a proper wheel in $G$ and $\deg_G(s)\geq 4$, it follows from Corollary~\ref{cor:component-number} that there exists a vertex $r$ in $G'$ that is not in $N_{G'}[x] \cup N_{G'}[y]$, and so $(\romannumeral1)$ of Theorem~\ref{main-theorem:triangle} does not hold (see Figure~\ref{fig:G'-1}). 

\begin{figure}[htbp]
\centering

\begin{minipage}{0.31\textwidth} 
\centering
\begin{tikzpicture}[
  scale=0.7,
  transform shape,
  thick,
  every node/.style={circle,draw=black,fill=black,inner sep=2.0pt}
]
\node[label={[above, yshift=-0.5mm]:{$s$}}] (0) at (-1.6,4) {};
\node[label={[above, yshift=-0.5mm]:{$x$}}] (1) at (-3.8,2.4) {};
\node[label={[above, yshift=-0.5mm]:{$y$}}] (2) at (-2.2,2.4) {};
\node[red] (3) at (-0.8,2.4) {};
\node (4) at (1.4,2.4) {};

\node[draw=none, fill=none] at (0.3,2.4) {$\cdots$};
\node[draw=none, fill=none] at (0.3,0.8) {$\cdots$};

\node (5) at (-3.8,0.8) {};
\node (6) at (-2.2,0.8) {};
\node (7) at (-0.8,0.8) {};
\node (8) at (1.4,0.8) {};

\path[draw, thick]
(0) edge[dashed] (1)
(0) edge[dashed] (2)
(0) edge[dashed] (3)
(0) edge[dashed] (4)
(1) edge (2);

\end{tikzpicture}
\end{minipage}
\hspace{-2mm} 
\begin{minipage}{0.31\textwidth}
\centering
\begin{tikzpicture}[
  scale=0.7,
  transform shape,
  thick,
  every node/.style={circle,draw=black,fill=black,inner sep=2.0pt}
]
\node[label={[above, yshift=-0.5mm]:{$s$}}] (0) at (-1.6,4) {};
\node[label={[above, yshift=-0.5mm]:{$x$}}] (1) at (-3.8,2.4) {};
\node (2) at (-2.2,2.4) {};
\node (3) at (-0.8,2.4) {};
\node (4) at (1.4,2.4) {};

\node[draw=none, fill=none] at (0.3,2.4) {$\cdots$};
\node[draw=none, fill=none] at (0.3,0.8) {$\cdots$};

\node[label={[below, yshift=-1.5mm]:{$y$}}] (5) at (-3.8,0.8) {};
\node (6) at (-2.2,0.8) {};
\node[red] (7) at (-0.8,0.8) {};
\node (8) at (1.4,0.8) {};

\path[draw, thick]
(0) edge[dashed] (1)
(0) edge[dashed] (2)
(0) edge[dashed] (3)
(0) edge[dashed] (4)
(1) edge (5)
(1) edge (6);

\end{tikzpicture}
\end{minipage}

\caption{Illustrations of the proof that $V(G') \neq N_{G'}[x] \cup N_{G'}[y]$.
}
\label{fig:G'-1}
\end{figure}
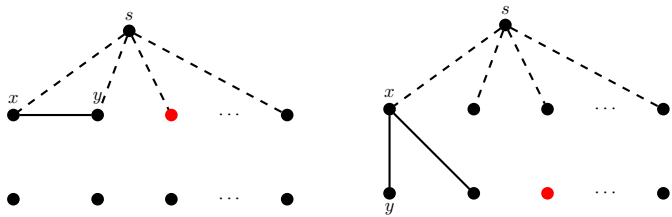

The  $(\romannumeral2)$  of Theorem~\ref{main-theorem:triangle} does not hold, because otherwise every vertex in $V(G') \setminus (N_{G'}[x] \cup N_{G'}[y])$ of degree one in $G'$ would have degree at most two in $G$, a contradiction. If $(\romannumeral 3)$, $(\romannumeral 4)$, or $(\romannumeral 5)$ of Theorem~\ref{main-theorem:triangle} holds, then there exists an induced cycle $C$ in $G'$ containing a vertex in $V(C) \setminus (N[x] \cup N[y] \cup N(s))$ of degree two in $G'$, and thus of degree two in $G$, a contradiction (see Figure~\ref{fig:10}). Therefore, $(\romannumeral6)$-$(\romannumeral10)$ of Theorem~\ref{main-theorem:triangle} holds. This implies that there exists an induced cycle $C$ with even length containing at least one of $x$, $y$ such that at most one vertex $v$ in $V(C)\setminus (N[x]\cup N[y])$ has $\deg_{G'}(v)>2$; or there exists an induced cycle $C$ with even length containing neither $x$ nor $y$ and a vertex $r\in V(C)$ such that at most one vertex $v$ in $V(C)\setminus N[r]$ has $\deg_{G'}(v)>2$. Since $G'$ has length seven, $C$ has length eight. However, it follows that $C$ contains a vertex in $V(C) \setminus (N[x] \cup N[y] \cup N(s))$ of degree two in $G'$, and thus of degree two in $G$, a contradiction (see Figure~\ref{fig:11}). Hence $|V(C_k)| > 1$, completing the proof of Claim~\ref{claim:CK=1}.

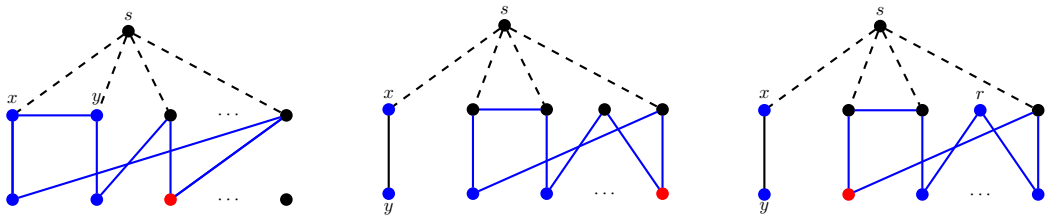
\begin{figure}[htbp]
\centering

\begin{minipage}{0.31\textwidth} 
\centering
\begin{tikzpicture}[
  scale=0.7,
  transform shape,
  thick,
  every node/.style={circle,draw=black,fill=black,inner sep=2.0pt}
]
\node[label={[above, yshift=-0.5mm]:{$s$}}] (0) at (-1.6,4) {};
\node[blue,label={[above, yshift=-0.5mm]:{$x$}}] (1) at (-3.8,2.4) {};
\node[blue,label={[above, yshift=-0.5mm]:{$y$}}] (2) at (-2.2,2.4) {};
\node (3) at (-0.8,2.4) {};
\node (4) at (1.4,2.4) {};

\node[draw=none, fill=none] at (0.3,2.4) {$\cdots$};
\node[draw=none, fill=none] at (0.3,0.8) {$\cdots$};

\node[blue] (5) at (-3.8,0.8) {};
\node[blue] (6) at (-2.2,0.8) {};
\node[red] (7) at (-0.8,0.8) {};
\node (8) at (1.4,0.8) {};

\path[draw, thick]
(0) edge[dashed] (1)
(0) edge[dashed] (2)
(0) edge[dashed] (3)
(0) edge[dashed] (4)
(1) edge[blue] (2)
(2) edge[blue] (6)
(3) edge[blue] (6)
(3) edge[blue] (7)
(4) edge[blue] (7)
(1) edge[blue] (5)
(7) edge[blue] (4)
(5) edge[blue] (4)
(1) edge[blue] (5);
\end{tikzpicture}
\end{minipage}
\hspace{-2mm}
\begin{minipage}{0.31\textwidth} 
\centering
\begin{tikzpicture}[
  scale=0.7,
  transform shape,
  thick,
  every node/.style={circle,draw=black,fill=black,inner sep=2.0pt}
]
\node[label={[above, yshift=-0.5mm]:{$s$}}] (0) at (-1.6,4) {};
\node[blue,label={[above, yshift=-0.5mm]:{$x$}}] (1) at (-3.8,2.4) {};
\node (2) at (-2.2,2.4) {};
\node (3) at (-0.8,2.4) {};
\node (4) at (1.4,2.4) {};

\node[label={[above, yshift=-0.5mm]:{}}]  (10) at (0.3,2.4) {};
\node[draw=none, fill=none] at (0.3,0.8) {$\cdots$};

\node[blue,label={[below, yshift=-1.5mm]:{$y$}}] (5) at (-3.8,0.8) {};
\node[blue] (6) at (-2.2,0.8) {};
\node[blue] (7) at (-0.8,0.8) {};
\node[red] (8) at (1.4,0.8) {};

\path[draw, thick]
(0) edge[dashed] (1)
(0) edge[dashed] (2)
(0) edge[dashed] (3)
(0) edge[dashed] (4)
(3) edge[blue] (2)
(3) edge[blue] (7)
(10) edge[blue] (7)
(10) edge[blue] (8)
(4) edge[blue] (8)
(1) edge (5)
(6) edge[blue] (2)
(4) edge[blue] (6);
\end{tikzpicture}
\end{minipage}
\hspace{-2mm}
\begin{minipage}{0.31\textwidth} 
\centering
\begin{tikzpicture}[
  scale=0.7,
  transform shape,
  thick,
  every node/.style={circle,draw=black,fill=black,inner sep=2.0pt}
]
\node[label={[above, yshift=-0.5mm]:{$s$}}] (0) at (-1.6,4) {};
\node[blue,label={[above, yshift=-0.5mm]:{$x$}}] (1) at (-3.8,2.4) {};
\node[label={[above, yshift=-0.5mm]:{}}] (2) at (-2.2,2.4) {};
\node (3) at (-0.8,2.4) {};
\node(4) at (1.4,2.4) {};

\node[blue,label={[above, yshift=-0.5mm]:{$r$}}]  (10) at (0.3,2.4) {};
\node[draw=none, fill=none] at (0.3,0.8) {$\cdots$};

\node[blue,label={[below, yshift=-1.2mm]:{$y$}}] (5) at (-3.8,0.8) {};
\node[red] (6) at (-2.2,0.8) {};
\node[blue] (7) at (-0.8,0.8) {};
\node[blue] (8) at (1.4,0.8) {};

\path[draw, thick]
(0) edge[dashed] (1)
(0) edge[dashed] (2)
(0) edge[dashed] (3)
(0) edge[dashed] (4)
(3) edge[blue] (2)
(3) edge[blue] (7)
(10) edge[blue] (7)
(10) edge[blue] (8)
(4) edge[blue] (8)
(1) edge (5)
(6) edge[blue] (2)
(4) edge[blue] (6);
\end{tikzpicture}
\end{minipage}

\caption{From left to right ($(\romannumeral3)$, $(\romannumeral4)$, $(\romannumeral5)$): the red vertex represents a degree-$2$ vertex, while the blue vertices represent possible degree-$3$ vertices and special vertices.}

\label{fig:10}
\end{figure}

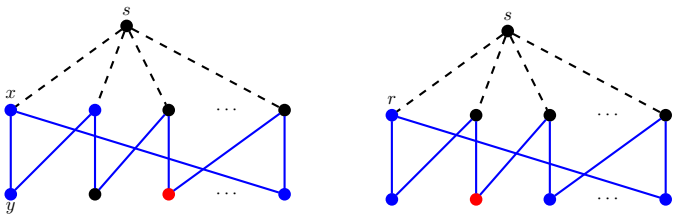
\begin{figure}[htbp]
\centering

\begin{minipage}{0.31\textwidth} 
\centering
\begin{tikzpicture}[
  scale=0.7,
  transform shape,
  thick,
  every node/.style={circle,draw=black,fill=black,inner sep=2.0pt}
]
\node[label={[above, yshift=-0.5mm]:{$s$}}] (0) at (-1.6,4) {};
\node[blue,label={[above, yshift=-0.5mm]:{$x$}}] (1) at (-3.8,2.4) {};
\node[blue,label={[above, yshift=-0.5mm]:{}}] (2) at (-2.2,2.4) {};
\node (3) at (-0.8,2.4) {};
\node (4) at (1.4,2.4) {};

\node[draw=none, fill=none] at (0.3,2.4) {$\cdots$};
\node[draw=none, fill=none] at (0.3,0.8) {$\cdots$};

\node[blue,label={[below, yshift=-1.2mm]:{$y$}}] (5) at (-3.8,0.8) {};
\node (6) at (-2.2,0.8) {};
\node[red] (7) at (-0.8,0.8) {};
\node[blue] (8) at (1.4,0.8) {};

\path[draw, thick]
(0) edge[dashed] (1)
(0) edge[dashed] (2)
(0) edge[dashed] (3)
(0) edge[dashed] (4)
(1) edge[blue] (5)
(5) edge[blue] (2)
(2) edge[blue] (6)
(3) edge[blue] (6)
(3) edge[blue] (7)
(7) edge[blue] (4)
(8) edge[blue] (4)
(8) edge[blue] (1)
;
\end{tikzpicture}
\end{minipage}
\begin{minipage}{0.31\textwidth} 
\centering
\begin{tikzpicture}[
  scale=0.7,
  transform shape,
  thick,
  every node/.style={circle,draw=black,fill=black,inner sep=2.0pt}
]
\node[label={[above, yshift=-0.5mm]:{$s$}}] (0) at (-1.6,4) {};
\node[blue,label={[above, yshift=-0.5mm]:{$r$}}] (1) at (-3.8,2.4) {};
\node[label={[above, yshift=-0.5mm]:{}}] (2) at (-2.2,2.4) {};
\node (3) at (-0.8,2.4) {};
\node(4) at (1.4,2.4) {};

\node[draw=none, fill=none]  (10) at (0.3,2.4) {$\cdots$};
\node[draw=none, fill=none] at (0.3,0.8) {$\cdots$};

\node[blue,label={[below, yshift=-1.2mm]:{}}] (5) at (-3.8,0.8) {};
\node[red] (6) at (-2.2,0.8) {};
\node[blue] (7) at (-0.8,0.8) {};
\node[blue] (8) at (1.4,0.8) {};

\path[draw, thick]
(0) edge[dashed] (1)
(0) edge[dashed] (2)
(0) edge[dashed] (3)
(0) edge[dashed] (4)
(1) edge[blue] (5)
(5) edge[blue] (2)
(2) edge[blue] (6)
(3) edge[blue] (6)
(3) edge[blue] (7)
(7) edge[blue] (4)
(8) edge[blue] (4)
(8) edge[blue] (1)
;
\end{tikzpicture}
\end{minipage}

\caption{Illustrations of $(\romannumeral 6)$--$(\romannumeral 10)$: the red vertex represents a degree-$2$ vertex, while the blue vertices represent possible degree-$3$ vertices and special vertices.}

\label{fig:11}
\end{figure}

 By Lemma~\ref{lem:ISK4-decomposition}, Claims~\ref{claim:series-parallel-2}, \ref{claim:wheel-G_i} and \ref{claim:CK=1}, we may assume that  $x\in V(C_k)$. 
 Since $G$ is diamond-free, $|N_k\cap \{y,z\}|\leq 1$. Without loss of generality, we may assume that $z\notin N_k$. 
 Now, let $G^*$ be obtained from $G$ by contracting $V(C_k) \cup N_k$ to a single vertex $x^*$, and by deleting $s$ and every vertex that is only adjacent to $x^*$. 
Note that $G^*-x^*$ has girth at least $7$ by Claim~\ref{claim:girth-7}. Therefore, if $G^*$ contains a triangle, then the triangle must contain $x^*$.

 \begin{claim}\label{claim:triangle-free-G*}
    $G^*$ is triangle-free.
 \end{claim}
 
For convenience, we denote this triangle by $x^*abx^*$, and we may further assume that $b \notin N(s)$. Since $G$ is diamond-free, $a$ and $b$ have no common neighbor in $N_k$. Then there exists an induced path $a'abb'$ with $a',b'\in N_k$.  Let $P$ be a shortest path between $a'$ and $b'$ with interior in $C_k$. Therefore, $G[\{s,a,b,a',b'\}\cup V(P)]$ is an $\iskfour$, a contradiction. This proves Claim~\ref{claim:triangle-free-G*}.

\begin{claim}\label{claim:no-4-cycle-z}
Every vertex in $V(G^*) \setminus \{x^*\}$ has at most one neighbor in $N_k$ in $G$. There is no $4$-cycle in $G^*$ containing $x^*$.
\end{claim}

First suppose that there is a vertex $v \in V(G^*) \setminus \{x^*\}$ with at least two neighbors $a, b \in N_k$ in $G$. Since $G$ is diamond-free, $v\notin N(s)$.
Since $v \in V(G^*)$ and there are no vertices of degree one adjacent to $x^*$ in $G^*$, it follows that $v$ has another neighbor $c \in N(s) \setminus N_k$.  Since $G$ is diamond-free, $\{a,b,c\}$ is an independent set. Let $P$ be a shortest  path connecting $a$ and $b$ with interior in $V(C_k)$. Such a path exists, since $a, b \in N_k$. It follows that $G[V(P) \cup \{a,b,c,v,s\}]$ induces an $\iskfour$ in $G$, a contradiction.

Suppose that $x^*$ is contained in a $4$-cycle with vertex set $\{a,b,c,x^*\}$ in $G^*$ such that $a, c \in N_{G^*}(x^*)$. Note that $a, c \not\in N(s)$ and $b \in N(s) \setminus N_k$. By Claim~\ref{claim:girth-7}, $G - (\{s\} \cup V(C_k))$ contains no 4-cycle, and thus $a$ and $c$ have no common neighbor in $N_k$. Let $a', c'$ be a neighbor of $a$ and $c$ in $N_k$, respectively; $a'$ and $c'$ exist since $a, c \in N_{G^*}(x^*)$. Let $P$ be a shortest path between $a'$ and $c'$ with interior in $C_k$. Since $b \not\in N_k$, it follows that  $b$ is anticomplete to $V(P)$. Therefore, $G[\{a,b,c, s\} \cup V(P)]$ is an $\iskfour$ in $G$, a contradiction. This proves Claim~\ref{claim:no-4-cycle-z}.

\begin{claim}\label{claim:G^*-isktk-free}
$G^*$ is $\isktk$-free.
\end{claim}

By Claim~\ref{claim:triangle-free-G*}, we have that $G^*$ is triangle-free. 
Suppose that $G^*$ contains an induced subgraph $H$ which is either a $K_{3,3}$ or an $\iskfour$. Since $G$ is $\iskk$-free, it follows that $x^*\in V(H)$. Suppose that $x^*$ has degree two in $H$. So $H$ is an $\iskfour$. Let $P$ be a  shortest path in $G$ connecting the neighbors of $x^*$ in $V(H)$ with interior in $V(C_k)\cup V(N_k)$. Then $G[V(H-x^*) \cup V(P)]$ is an $\iskfour$ in $G$, a contradiction. 

 It follows that $x^*$ has degree three in $H$. Let $a, b, c$ be the neighbors of $x^*$ in $H$.
 By Claim~\ref{claim:no-4-cycle-z}, each of $a,b,c$ has a unique neighbor in $N_k$. Let $a', b', c'$ be neighbors of $a,b,c$ in $N_k$. Let $H'$ be a minimal induced subgraph of $G[V(C_k) \cup \{a, b, c, a',b',c'\}]$ which is connected and contains $\{a,b,c\}$.  By Lemma~\ref{lem:degree-1}, $H'$ either is a subdivision of $K_{1,3}$ in which $a, b, c$ are the vertices of degree one or $H'$ is a tripod graph with $a,b,c$ being the vertices of degree one. Consequently, $G[(V(H) \setminus \{x^*\})   \cup V(H')]$  is an induced subgraph of $G$ which is either a $K_{3,3}$ or an $\iskfour$, a contradiction.  This proves Claim~\ref{claim:G^*-isktk-free}.

\begin{claim}\label{claim:center=x^*}
$G^*$ does not contain a proper wheel with center different from $x^*$.
\end{claim}

Otherwise, assume that  $v \neq x^*$ is the center of a proper wheel $G^*$. By Theorem~\ref{thm:proper-wheel-property}, there is a component $C$  of $G^*- N_{G^*}[v]$ that is disjoint from $N[x^*]$. Let $N$ denote the set of vertices in $N_{G^*}(v)$ with a neighbor in $C$. 

Let $H= G^*[(N \cup V(C) \cup \{v\}]$. Then $(H ,v)$ satisfies the hypotheses of Theorem~\ref{main-theorem:triangle}. Since $V(C) \neq \emptyset$, it follows that $(\romannumeral1)$ of Theorem~\ref{main-theorem:triangle} does not hold. Moreover, every vertex in $V(H) \setminus N[v]$ of degree one in $H$ has degree at most two in $G$, since such a vertex belongs to $C$ and the only additional neighbor that such a vertex may have in $G$ is $s$. Furthermore, such a vertex is in $V(G^*) \setminus N[x^*]$, and hence in $V(G) \setminus (N[x] \cup N[y]\cup N[z])$ as $x \in V(C_k)$. It follows that $(\romannumeral2)$ of Theorem~\ref{main-theorem:triangle} does not hold. 

Therefore, one of $(\romannumeral3),(\romannumeral4),(\romannumeral5)$ of Theorem~\ref{main-theorem:triangle} holds, and hence it follows from Claim~\ref{claim:girth-7} that  there exists an induced cycle $C'$  of length at least $7$ in $H$ such that there exists a vertex $v'$ in $V(C') \setminus (N_{G^*}[x^*] \cup N_G(s))$ of degree two in $G^*$.  
 Since $x^* \not\in V(H)$, $v'\notin N_G[x]\cup N_G[y]\cup N_G[z]$. Consequently, $v'$ has  degree two in $G^*$ and thus has degree two in $G$, a contradiction. This proves Claim~\ref{claim:center=x^*}.

\begin{claim}\label{claim:forest}
For every component $K$ of $G^*- N[x^*]$, $G^*[V(K) \cup N(x^*)]$ is a forest.
\end{claim}

Suppose not. Let $K$ be a component of $G^*- N[x^*]$ such that $G^*[V(K) \cup N(x^*)]$ is not a forest. First suppose that $H = G^*[V(K) \cup N[x^*]]$ is not series-parallel. Then $H$ contains a proper wheel by Lemma~\ref{lem:proper-wheel-exist}.
Let $v$ be the center of a proper wheel in $H$. Since $H -N[x^*]$ is connected, it follows from Corollary~\ref{cor:component-number} that $v \neq x^*$. 
By Lemma~\ref{lem:component-center}, it follows that $v$ is the center of a proper wheel in $G^*$, a contradiction to Claim~\ref{claim:center=x^*}.  

Therefore, $H$ is series-parallel. By assumption, $H -x^*$ contains a cycle. By applying Lemma~\ref{lem:farcycle} to $H$ and $x^*$, it follows that there is either a vertex in $V(H) \setminus N[x^*]$ of degree one, or a cycle $C$ not containing $x^*$, with all but at most two vertices of degree two in $H$. In the latter case, since $G^* -x^*$ has girth at least seven, $C$ contains one vertex in $V(H) \setminus (N[x^*]\cup N[s])$ of degree two in $H$, and thus of degree two in $G$. It follows that in both cases $G$ contains a vertex of degree at most two not in $N[x^*]$, and thus not in $N[x] \cup N[y]\cup N[z]$, a contradiction.
This proves Claim~\ref{claim:forest}.

\begin{claim}\label{claim:girth-12}
    The girth of $G^* - x^*$ is at least $12$.
\end{claim}

Suppose not. Let $C$ be an induced cycle in $G^* - x^*$ of length less than $12$. Since by Claim~\ref{claim:forest}, for every component $K$ of $G^*- N[x^*]$, we have that $G^*[V(K) \cup N(x^*)]$ is a forest, it follows that $C \setminus N[x^*]$ has at least two components. Since $x^*$ is not contained in a $4$-cycle in $G^*$ by~Claim~\ref{claim:no-4-cycle-z},  it follows that each component of $C \setminus N[x^*]$ has at least two vertices. If $C \setminus N[x^*]$ has at least four components, it follows that $C$ has length at least $12$. If $C \setminus N[x^*]$ has exactly three components, then $G[V(C) \cup \{x^*\}]$ is an $\iskfour$, a contradiction. So $C - N[x^*]$ has exactly two components. For every  component $K$ of $G^* \setminus N[x^*]$, by Claim~\ref{claim:forest}, we have that $G^*[(V(K) \cup N(x^*))$ is a forest.

Therefore, the two components of $C - N[x^*]$ are contained in two different components of $G^*- N[x^*]$; say $A$ and $B$. Let $N_A, N_B$ denote the set of vertices in $N(x^*)$ with a neighbor in $A$, $B$, respectively. Suppose that $|N_A| \geq 3$. Since $|V(C) \cap N(x^*)| = 2$, it follows that there is a path $P$ from a vertex $c$ in $N_A \setminus V(C)$ to $V(C)$ with interior in $V(A)$. Since $G^*[V(A) \cup N_A]$ and $G^*[V(B) \cup N_B]$ are trees, it follows that $c$ has at most one neighbor in each component $K$ of $C \setminus N[x^*]$. Therefore, $G^*[(V(P) \cup V(C) \cup \{x^*\})]$ contains an induced subgraph of $G^*$ which is either a subdivision of $K_4$ or of $K_{3,3}$, thus a subdivision of $K_4$ or $K_{3,3}$ by Lemma \ref{lem:ISK4-decomposition}, a contradiction.

So $|N_A| = 2$. Since $G^*[V(A) \cup N_A]$ is a tree, it follows that either $A$ contains a vertex of degree one in $G^*$, which is not adjacent to $x^*$, or $G^*[V(A) \cup N_A]$ is a path containing at least four vertices, and hence $A$ contains two adjacent vertices of degree two in $G^*$, which are not adjacent to $x^*$. If $A$ contains a vertex $v$ of degree one in $G^*$ that is not adjacent to $x^*$, then $v$ has degree at most two outside $N[x^*]$, and thus outside $N[x]\cup N[y]\cup N[z]$, a contradiction. This implies that $G^*[V(A)\cup N_A]$ is a path with at least four vertices, and $A$ contains two adjacent vertices of degree two in $G^*$, both not adjacent to $x^*$. Similarly, $G^*[V(B)\cup N_B]$ is a path with at least four vertices, and $B$ contains two adjacent vertices of degree two in $G^*$, both not adjacent to $x^*$.

Since $G^*-x^*$ has girth at least seven and $x^*$ is not contained in any $4$-cycle of $G^*$ by Claims~\ref{claim:girth-7} and~\ref{claim:no-4-cycle-z},
  at least one of the  $G^*[V(A)\cup N_A]$ and $G^*[V(B)\cup N_B]$ is a path of length at least $4$ (i.e. on at least $5$ vertices). However,
it follows that $G$ contains a vertex of degree at most two not in $N[x^*]$, and thus not in $N[x] \cup N[y]\cup N[z]$, a contradiction. This proves Claim~\ref{claim:girth-12}.

Recall that $\{x,y,z\} \cap V(C_k) \neq \emptyset$, and we may assume that $x \in V(C_k)$. 
Since $G$ is diamond-free, $|\{y,z\}\cap N_k|\leq 1$.
Without loss of generality, we may assume that $z\notin N_k$.  Let $G''$ be the graph obtained from $G$ by deleting $\{s\} \cup (V(C_k) \setminus \{x\}) \cup (N_k \setminus \{y\})$, and every vertex other than $x$ with neighbors only in $N_k$ (this last operation does not change the degree of any vertex in $V(G'')$ except for possibly $y$). Then $N_{G''}(x) \subseteq \{y\}$.
It follows from Claim~\ref{claim:girth-12} that $G'' \setminus \{y\}$ has girth at least $12$, and from Claim~\ref{claim:girth-7} that $G''$ has girth at least $7$. If $y \in V(G'')$, let $y' = y$; otherwise, let $y' = x$. It follows that if $y' = y$, then $y \in N_k$.

Note that $\deg_{G''}(x)\leq 1$ and $G''-x$ is an induced subgraph of $G^*$. Then $G''$ is triangle-free. 
Since $G''$ is an induced subgraph of $G$ and $G''$ is triangle-free, it follows that $G''$ and $x, y'$ satisfy the hypotheses of Theorem~\ref{main-theorem:triangle}. 

Since $s$ is the center of a proper wheel, it follows from Theorem~\ref{thm:proper-wheel-property} $(\romannumeral2)$ that there are at least one component of $G- N[s]$ in which $y'$ has no neighbors. On the one hand, by Claim~\ref{claim:no-4-cycle-z}, we have that every vertex in $V(G^*) \setminus \{x^*\}$ has at most one neighbor in $N_k$ in $G$.  On the other hand, since every vertex in $V(G) \setminus (N[x] \cup N[y]\cup N[z])$ has degree at least three in $G$,   $V(G'') \neq N[x] \cup N[y']$, and thus $(\romannumeral1)$ of Theorem~\ref{main-theorem:triangle} does not hold (see Figure~\ref{maintheoremfig1}).

\begin{figure}[H]
\centering
\begingroup

\begin{tikzpicture}[x=0.5pt,y=0.5pt,yscale=-1,xscale=1]
\tikzset{
  fig/edge/.style     ={line width=1.5, draw=black},
  fig/dedge/.style    ={fig/edge, dash pattern={on 1.69pt off 2.76pt}},
  fig/thickdedge/.style={line width=2.25, draw=black, dash pattern={on 2.53pt off 3.02pt}},
  fig/v/.style        ={draw=black, fill=black},
  fig/vred/.style     ={draw=red, fill=red},
}
\def\PointR{4}

\coordinate (C)      at (265.67,220); 
\def\Cr{45} 

\coordinate (s)      at (340,100);
\coordinate (yp)     at (271,150);

\coordinate (vL)     at (318,150);
\coordinate (vM)     at (362.67,150);
\coordinate (vR)     at (447,150);

\coordinate (x)      at (245.33,220);
\coordinate (z)      at (285.33,220);
\coordinate (zproj)  at (285.33,220);     
\coordinate (w)      at (331.33,220);

\coordinate (rred)   at (387.67,220);

\draw[draw=black] (C) circle[radius=\Cr];

\draw[fig/edge]  (yp)--(s);
\draw[fig/edge]  (yp)--(x);
\draw[fig/edge]  (s)--(vL);
\draw[fig/edge]  (s)--(vM);
\draw[fig/edge]  (s)--(vR);

\draw[fig/dedge] (yp)--(z);
\draw[fig/dedge] (x)--(zproj);
\draw[fig/dedge] (yp)--(w);

\draw[fig/thickdedge] (384,150)--(426,150);

\path[fig/v]    (s)    circle[radius=\PointR];
\path[fig/v]    (yp)   circle[radius=\PointR];
\path[fig/v]    (vL)   circle[radius=\PointR];
\path[fig/v]    (vM)   circle[radius=\PointR];
\path[fig/v]    (vR)   circle[radius=\PointR];
\path[fig/v]    (x)    circle[radius=\PointR];
\path[fig/v]    (z)    circle[radius=\PointR];
\path[fig/v]    (w)    circle[radius=\PointR];
\path[fig/vred] (rred) circle[radius=\PointR];

\node[inner sep=0pt, label={[above]:{$s$}}]        at (s)  {};
\node[inner sep=0pt, label={[left]:{$y'$}}]        at (yp) {};
\node[inner sep=0pt, label={[below left]:{$x$}}]   at (x)  {};
\node[inner sep=0pt, label={[below right]:{$z$}}]  at (z)  {};

\end{tikzpicture}

\endgroup
\caption{Illustration of the proof that $V(G'') \neq N[x] \cup N[y']$.}
\label{maintheoremfig1}
\end{figure}
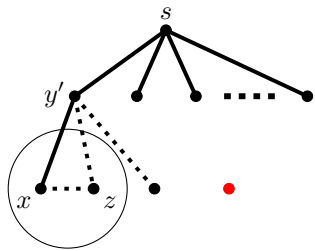

The case $(\romannumeral2)$ of Theorem~\ref{main-theorem:triangle} does not hold, because if $G''$ contains a vertex $v$ of degree one non-adjacent to $y'$, then $v$ has degree at most two in $G$, and $v \not\in N[x] \cup N[y]\cup N[z]$, a contradiction.

Suppose that $(\romannumeral3)$ of Theorem~\ref{main-theorem:triangle} holds. So $G''$ contains an induced cycle $C$ containing $y'$ (since $d_{G''}(x)=1$) such that each vertex in $V(C) \setminus N_{G''}[y]$ has degree  two. Since $G''$ has girth at least seven, $|V(C)| \geq 7$, and in particular $C$ contains a vertex $v\notin N(s)$ of distance three from $y'$ in $C$ and degree two in $G''$ (see Figure~\ref{maintheoremfig3}). 
Let $y'abv$ be the three-edge path from $y'$ to $v$ in $C$.  
Moreover, $v$ anticomplete to $N_k$, otherwise $x^*abvx^*$ is a 4-cycle in $G'$ using $x^*$, contradicting Claim~\ref{claim:no-4-cycle-z}. So $v$ has degree two in $G$ and is not in $N[x] \cup N[y]\cup N[z]$, a contradiction. 

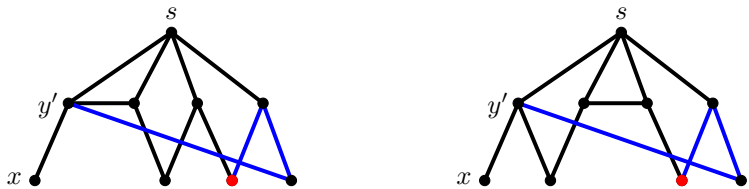
\begin{figure}[H]
\centering
\begingroup

\begin{tikzpicture}[x=0.75pt,y=0.75pt,yscale=-1,xscale=1,scale=0.65]
\tikzset{
  figY/edge/.style     ={line width=1.5, draw=black},
  figY/blueedge/.style ={figY/edge, draw=blue},
  figY/v/.style        ={draw=black, fill=black},
  figY/vred/.style     ={draw=red,  fill=red},
}

\def\PointR{4}

\begin{scope}
  \coordinate (yp) at (98,160);    
  \coordinate (a)  at (149,160);
  \coordinate (s)  at (178,105);    
  \coordinate (d)  at (198,160);
  \coordinate (e)  at (249,160);
  \coordinate (b)  at (173,220);
  \coordinate (r)  at (225,220);   
  \coordinate (h)  at (271,220);
  \coordinate (x)  at (72,220);    

  \draw[figY/edge] (a)--(yp);
  \draw[figY/edge] (a)--(b);
  \draw[figY/edge] (b)--(d);
  \draw[figY/edge] (a)--(s);
  \draw[figY/edge] (s)--(yp);
  \draw[figY/edge] (s)--(d);
  \draw[figY/edge] (s)--(e);
  \draw[figY/edge] (d)--(r);
  \draw[figY/edge] (x)--(yp);

  \draw[figY/blueedge] (yp)--(h);
  \draw[figY/blueedge] (e)--(h);
  \draw[figY/blueedge] (r)--(e);

  \path[figY/v] (yp) circle[radius=\PointR];
  \path[figY/v] (a)  circle[radius=\PointR];
  \path[figY/v] (s)  circle[radius=\PointR];
  \path[figY/v] (d)  circle[radius=\PointR];
  \path[figY/v] (e)  circle[radius=\PointR];
  \path[figY/v] (b)  circle[radius=\PointR];
  \path[figY/v] (r)  circle[radius=\PointR];
  \path[figY/v] (h)  circle[radius=\PointR];
  \path[figY/v] (x)  circle[radius=\PointR];

  \path[figY/vred] (r) circle[radius=\PointR];

  \node[inner sep=0pt, label={[left, yshift=-0.5mm]:{$y'$}}] at (yp) {};
  \node[inner sep=0pt, label={[left,  xshift=-0.5mm]:{$x$}}]  at (x)  {};
  \node[inner sep=0pt, label={[above, yshift=0.3mm]:{$s$}}]  at (s)  {};
\end{scope}

\begin{scope}
  \coordinate (yp) at (447,160);   
  \coordinate (b)  at (498,160);
  \coordinate (s)  at (527,105);    
  \coordinate (d)  at (547,160);
  \coordinate (e)  at (598,160);
  \coordinate (m)  at (472,220);
  \coordinate (x)  at (421,220);   
  \coordinate (h)  at (620,220);
  \coordinate (r)  at (574,220);   

  \draw[figY/edge] (m)--(yp);
  \draw[figY/edge] (m)--(b);
  \draw[figY/edge] (x)--(yp);

  \draw[figY/edge] (b)--(s);
  \draw[figY/edge] (s)--(yp);
  \draw[figY/edge] (s)--(d);
  \draw[figY/edge] (b)--(d);

  \draw[figY/edge] (s)--(e);
  \draw[figY/edge] (d)--(r);

  \draw[figY/blueedge] (yp)--(h);
  \draw[figY/blueedge] (e)--(h);
  \draw[figY/blueedge] (r)--(e);

  \path[figY/v] (yp) circle[radius=\PointR];
  \path[figY/v] (b)  circle[radius=\PointR];
  \path[figY/v] (s)  circle[radius=\PointR];
  \path[figY/v] (d)  circle[radius=\PointR];
  \path[figY/v] (e)  circle[radius=\PointR];
  \path[figY/v] (m)  circle[radius=\PointR];
  \path[figY/v] (x)  circle[radius=\PointR];
  \path[figY/v] (h)  circle[radius=\PointR];
  \path[figY/v] (r)  circle[radius=\PointR];

  \path[figY/vred] (r) circle[radius=\PointR];

  \node[inner sep=0pt, label={[left, yshift=-0.5mm]:{$y'$}}] at (yp) {};
  \node[inner sep=0pt, label={[left,  xshift=-0.5mm]:{$x$}}]  at (x)  {};
  \node[inner sep=0pt, label={[above, yshift=0.3mm]:{$s$}}]  at (s)  {};
\end{scope}

\end{tikzpicture}

\endgroup
\caption{Illustrations for $(\romannumeral3)$ showing the existence of a $3$-edge path:
blue edges indicate the position of the $3$-edge path
and red vertices indicate degree-$2$ vertices.}
\label{maintheoremfig3}
\end{figure}

Suppose that $(\romannumeral4)$ of Theorem~\ref{main-theorem:triangle} holds. So $G''$ contains an induced cycle $C$ not containing $x,y'$ and all but at most two vertices of $C$ have degree two in $G''$.  For convenience, assume that $C=v_1v_2\ldots v_nv_1$. Furthermore, by symmetry, we may assume that every vertex in $\{v_7,\ldots,v_n\}$ has degree two in $G''$. Since $|V(C)| \geq 12$, it follows that $C$ contains a path $P = p_1p_2p_3$ of three vertices with $N(s) \cap V(P)=\{p_2\}$, all of degree two in $G''$ and non-adjacent to $x,y'$ (see Figure~\ref{maintheoremfig4}). Since $x^*$ is not in a $4$-cycle in $G^*$ by Claim~\ref{claim:no-4-cycle-z}, not both $p_1$ and $p_3$ have a neighbor in $N_k$.  
Therefore,  $\{p_1,p_2,p_3\}$ is anticomplete to $\{x,y,z\}$. 
It follows that either $p_1$ or $p_3$ has degree two in $G$, a contradiction.

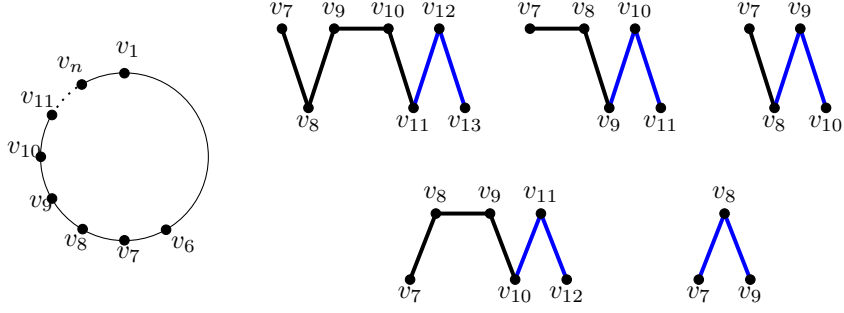
\begin{figure}[H]
\centering
\begingroup

\begin{tikzpicture}[x=0.5pt,y=0.5pt,yscale=-1,xscale=1]
\tikzset{
  fig/edge/.style     ={line width=1.5, draw=black},
  fig/blueedge/.style ={fig/edge, draw=blue},
  fig/v/.style        ={draw=black, fill=black},
}
\def\PointR{3.6}


\draw  [draw opacity=0] (65.47,70.77) .. controls (74.29,66.13) and (84.34,63.5) .. (95,63.5) .. controls (130.07,63.5) and (158.5,91.93) .. (158.5,127) .. controls (158.5,162.07) and (130.07,190.5) .. (95,190.5) .. controls (59.93,190.5) and (31.5,162.07) .. (31.5,127) .. controls (31.5,115.66) and (34.47,105.01) .. (39.68,95.79) -- (95,127) -- cycle ; \draw   (65.47,70.77) .. controls (74.29,66.13) and (84.34,63.5) .. (95,63.5) .. controls (130.07,63.5) and (158.5,91.93) .. (158.5,127) .. controls (158.5,162.07) and (130.07,190.5) .. (95,190.5) .. controls (59.93,190.5) and (31.5,162.07) .. (31.5,127) .. controls (31.5,115.66) and (34.47,105.01) .. (39.68,95.79) ;

\draw  [fill={rgb, 255:red, 0; green, 0; blue, 0 }  ,fill opacity=1 ] (122.9,182.25) .. controls (122.9,180.26) and (124.51,178.65) .. (126.5,178.65) .. controls (128.49,178.65) and (130.1,180.26) .. (130.1,182.25) .. controls (130.1,184.24) and (128.49,185.85) .. (126.5,185.85) .. controls (124.51,185.85) and (122.9,184.24) .. (122.9,182.25) -- cycle ;
\draw  [fill={rgb, 255:red, 0; green, 0; blue, 0 }  ,fill opacity=1 ] (59.13,72.29) .. controls (59.13,70.3) and (60.75,68.69) .. (62.74,68.69) .. controls (64.73,68.69) and (66.34,70.3) .. (66.34,72.29) .. controls (66.34,74.29) and (64.73,75.9) .. (62.74,75.9) .. controls (60.75,75.9) and (59.13,74.29) .. (59.13,72.29) -- cycle ;
\draw  [fill={rgb, 255:red, 0; green, 0; blue, 0 }  ,fill opacity=1 ] (36.62,95.38) .. controls (36.62,93.38) and (38.23,91.77) .. (40.22,91.77) .. controls (42.21,91.77) and (43.83,93.38) .. (43.83,95.38) .. controls (43.83,97.37) and (42.21,98.98) .. (40.22,98.98) .. controls (38.23,98.98) and (36.62,97.37) .. (36.62,95.38) -- cycle ;
\draw  [fill={rgb, 255:red, 0; green, 0; blue, 0 }  ,fill opacity=1 ] (28.15,127) .. controls (28.15,125.01) and (29.76,123.4) .. (31.75,123.4) .. controls (33.74,123.4) and (35.35,125.01) .. (35.35,127) .. controls (35.35,128.99) and (33.74,130.6) .. (31.75,130.6) .. controls (29.76,130.6) and (28.15,128.99) .. (28.15,127) -- cycle ;
\draw  [fill={rgb, 255:red, 0; green, 0; blue, 0 }  ,fill opacity=1 ] (36.62,158.63) .. controls (36.62,156.63) and (38.23,155.02) .. (40.22,155.02) .. controls (42.21,155.02) and (43.83,156.63) .. (43.83,158.63) .. controls (43.83,160.62) and (42.21,162.23) .. (40.22,162.23) .. controls (38.23,162.23) and (36.62,160.62) .. (36.62,158.63) -- cycle ;
\draw  [fill={rgb, 255:red, 0; green, 0; blue, 0 }  ,fill opacity=1 ] (59.77,181.78) .. controls (59.77,179.79) and (61.38,178.17) .. (63.38,178.17) .. controls (65.37,178.17) and (66.98,179.79) .. (66.98,181.78) .. controls (66.98,183.77) and (65.37,185.38) .. (63.38,185.38) .. controls (61.38,185.38) and (59.77,183.77) .. (59.77,181.78) -- cycle ;
\draw  [fill={rgb, 255:red, 0; green, 0; blue, 0 }  ,fill opacity=1 ] (91.4,190.25) .. controls (91.4,188.26) and (93.01,186.65) .. (95,186.65) .. controls (96.99,186.65) and (98.6,188.26) .. (98.6,190.25) .. controls (98.6,192.24) and (96.99,193.85) .. (95,193.85) .. controls (93.01,193.85) and (91.4,192.24) .. (91.4,190.25) -- cycle ;
\draw  [fill={rgb, 255:red, 0; green, 0; blue, 0 }  ,fill opacity=1 ] (91.4,63.75) .. controls (91.4,61.76) and (93.01,60.15) .. (95,60.15) .. controls (96.99,60.15) and (98.6,61.76) .. (98.6,63.75) .. controls (98.6,65.74) and (96.99,67.35) .. (95,67.35) .. controls (93.01,67.35) and (91.4,65.74) .. (91.4,63.75) -- cycle ;
\draw  [draw opacity=0][dash pattern={on 0.84pt off 2.51pt}][line width=0.75]  (40.22,95.38) .. controls (45.46,86.04) and (52.98,78.16) .. (62.03,72.49) -- (95.66,126.37) -- cycle ; \draw  [dash pattern={on 0.84pt off 2.51pt}][line width=0.75]  (40.22,95.38) .. controls (45.46,86.04) and (52.98,78.16) .. (62.03,72.49) ;

\draw (86.5,38) node [anchor=north west][inner sep=0.75pt]   [align=left] {$\displaystyle v_{1}$};
\draw (128.5,185.25) node [anchor=north west][inner sep=0.75pt]   [align=left] {$\displaystyle v_{6}$};
\draw (42,49) node [anchor=north west][inner sep=0.75pt]   [align=left] {$\displaystyle v_{n}$};
\draw (14.5,76) node [anchor=north west][inner sep=0.75pt]   [align=left] {$\displaystyle v_{11}$};
\draw (4.5,114) node [anchor=north west][inner sep=0.75pt]   [align=left] {$\displaystyle v_{10}$};
\draw (21,154.5) node [anchor=north west][inner sep=0.75pt]   [align=left] {$\displaystyle v_{9}$};
\draw (47.5,184) node [anchor=north west][inner sep=0.75pt]   [align=left] {$\displaystyle v_{8}$};
\draw (87,193.5) node [anchor=north west][inner sep=0.75pt]   [align=left] {$\displaystyle v_{7}$};


\begin{scope}
  \coordinate (B7)  at (214.2,30);
  \coordinate (B8)  at (234.2,90);
  \coordinate (B9)  at (253.7,30);
  \coordinate (B10) at (294.7,30);
  \coordinate (B11) at (313.7,90);
  \coordinate (B12) at (333.2,30);
  \coordinate (B13) at (352.7,90);

  \draw[fig/edge]     (B7)--(B8);
  \draw[fig/edge]     (B9)--(B8);
  \draw[fig/edge]     (B9)--(B10);
  \draw[fig/edge]     (B10)--(B11);
  \draw[fig/blueedge] (B12)--(B11);
  \draw[fig/blueedge] (B12)--(B13);

  \path[fig/v] (B7)  circle[radius=\PointR];
  \path[fig/v] (B8)  circle[radius=\PointR];
  \path[fig/v] (B9)  circle[radius=\PointR];
  \path[fig/v] (B10) circle[radius=\PointR];
  \path[fig/v] (B11) circle[radius=\PointR];
  \path[fig/v] (B12) circle[radius=\PointR];
  \path[fig/v] (B13) circle[radius=\PointR];

  \node[inner sep=0pt, label={[above]:{$v_{7}$}}]  at (B7)  {};
  \node[inner sep=0pt, label={[below]:{$v_{8}$}}]  at (B8)  {};
  \node[inner sep=0pt, label={[above]:{$v_{9}$}}]  at (B9)  {};
  \node[inner sep=0pt, label={[above]:{$v_{10}$}}] at (B10) {};
  \node[inner sep=0pt, label={[below]:{$v_{11}$}}] at (B11) {};
  \node[inner sep=0pt, label={[above]:{$v_{12}$}}] at (B12) {};
  \node[inner sep=0pt, label={[below]:{$v_{13}$}}] at (B13) {};
\end{scope}

\begin{scope}
  \coordinate (C7)  at (401.8,30);
  \coordinate (C8)  at (442.8,30);
  \coordinate (C9)  at (461.8,90);
  \coordinate (C10) at (481.3,30);
  \coordinate (C11) at (500.8,90);

  \draw[fig/edge]     (C7)--(C8);
  \draw[fig/edge]     (C8)--(C9);
  \draw[fig/blueedge] (C10)--(C9);
  \draw[fig/blueedge] (C10)--(C11);

  \path[fig/v] (C7)  circle[radius=\PointR];
  \path[fig/v] (C8)  circle[radius=\PointR];
  \path[fig/v] (C9)  circle[radius=\PointR];
  \path[fig/v] (C10) circle[radius=\PointR];
  \path[fig/v] (C11) circle[radius=\PointR];

  \node[inner sep=0pt, label={[above]:{$v_{7}$}}]  at (C7)  {};
  \node[inner sep=0pt, label={[above]:{$v_{8}$}}]  at (C8)  {};
  \node[inner sep=0pt, label={[below]:{$v_{9}$}}]  at (C9)  {};
  \node[inner sep=0pt, label={[above]:{$v_{10}$}}] at (C10) {};
  \node[inner sep=0pt, label={[below]:{$v_{11}$}}] at (C11) {};
\end{scope}

\begin{scope}
  \coordinate (D7)  at (311.1,220);
  \coordinate (D8)  at (330.6,170);
  \coordinate (D9)  at (371.6,170);
  \coordinate (D10) at (390.6,220);
  \coordinate (D11) at (410.1,170);
  \coordinate (D12) at (429.6,220);

  \draw[fig/edge]     (D8)--(D7);
  \draw[fig/edge]     (D8)--(D9);
  \draw[fig/edge]     (D9)--(D10);
  \draw[fig/blueedge] (D11)--(D10);
  \draw[fig/blueedge] (D11)--(D12);

  \path[fig/v] (D7)  circle[radius=\PointR];
  \path[fig/v] (D8)  circle[radius=\PointR];
  \path[fig/v] (D9)  circle[radius=\PointR];
  \path[fig/v] (D10) circle[radius=\PointR];
  \path[fig/v] (D11) circle[radius=\PointR];
  \path[fig/v] (D12) circle[radius=\PointR];

  \node[inner sep=0pt, label={[below]:{$v_{7}$}}]  at (D7)  {};
  \node[inner sep=0pt, label={[above]:{$v_{8}$}}]  at (D8)  {};
  \node[inner sep=0pt, label={[above]:{$v_{9}$}}]  at (D9)  {};
  \node[inner sep=0pt, label={[below]:{$v_{10}$}}] at (D10) {};
  \node[inner sep=0pt, label={[above]:{$v_{11}$}}] at (D11) {};
  \node[inner sep=0pt, label={[below]:{$v_{12}$}}] at (D12) {};
\end{scope}

\begin{scope}
  \coordinate (E7) at (529.5,220);
  \coordinate (E8) at (549,170);
  \coordinate (E9) at (568.5,220);

  \draw[fig/blueedge] (E8)--(E7);
  \draw[fig/blueedge] (E8)--(E9);

  \path[fig/v] (E7) circle[radius=\PointR];
  \path[fig/v] (E8) circle[radius=\PointR];
  \path[fig/v] (E9) circle[radius=\PointR];

  \node[inner sep=0pt, label={[below]:{$v_{7}$}}] at (E7) {};
  \node[inner sep=0pt, label={[above]:{$v_{8}$}}] at (E8) {};
  \node[inner sep=0pt, label={[below]:{$v_{9}$}}] at (E9) {};
\end{scope}

\begin{scope}
  \coordinate (F7)  at (567.8,30);
  \coordinate (F8)  at (586.8,90);
  \coordinate (F9)  at (606.3,30);
  \coordinate (F10) at (625.8,90);

  \draw[fig/edge]     (F7)--(F8);
  \draw[fig/blueedge] (F9)--(F8);
  \draw[fig/blueedge] (F9)--(F10);

  \path[fig/v] (F7)  circle[radius=\PointR];
  \path[fig/v] (F8)  circle[radius=\PointR];
  \path[fig/v] (F9)  circle[radius=\PointR];
  \path[fig/v] (F10) circle[radius=\PointR];

  \node[inner sep=0pt, label={[above]:{$v_{7}$}}]  at (F7)  {};
  \node[inner sep=0pt, label={[below]:{$v_{8}$}}]  at (F8)  {};
  \node[inner sep=0pt, label={[above]:{$v_{9}$}}]  at (F9)  {};
  \node[inner sep=0pt, label={[below]:{$v_{10}$}}] at (F10) {};
\end{scope}

\end{tikzpicture}

\endgroup
\caption{Illustrations for $(\romannumeral5)$ showing the existence of a $2$-edge path:
blue edges indicate the position of the $2$-edge path.}
\label{maintheoremfig4}
\end{figure}

Suppose that $(\romannumeral5)$ of Theorem~\ref{main-theorem:triangle} holds, and so $G''$ contains an induced cycle $C$ not containing $x,y'$ and containing a vertex $r$ such that every vertex in $V(C)\setminus N[r]$ has degree two in  $G''$.
Since $|V(C)|\ge 12$, by a similar analysis of the case $(\romannumeral4)$, it follows that $C$ contains a path $P=p_1p_2p_3$ of order three such that $N(s)\cap V(P)=\{p_2\}$, and all vertices of $P$ have degree two in $G''$ and are non-adjacent to $x$ and $y'$.
  Since $x^*$ is not in a 4-cycle in $G'$ by Claim~\ref{claim:no-4-cycle-z}, not both $p_1$ and $p_3$ have a neighbor in $N_k$.   
  Therefore,  $\{p_1,p_2,p_3\}$ is anticomplete to $\{x,y,z\}$. It follows that either $p_1$ or $p_3$ has degree two in $G$, a contradiction.

Therefore, $(\romannumeral6)$--$(\romannumeral10)$ of Theorem~\ref{main-theorem:triangle} holds. This implies that there exists an induced cycle $C$ in $G''$ with even length containing $y'$
such that at most one vertex $v$ in $V(C)\setminus N[y']$ has $\deg_{G''}(v)>2$; or there exists
an induced cycle $C$ in $G''$ with even length containing neither $x$ nor $y'$ and a vertex
$r\in V(C)$ such that at most one vertex $v$ in $V(C)\setminus N[r]$ has $\deg_{G''}(v)>2$. Note that $G''$ has girth at least $7$ and $G''-y'$ has girth at least $12$. 

If $C$ contains $y'$, then $C$ contains a vertex $v \notin N(s)$ at distance three from $y'$ on $C$ and of degree two in $G''$ (see Figure~\ref{fig-y'}).
Let $y'abv$ be the three-edge path from $y'$ to $v$ on $C$.
Moreover, $v$ is anticomplete to $N_k$; otherwise, $x^*abvx^*$ forms a $4$-cycle in $G'$, contradicting Claim~\ref{claim:no-4-cycle-z}.
Thus $v$ has degree two in $G$ and is not contained in $N[x] \cup N[y] \cup N[z]$, a contradiction.

\begin{figure}[!ht]
\centering
\begingroup

\begin{tikzpicture}[x=0.5pt,y=0.5pt,yscale=-1,xscale=1]
\tikzset{
  figY/edge/.style     ={line width=1.5, draw=black},
  figY/blueedge/.style ={figY/edge, draw=blue},
  figY/v/.style        ={draw=black, fill=black},
  figY/vblue/.style    ={draw=blue, fill=blue},
  figY/vred/.style     ={draw=red,  fill=red},
}

\def\PointR{4}

\begin{scope}
  \coordinate (x) at (118,228);
  \coordinate (u) at (171,228);
  \coordinate (v) at (220,228);
  \coordinate (r) at (272,228);
  \coordinate (z) at (319,228);

  \coordinate (y) at (145,155);
  \coordinate (p) at (196,155);
  \coordinate (w) at (245,155);
  \coordinate (t) at (296,155);

  \coordinate (s) at (225,100);

  \draw[figY/edge] (x)--(y);
  \draw[figY/edge] (u)--(y);

  \draw[figY/edge] (p)--(u);
  \draw[figY/edge] (p)--(v);

  \draw[figY/edge] (v)--(w);
  \draw[figY/edge] (w)--(r);

  \draw[figY/edge] (r)--(t);
  \draw[figY/edge] (t)--(z);
  \draw[figY/edge] (z)--(y);

  \draw[figY/edge] (p)--(s);
  \draw[figY/edge] (s)--(y);
  \draw[figY/edge] (s)--(w);
  \draw[figY/edge] (s)--(t);

  \draw[figY/blueedge] (y)--(z);
  \draw[figY/blueedge] (z)--(t);
  \draw[figY/blueedge] (t)--(r);

  \path[figY/v] (x) circle[radius=\PointR];
  \path[figY/v] (u) circle[radius=\PointR];
  \path[figY/v] (y) circle[radius=\PointR];
  \path[figY/v] (p) circle[radius=\PointR];
  \path[figY/v] (v) circle[radius=\PointR];
  \path[figY/v] (w) circle[radius=\PointR];
  \path[figY/v] (r) circle[radius=\PointR];
  \path[figY/v] (t) circle[radius=\PointR];
  \path[figY/v] (z) circle[radius=\PointR];
  \path[figY/v] (s) circle[radius=\PointR];

  \path[figY/vblue] (v) circle[radius=\PointR];
  \path[figY/vred]  (r) circle[radius=\PointR];

  \node[inner sep=0pt, label={[above, yshift=-0.5mm]:{$s$}}] at (s) {};
  \node[inner sep=0pt, label={[left,  xshift=-0.5mm]:{$x$}}] at (x) {};
  \node[inner sep=0pt, label={[above, yshift=-0.5mm]:{$y'$}}] at (y) {};
\end{scope}

\begin{scope}[xshift=200]
  \coordinate (x) at (118,228);
  \coordinate (u) at (171,228);
  \coordinate (v) at (220,228);
  \coordinate (r) at (272,228);
  \coordinate (z) at (319,228);

  \coordinate (y) at (145,155);
  \coordinate (p) at (196,155);
  \coordinate (w) at (245,155);
  \coordinate (t) at (296,155);

  \coordinate (s) at (225,100);

  \draw[figY/edge] (x)--(y);
  \draw[figY/edge] (u)--(y);

  \draw[figY/edge] (p)--(u);
  \draw[figY/edge] (p)--(v);

  \draw[figY/edge] (v)--(w);
  \draw[figY/edge] (w)--(r);

  \draw[figY/edge] (r)--(t);
  \draw[figY/edge] (t)--(z);
  \draw[figY/edge] (z)--(y);

  \draw[figY/edge] (p)--(s);
  \draw[figY/edge] (s)--(y);
  \draw[figY/edge] (s)--(w);
  \draw[figY/edge] (s)--(t);

  \draw[figY/blueedge] (y)--(u);
  \draw[figY/blueedge] (u)--(p);
  \draw[figY/blueedge] (p)--(v);

  \path[figY/v] (x) circle[radius=\PointR];
  \path[figY/v] (u) circle[radius=\PointR];
  \path[figY/v] (y) circle[radius=\PointR];
  \path[figY/v] (p) circle[radius=\PointR];
  \path[figY/v] (v) circle[radius=\PointR];
  \path[figY/v] (w) circle[radius=\PointR];
  \path[figY/v] (r) circle[radius=\PointR];
  \path[figY/v] (t) circle[radius=\PointR];
  \path[figY/v] (z) circle[radius=\PointR];
  \path[figY/v] (s) circle[radius=\PointR];

  \path[figY/vblue] (r) circle[radius=\PointR];
  \path[figY/vred]  (v) circle[radius=\PointR];

  \node[inner sep=0pt, label={[above, yshift=-0.5mm]:{$s$}}] at (s) {};
  \node[inner sep=0pt, label={[left,  xshift=-0.5mm]:{$x$}}] at (x) {};
  \node[inner sep=0pt, label={[above, yshift=-0.5mm]:{$y'$}}] at (y) {};
\end{scope}

\end{tikzpicture}

\endgroup
\caption{Illustrations for $(\romannumeral6)$--$(\romannumeral10)$ showing the existence of a $3$-edge path:
blue edges indicate the position of the $3$-edge path,
blue vertices indicate possible degree-$3$ vertices,
and red vertices indicate degree-$2$ vertices.}
\label{fig-y'}
\end{figure}
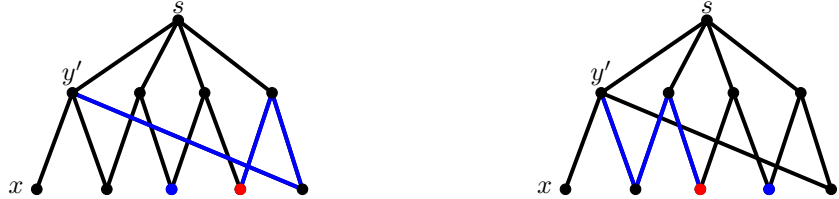

Therefore, $C$ contains neither $x$ nor $y'$, and there exists a vertex
$r \in V(C)$ such that at most one vertex $v \in V(C) \setminus N[r]$ satisfies $\deg_{G''}(v) > 2$.
Since $|V(C)| \geq 12$, it follows that $C$ contains a path $P = p_1p_2p_3$ of length two such that
$N(s) \cap V(P) = \{p_2\}$, all vertices of $P$ have degree two in $G''$, and $P$ is anticomplete to $\{x,y'\}$ (see Figure~\ref{fig-r}).
Since $x^*$ is not contained in any $4$-cycle of $G^*$ by Claim~\ref{claim:no-4-cycle-z}, not both $p_1$ and $p_3$ have a neighbor in $N_k$.
Therefore, $\{p_1,p_2,p_3\}$ is anticomplete to $\{x,y,z\}$.
It follows that either $p_1$ or $p_3$ has degree two in $G$, a contradiction.

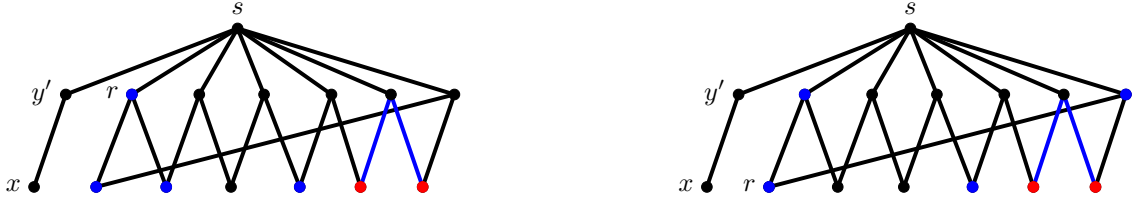
\begin{figure}[H]
\centering
\begingroup

\begin{tikzpicture}[x=0.5pt,y=0.5pt,yscale=-1,xscale=1]
\tikzset{
  fig/edge/.style     ={line width=1.5, draw=black},
  fig/blueedge/.style ={fig/edge, draw=blue},
  fig/v/.style        ={draw=black, fill=black},
  fig/vblue/.style    ={draw=blue,  fill=blue},
  fig/vred/.style     ={draw=red,   fill=red},
}

\def\PointR{4}

\newdimen\FigSep
\FigSep=80pt

\begin{scope}[xshift=-\dimexpr\FigSep/2\relax]
  \coordinate (x)   at (24,240);
  \coordinate (yp)  at (48,170);     
  \coordinate (a)   at (71,240);     
  \coordinate (r)   at (98,170);     
  \coordinate (c)   at (124,240);    
  \coordinate (d)   at (149,170);
  \coordinate (e)   at (173,240);
  \coordinate (s)   at (178,120);     
  \coordinate (f)   at (198,170);
  \coordinate (g)   at (225,240);    
  \coordinate (h)   at (249,170);
  \coordinate (r2)  at (271,240);    
  \coordinate (m)   at (294,170);    
  \coordinate (r3)  at (318,240);    
  \coordinate (t)   at (342,170);    

  \draw[fig/edge] (x)--(yp);
  \draw[fig/edge] (s)--(yp);
  \draw[fig/edge] (s)--(r);
  \draw[fig/edge] (s)--(d);
  \draw[fig/edge] (a)--(r);
  \draw[fig/edge] (c)--(r);

  \draw[fig/edge] (c)--(d);
  \draw[fig/edge] (d)--(e);
  \draw[fig/edge] (e)--(f);

  \draw[fig/edge] (f)--(g);
  \draw[fig/edge] (g)--(h);
  \draw[fig/edge] (h)--(r2);

  \draw[fig/edge] (s)--(f);
  \draw[fig/edge] (s)--(h);
  \draw[fig/edge] (s)--(m);
  \draw[fig/edge] (s)--(t);

  \draw[fig/edge] (a)--(t);
  \draw[fig/edge] (r3)--(t);

  \draw[fig/blueedge] (r2)--(m);
  \draw[fig/blueedge] (m)--(r3);

  \path[fig/v] (x)  circle[radius=\PointR];
  \path[fig/v] (yp) circle[radius=\PointR];
  \path[fig/v] (a)  circle[radius=\PointR];
  \path[fig/v] (r)  circle[radius=\PointR];
  \path[fig/v] (c)  circle[radius=\PointR];
  \path[fig/v] (d)  circle[radius=\PointR];
  \path[fig/v] (e)  circle[radius=\PointR];
  \path[fig/v] (s)  circle[radius=\PointR];
  \path[fig/v] (f)  circle[radius=\PointR];
  \path[fig/v] (g)  circle[radius=\PointR];
  \path[fig/v] (h)  circle[radius=\PointR];
  \path[fig/v] (r2) circle[radius=\PointR];
  \path[fig/v] (m)  circle[radius=\PointR];
  \path[fig/v] (r3) circle[radius=\PointR];
  \path[fig/v] (t)  circle[radius=\PointR];

  \path[fig/vblue] (a) circle[radius=\PointR];
  \path[fig/vblue] (c) circle[radius=\PointR];
  \path[fig/vblue] (g) circle[radius=\PointR];
  \path[fig/vblue] (r) circle[radius=\PointR];

  \path[fig/vred] (r2) circle[radius=\PointR];
  \path[fig/vred] (r3) circle[radius=\PointR];

  \node[inner sep=0pt, label={[left,  xshift=-0.5mm]:{$x$}}]  at (x)  {};
  \node[inner sep=0pt, label={[left,  xshift=-0.5mm]:{$y'$}}] at (yp) {};
  \node[inner sep=0pt, label={[above, yshift=0.5mm]:{$s$}}]  at (s)  {};
  \node[inner sep=0pt, label={[left, xshift=-0.5mm]:{$r$}}]   at (r)  {};
\end{scope}

\begin{scope}[xshift=\dimexpr\FigSep/2\relax]
  \coordinate (x)   at (373,240);
  \coordinate (yp)  at (397,170);    
  \coordinate (a)   at (420,240);    
  \coordinate (r)   at (447,170);    
  \coordinate (c)   at (472,240);
  \coordinate (d)   at (498,170);
  \coordinate (e)   at (522,240);
  \coordinate (s)   at (527,120);     
  \coordinate (f)   at (547,170);
  \coordinate (g)   at (574,240);    
  \coordinate (h)   at (598,170);
  \coordinate (r2)  at (620,240);    
  \coordinate (m)   at (643,170);
  \coordinate (r3)  at (667,240);    
  \coordinate (t)   at (690,170);    

  \draw[fig/edge] (x)--(yp);
  \draw[fig/edge] (s)--(yp);
  \draw[fig/edge] (s)--(r);
\draw[fig/edge] (s)--(d);
  \draw[fig/edge] (a)--(r);
  \draw[fig/edge] (c)--(r);

  \draw[fig/edge] (d)--(c);
  \draw[fig/edge] (d)--(e);
  \draw[fig/edge] (e)--(f);

  \draw[fig/edge] (f)--(g);
  \draw[fig/edge] (g)--(h);
  \draw[fig/edge] (h)--(r2);

  \draw[fig/edge] (s)--(f);
  \draw[fig/edge] (s)--(h);
  \draw[fig/edge] (s)--(m);
  \draw[fig/edge] (s)--(t);

  \draw[fig/edge] (a)--(t);
  \draw[fig/edge] (r3)--(t);

  \draw[fig/blueedge] (r2)--(m);
  \draw[fig/blueedge] (m)--(r3);

  \path[fig/v] (x)  circle[radius=\PointR];
  \path[fig/v] (yp) circle[radius=\PointR];
  \path[fig/v] (a)  circle[radius=\PointR];
  \path[fig/v] (r)  circle[radius=\PointR];
  \path[fig/v] (c)  circle[radius=\PointR];
  \path[fig/v] (d)  circle[radius=\PointR];
  \path[fig/v] (e)  circle[radius=\PointR];
  \path[fig/v] (s)  circle[radius=\PointR];
  \path[fig/v] (f)  circle[radius=\PointR];
  \path[fig/v] (g)  circle[radius=\PointR];
  \path[fig/v] (h)  circle[radius=\PointR];
  \path[fig/v] (r2) circle[radius=\PointR];
  \path[fig/v] (m)  circle[radius=\PointR];
  \path[fig/v] (r3) circle[radius=\PointR];
  \path[fig/v] (t)  circle[radius=\PointR];

  \path[fig/vblue] (a) circle[radius=\PointR];
  \path[fig/vblue] (g) circle[radius=\PointR];
  \path[fig/vblue] (r) circle[radius=\PointR];
  \path[fig/vblue] (t) circle[radius=\PointR];

  \path[fig/vred] (r2) circle[radius=\PointR];
  \path[fig/vred] (r3) circle[radius=\PointR];

  \node[inner sep=0pt, label={[left,  xshift=-0.5mm]:{$x$}}]  at (x)  {};
  \node[inner sep=0pt, label={[left,  xshift=-0.5mm]:{$y'$}}] at (yp) {};
  \node[inner sep=0pt, label={[above, yshift=0.5mm]:{$s$}}]  at (s)  {};
  \node[inner sep=0pt, label={[left, xshift=-0.5mm]:{$r$}}]   at (a)  {};
\end{scope}

\end{tikzpicture}

\endgroup
\caption{Illustrations for $(\romannumeral6)$--$(\romannumeral10)$ showing the existence of a $2$-edge path:
blue edges indicate the position of the $2$-edge path,
blue vertices indicate possible degree-$3$ vertices and special vertices,
and red vertices indicate possible degree-$2$ vertices.}
\label{fig-r}
\end{figure}
  
  This completes the proof of Theorem~\ref{thm:property-p}.
\end{proof}

\begin{corollary}\label{cor:degree-2}
Every   $\iskdbpk$-free graph  without a clique cutset contains a vertex of degree at most two.
\end{corollary}
\begin{proof}
 If $G$ is seires-parallel, then $G$ contains a vertex of degree at most two. If $G$ is not series-parallel, then $G$ contains a vertex of degree at most two by Theorem~\ref{thm:property-p} applied to the graph obtaining from $G$ by adding an isolated vertex $x$ with the noncenter triple $(x,x,x)$.
 This completes the proof of Corollary~\ref{cor:degree-2}.
\end{proof}

\medskip

We can now prove Theorem~\ref{Main-theorem-structure}, which we restate:
\begin{theorem}\label{Main-theorem-structure-2}
	Let $G$ be an $\iskdb$-free graph. Then $G$ is either series-parallel, a complete bipartite graph, the line graph of a sparse graph with maximum degree at most three, a graph with a clique cutset or a proper $2$-cutset, or a graph having a vertex of degree at most two.
Furthermore, if $G$ is prism-free, then $G$ is either series-parallel, a complete bipartite graph, a graph with a clique cutset, or a graph having a vertex of degree at most two. 
\end{theorem}
\begin{proof}
Let $G$ be an $\iskdb$-free graph. If $G$ contains $K_{3,3}$ as an induced subgraph, then by Lemma~\ref{lemma: K33}, either $G$ is a complete bipartite or complete tripartite graph, or $G$ has a clique cutset of size at most three. Since $G$ is diamond-free, if $G$ is a complete tripartite graph, it must be a triangle. Therefore, $G$ contains a vertex of degree two.  If $G$ is an $\iskdbpk$-free graph without a clique cutset, then by Corollary~\ref{cor:degree-2}, $G$ contains a vertex of degree at most two.

If $G$ contains a prism as an induced subgraph, then by Lemma~\ref{lemma: prism}, either $G$ is the line graph of a sparse graph with maximum degree $3$, or $G$ has a clique cutset of size at most $3$, or $G$ has a proper $2$-cutset.
 This completes the proof of Theorem~\ref{Main-theorem-structure-2}.
\end{proof}

Next, we  prove Theorem~\ref{Main-theorem-color}, which we restate:
\begin{theorem}\label{Main-theorem-color-1}
	If $G$ is $\iskdb$-free graph, then $G$ is $3$-colorable.
\end{theorem}
\begin{proof}
The proof is by induction on $|V(G)|$ using Theorem~\ref{Main-theorem-structure}. Then $G$ is connected. If $G$ is a complete bipartite graph, then $G$ is 2-colorable. If $G$ has a vertex $v$ of degree at most two, then by induction, $G \setminus v$ is 3-colorable, and hence $G$ is 3-colorable. If $G$ has a clique cutset $C$ such that $(A, C, B)$ is a separation of $G$ with $A$ anticomplete to $B$ and $C$ a clique, then $\chi(G) = \max\{\chi(G[A \cup C]), \chi(G[B \cup C])\}$, and again by induction, $G$ is 3-colorable.  If $G$ is the line graph of a sparse graph, then  $G$ has maximum degree three. Note that $G$ is $K_4$-free. By Brooks' theorem, we have $\chi(G)\leq 3$. Therefore, we may assume that $G$ is $K_{3,3}$-free, has no clique cutset, has minimum degree at least three, and has a proper 2-cutset. This implies that $G$ is $2$-connected.

By Theorem~\ref{Main-theorem-structure-2}, we have that $G$ has a proper 2-cutset. Let $T=\{a,b\}$ be a proper $2$-cutset in $G$. Then $V(G)\setminus T$ can be partitioned into non-empty sets $X_T$ and $Y_T$ such that there are no edges between them. Without loss of generality, we may assume that $|X_T|\leq |Y_T|$. Among all proper 2-cutsets in $G$, we choose $T$ such that $|X_T|$ is minimum. This implies that either $G[X_T]$ is connected, or $G[X_T]$ has exactly two components and $G[X_T\cup T]$ is a cycle. 

Note that $G[T\cup Y_T]$ is $\iskdb$-free. By the induction hypothesis, it is $3$-colorable. Let $c$ be a $3$-coloring of $G[T\cup Y_T]$. Suppose that $G[X_T]$ has exactly two components and $G[X_T\cup T]$ is a cycle. It is clear that  there exist two 3-coloring $c_1,c_2$ of $G[X_T\cup T]$ such that $c_1(a)=c_1(b)$ and $c_2(a)\neq c_2(b)$. We can combine $c_1'$ and $c$ (if $c(a)=c(b)$) or $c_2'$ and $c$ (if $c(a)\neq c(b)$) such that they coincide on $T$, and hence obtain a $3$-coloring of $G$. Therefore, we further assume that $G[X_T]$ is connected.

Let $G'$ be the graph obtained from $G$ by adding an extra vertex $u$ and  two edges $au$ and $bu$. Let $P_Y$ be an induced path with ends $a$ and $b$, and internal vertices in $Y$. Clearly, $G'$ is $\isk$-free. Indeed, if $G'$ contains an $\isk$, say $H$. Since $G[X_T \cup T]$ is $\isk$-free, we have $u \in V(H)$. Since $u$ has degree two in $G$, the path $aub$ is a subpath of some branch of $H$. But then, $G[V(H-u)\cup V(P_Y)]$ is an $\isk$ in $G$, a contradiction. Note  that $u$ has degree two and $G$ is $\iskdbk$-free. Therefore, $G'$ is $\iskdbk$-free.

\begin{claim}\label{clique-proper2}
  $G'$ has neither a clique cutset nor a proper $2$-cutset.
\end{claim}

Suppose that $G'$ has a clique cutset $S$  such that $(A, B, S)$ is a separation of $V(G')$ with $A$ anticomplete to $B$.  Clearly, $u\notin S$. That is, $S\subseteq X_T\cup T$.  Without loss of generality, we may assume that $\{u,a,b\}\subseteq A\cup S$.  However, it follows that $S$ is a clique cutset of $G$ with  $(A\cup Y_T, B, S)$ is a separation of $V(G)$ with $A\cup Y_T$ anticomplete to $B$, a contradiction.

Suppose that $G'$ has a proper 2-cutset $T'=\{a',b'\}$ such that $(A', B', T')$ is a separation of $V(G')$ with $A'$ anticomplete to $B'$ and neither $G'[A'\cup T']$ nor $G'[B'\cup T']$ is a path between $a'$ and $b'$. Without loss of generality, we may assume that $|A'|\leq |B'|$. Since $G'$ has no clique cutset, $G'$ is 2-connected. Suppose $u \in T'$, and further assume that $u = a'$.
 This implies that  $T\cap A'\neq \emptyset$ and  $T\cap  B'\neq \emptyset$. Without loss of generality, we may assume that $a\in A'$ and $b\in B'$. 
 Note that $G'$ has no clique cutset, $ab'\notin E(G')$. Since $G'[A'\cup T']$  is not a path between $a'$ and $b'$,  $G'[A'\cup \{b'\}]$ is not a path between $a$ and $b'$. 
 However, it follows that $\{a,b'\}$ is a proper 2-cutset with $|A'\setminus \{a\}|<|X_T|$, a contradiction to the choice of $T$. Therefore, $u\notin T'$. 
 Without loss of generality, we may assume that $u\in A'$. But now, $\{a', b'\}$ is a proper 2-cutset with $|B'| < |X_T|$, which again contradicts the choice of $T$. 
 Therefore, $G'$ has no proper 2-cutset. This proves Claim~\ref{clique-proper2}.

 We now distinguish two cases according to whether $G'$ contains a prism or not.
 \begin{case}
    $G'$ is prism-free.
 \end{case}
 
 Now, $G'$ is an $\iskdbpk$-free graph without a clique cutset.  If $G'$ is not a series-parallel graph, then by Theorem~\ref{thm:property-p}, we have that $(G',u)$ has property $\mathcal{P}$. However, it follows that there exists $x \in X_T$ with degree at most two in $G'$, and hence $x$ has degree at most two in $G$, a contradiction. Therefore, $G'$ is a series-parallel graph. Suppose that $G'-u$ is a forest.  Clearly, $(\romannumeral1)$, $(\romannumeral2)$, $(\romannumeral4)$ and $(\romannumeral5)$ of Lemma~\ref{lem:tree} do not hold. Therefore, $G'$ has an induced cycle $C$ containing $u$, and $V(C) \setminus \{u\}$ has at most one vertex with degree greater than two in $G'$. Moreover, if such a vertex exists, it is a neighbor of $u$ on $C$. Since $ab\notin E(G)$, $C$ has length at least four. Then there exists $x \in X_T\cap V(C)$ with degree two in $G'$, and hence $x$ has degree at most two in $G$, a contradiction. Therefore, $G'-u$ contains a cycle. Clearly,  $(\romannumeral2)$ of Lemma~\ref{lem:farcycle} do not hold. Therefore, it follows from Lemma~\ref{lem:farcycle}  that there is an induced cycle $C$ in $G'$ such that $u\notin V(C)$ and all but at most two vertices of $C$ have degree two in $G'$. Let $x \in V(C)$ be of degree two in $G'$. Since $u \notin V(C)$, we have $x \notin \{a, b\}$. This implies that $x \in X_T$. Therefore, $x$ has degree at most two in $G$, a contradiction.

 \begin{case}
     $G'$ contains a prism.
 \end{case}
 
First suppose that $a$ has degree two in $G'$. That is, $a$ has exactly one neighbor $a'$ in $X_T$. Since $G'$ has no clique cutset, $a'b\notin E(G)$. However, it follows that $\{a',b\}$ is a proper $2$-cutset with $|X_T\setminus \{a'\}|<|X_T|$, a contradiction to the choice of $T$. Therefore, $a$ has degree at least three in $G'$. Similarly, $b$ has degree at least three in $G'$. 

Suppose $G$ does not contain the line graph of a substantial graph as an induced graph.  By Lemma~\ref{lem:no-substantial-graph}, we have that $G'$ is a prism. Note that $u$ is the unique vertex of degree two in $G'$. Therefore, $G'$ has order seven (see Figure~\ref{fig:seven-prism} for a description). It is clear that  there exist two 3-coloring $c_1,c_2$ of $G[X_T\cup T]$ such that $c_1(a)=c_1(b)$ and $c_2(a)\neq c_2(b)$ (see figure~\ref{fig:seven-prism} for a description). We can combine $c_1'$ and $c$ (if $c(a)=c(b)$) or $c_2'$ and $c$ (if $c(a)\neq c(b)$) such that they coincide on $T$, and hence obtain a $3$-coloring of $G$. 
 Therefore, $G$  contains the line graph of a substantial graph $H$ as an induced graph. By Lemma~\ref{lem:-substantial-graph} and Claim~\ref{clique-proper2}, we have that $G'=L(H)$. 
 Since $G'$ is $\db$-free graph, $H$ is a sparse graph. 
 By Lemma~\ref{lem:cyclically-3-connected} and $G'$ is not a prism, we have that $H$ is a subdivision of a $3$-connected graph $H'$.
 Since $G'$ is $\iskfour$-free, $H'$ is a cubic graph. Since $u$ is the unique vertex of degree two in $G'$ and $H$ is a sparse graph, $H'$ has exactly one edge subdivided twice, and every other edge is subdivided exactly once.  By Lemma~\ref{lemma:cubic-edge-color},  there exist two 3-coloring $c_1,c_2$ of $G[X_T\cup T]$ such that $c_1(a)=c_1(b)$ and $c_2(a)\neq c_2(b)$. 
 We can combine $c_1'$ and $c$ (if $c(a)=c(b)$) or $c_2'$ and $c$ (if $c(a)\neq c(b)$) such that they coincide on $T$, and hence obtaining a $3$-coloring of $G$.  
 This completes the proof of Theorem~\ref{Main-theorem-color-1}.
\end{proof}

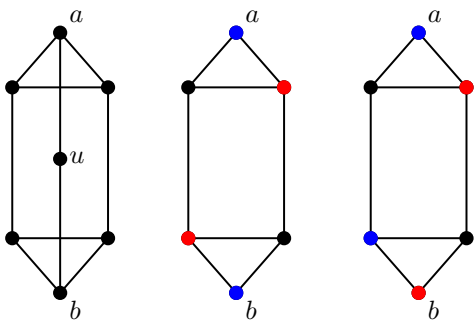
\begin{figure}[H]
\centering
\begingroup

\begin{tikzpicture}[x=0.75pt,y=0.75pt,yscale=-1,xscale=1,scale=0.69]
\tikzset{
  prism/edge/.style  ={line width=0.75pt, draw=black},
  prism/v/.style     ={draw=black, fill=black, line width=0.75pt},
  prism/vblue/.style ={draw=blue,  fill=blue,  line width=0.75pt},
  prism/vred/.style  ={draw=red,   fill=red,   line width=0.75pt},
}
\def\PointR{3.35pt}

\begin{scope}
  \coordinate (T)  at (199.33,56.83);
  \coordinate (TL) at (164.33,96.83);
  \coordinate (TR) at (234.33,96.83);

  \coordinate (B)  at (199.33,247.17);
  \coordinate (BL) at (164.33,207.17);
  \coordinate (BR) at (234.33,207.17);

  \coordinate (U)  at (199.33,149.17);

  \draw[prism/edge] (T)--(TL)--(TR)--cycle;
  \draw[prism/edge] (B)--(BL)--(BR)--cycle;

  \draw[prism/edge] (TL)--(BL);
  \draw[prism/edge] (TR)--(BR);

  \draw[prism/edge] (T)--(U);
  \draw[prism/edge] (U)--(B);

  \path[prism/v] (T)  circle[radius=\PointR];
  \path[prism/v] (TL) circle[radius=\PointR];
  \path[prism/v] (TR) circle[radius=\PointR];
  \path[prism/v] (U)  circle[radius=\PointR];
  \path[prism/v] (B)  circle[radius=\PointR];
  \path[prism/v] (BL) circle[radius=\PointR];
  \path[prism/v] (BR) circle[radius=\PointR];

  \node[inner sep=0pt, label={[above right]:{$a$}}] at (T) {};
  \node[inner sep=0pt, label={[right]:{$u$}}]       at (U) {};
  \node[inner sep=0pt, label={[below right]:{$b$}}] at (B) {};
\end{scope}

\begin{scope}
  \coordinate (T)  at (328,56.83);
  \coordinate (TL) at (293,96.83);
  \coordinate (TR) at (363,96.83);

  \coordinate (B)  at (328,247.17);
  \coordinate (BL) at (293,207.17);
  \coordinate (BR) at (363,207.17);

  \draw[prism/edge] (T)--(TL)--(TR)--cycle;
  \draw[prism/edge] (B)--(BL)--(BR)--cycle;

  \draw[prism/edge] (TL)--(BL);
  \draw[prism/edge] (TR)--(BR);

  \path[prism/v] (T)  circle[radius=\PointR];
  \path[prism/v] (TL) circle[radius=\PointR];
  \path[prism/v] (TR) circle[radius=\PointR];
  \path[prism/v] (B)  circle[radius=\PointR];
  \path[prism/v] (BL) circle[radius=\PointR];
  \path[prism/v] (BR) circle[radius=\PointR];

  \path[prism/vblue] (T)  circle[radius=\PointR];
  \path[prism/vblue] (B)  circle[radius=\PointR];
  \path[prism/vred]  (TR) circle[radius=\PointR];
  \path[prism/vred]  (BL) circle[radius=\PointR];

  \node[inner sep=0pt, label={[above right]:{$a$}}] at (T) {};
  \node[inner sep=0pt, label={[below right]:{$b$}}] at (B) {};
\end{scope}

\begin{scope}
  \coordinate (T)  at (461.33,56.83);
  \coordinate (TL) at (426.33,96.83);
  \coordinate (TR) at (496.33,96.83);

  \coordinate (B)  at (461.33,247.17);
  \coordinate (BL) at (426.33,207.17);
  \coordinate (BR) at (496.33,207.17);

  \draw[prism/edge] (T)--(TL)--(TR)--cycle;
  \draw[prism/edge] (B)--(BL)--(BR)--cycle;

  \draw[prism/edge] (TL)--(BL);
  \draw[prism/edge] (TR)--(BR);

  \path[prism/v] (T)  circle[radius=\PointR];
  \path[prism/v] (TL) circle[radius=\PointR];
  \path[prism/v] (TR) circle[radius=\PointR];
  \path[prism/v] (B)  circle[radius=\PointR];
  \path[prism/v] (BL) circle[radius=\PointR];
  \path[prism/v] (BR) circle[radius=\PointR];

  \path[prism/vblue] (T)  circle[radius=\PointR];
  \path[prism/vblue] (BL) circle[radius=\PointR];
  \path[prism/vred]  (TR) circle[radius=\PointR];
  \path[prism/vred]  (B)  circle[radius=\PointR];

  \node[inner sep=0pt, label={[above right]:{$a$}}] at (T) {};
  \node[inner sep=0pt, label={[below right]:{$b$}}] at (B) {};
\end{scope}

\end{tikzpicture}

\endgroup
\caption{$G'$, $c_1(a)=c_1(b)$, $c_1(a)\neq c_1(b)$}
\label{fig:seven-prism}
\end{figure}

\section{\texorpdfstring{Decomposition and coloring algorithm for $\iskdb$-free graphs}{Decomposition and coloring algorithm for $\iskdb$-free graphs}}
In this section, we present a decomposition algorithm for the class of $\iskdb$-free graphs, based on the structural results established in the previous sections. 
Let $G$ be an graph with $n$ vertices and $m$ edges. Our goal is to decompose $G$ in polynomial time. This is achieved by repeatedly applying clique cutsets and proper $2$-cutsets, and by removing vertices of degree at most two.

Suppose that $G$ has a clique cutset $K$. Then $V(G) \setminus K$ can be partitioned into two nonempty sets $X$ and $Y$ such that there are no edges between $X$ and $Y$. Let $G_X$ and $G_Y$ denote the subgraphs of $G$ induced by $V[X \cup K]$ and $V[Y \cup K]$, respectively. Then we say that $G$ is decomposed into $G_X$ and $G_Y$. These subgraphs may themselves be further decomposed by their clique cutsets.

\vspace{-8pt}

\begin{definition}
    A \textit{clique cutset decomposition tree} of $ G $, denoted by $T$, is defined as follows:
 \begin{itemize}[nosep]
     \item[$(\romannumeral1)$]  Each node of $T$ is an induced subgraph of $ G $.    
     \item[$(\romannumeral2)$]  The root of $T$ is $ G $.
    \item[$(\romannumeral3)$]  For any non-leaf node $ X \subseteq V $ of $T$:
    \begin{itemize}[nosep]
         \item[$(\romannumeral3.1)$]  If $ X $ contains a clique cutset $ K $, then $ X \setminus K $ has at least two connected components, say $ C_1, \dots, C_r $.
         \item[$(\romannumeral3.2)$]  The children of $ X $ in $T$ are the vertex sets $ X_i = C_i \cup K $ for $ i = 1, \dots, r $.
    \end{itemize}
\end{itemize}\vspace{-8pt}

\medskip

\end{definition}

Note that a clique cutset decomposition tree $T$ can be constructed in $O(nm)$ time (see~\cite{Tarjan}).
Now suppose that $G$ has no clique cutset but admits a proper $2$-cutset $\{a, b\}$. Then $V(G) \setminus \{a, b\}$ can be partitioned into two nonempty sets $X$ and $Y$ such that there are no edges between $X$ and $Y$, and neither $G[X \cup \{a, b\}]$ nor $G[Y \cup \{a, b\}]$ is a path with ends $a,b$.
The decomposition by proper $2$-cutsets can be constructed in time $O(n^2m)$ by~\cite{Leveque2012}.
We start with Algorithm~\ref{alg:decomp1} to perform a single decomposition.

\begin{algorithm}[H]
\caption{Construction of $D_l$ Sets from a Clique Cutset Decomposition Tree}
\label{alg:decomp1}
\begin{algorithmic}[1]
\Require An $\iskdb$-free graph $G$.
\Ensure A collection $\{D_1, \dots, D_t\}$  and a set $\mathcal{R}$ of removed vertices of degree at most two.

\State  Let $T$ be a clique cutset decomposition tree of $ G $ with leaves $\{A_1,A_2,\dots,A_t\}$ with $t\ge1$
\State $\mathcal{R}=\emptyset$
\For{each component $A_i (1\le i\le t)$}
    \State Let $R_i$ be the set of all vertices of degree at most two in $A_i$
    \State $D_i \gets A_i \setminus R_i$
    \State $\mathcal{R} \gets \mathcal{R} \cup R_{i}$  \Comment{Record all removed vertices}
\EndFor
\State \Return $\{D_1, \dots, D_t\}$, $\mathcal{R}$
\end{algorithmic}
\end{algorithm}





Then we present a decomposition algorithm that yields graphs with no clique cutset and no vertex of degree of at most two, which are called \textit{basic graphs}.

\begin{algorithm}[H]
\caption{Obtain Basic Graphs}
\label{alg:alg2}
\begin{algorithmic}[1]
\Require An $\iskdb$-free graph $G$.
\Ensure Basic graphs $\mathcal{H}$, the number of decomposition layer $j(G)$  and a set $\mathcal{R}$ of removed vertices of degree at most two.

\State $\mathcal{S}$ $ \gets $ $\{G\}$, $\mathcal{H}$ $ \gets $ $\emptyset$, $\mathcal{R}$ $ \gets$ $ \emptyset$
\State $j\gets0$
\While{$\mathcal{S} \neq \emptyset$}
    \State $j\gets j+1$
    \State $\mathcal{S}'$ $ \gets $ $\emptyset$
    \For{each $F$ in $\mathcal{S}$}
        \If{$F$ has a clique cutset or a vertex of degree at most two}
            \State $\{F_1^j, \dots, F_t^j\}, \mathcal{R_F}^j \gets$ Algorithm~\ref{alg:decomp1} applied to $F$
            \State $\mathcal{S}' \gets \mathcal{S}' \cup \{F_1^j, \dots, F_t^j\}$
            \State $\mathcal{R} \gets \mathcal{R} \cup \mathcal{R_F}^j$
        \Else
            \State $F^j\gets F$
            \State $\mathcal{H} \gets \mathcal{H} \cup \{F^j\}$
        \EndIf
    \EndFor
    \State $\mathcal{S} \gets \mathcal{S}'$
\EndWhile
\State $j(G)\gets j$
\State \Return $\mathcal{H}, \mathcal{R}, j(G)$
\end{algorithmic}
\end{algorithm}

The following algorithm handles the case when the graph has a proper 2-cutset.
Note that $\chi_X^0$ and $\chi_X^1$ in the  algorithm exist due to Theorem~\ref{Main-theorem-color-1}.

\begin{algorithm}[H]
\caption{Recursive Processing of a Basic Graph}
\label{alg:process-basic}
\begin{algorithmic}[1]
\Require Let $G $ be an $\iskdb$ graph with no clique cutset and minimum degree at least~3.
\Ensure Either a proper $3$-coloring $\{\chi_G\}$, or two graphs $\{T_X,T_Y\}$ and two colorings $\{\chi_X^0,\chi_X^1\}$ of $T_X$.

\If{$G$ is complete bipartite}
    \State $\chi_G \gets$ a proper $2$-coloring of $G$
    \State \Return  $\{\chi_G\}$
    
\ElsIf{$G$ is the line graph of a sparse graph}
    \State $\chi_G \gets$ a proper $3$-coloring of $G$
    \State \Return $\{\chi_G\}$
    
\Else \Comment{$G$ has a proper 2-cutset}
    \State Let $T = \{a, b\}$ be a proper 2-cutset in $G$
    \State Let $X_T, Y_T$ be a partition of $V(G) \setminus T$ with no edges between $X_T$ and $Y_T$
    \State Choose $T$ such that $|X_T|$ is minimum among all proper 2-cutsets
    
    \State $T_X \gets$ induced subgraph on $X_T \cup T$
    \State $T_Y \gets$ induced subgraph on $Y_T \cup T$
    
    \State \Comment{Find two colorings of $T_X$ with different conditions on $T$}
    \State $\chi_X^0 \gets$ a proper $3$-coloring of $T_X$ where $a$ and $b$ have the same color
    \State $\chi_X^1 \gets$ a proper $3$-coloring of $T_X$ where $a$ and $b$ have different colors
    
    \Return $\{T_X,T_Y,\chi_X^0,\chi_X^1\}$  \Comment{$T_Y$ needs further processing}
\EndIf
\end{algorithmic}
\end{algorithm}

Our main algorithm is as follows. By recursively applying the decomposition procedure, we ultimately obtain a proper $3$-coloring of $G$.
For each $H_i \in \mathcal{H}$ ($1 \leq i \leq k$), let $n_i = |V(H_i)|$.
It is easy to see that $\sum_{i=1}^k n_i = O(n)$. Thus Algorithm~\ref{alg:main-full} runs in $O(n^2 m)$ time.

\noindent\rule{\textwidth}{0.8pt}\vspace{-8pt}
\leavevmode
\captionof{algorithm}{ Coloring $\iskdb$-free Graphs by Decomposition}\vspace{-8pt}
\noindent\rule{\textwidth}{0.4pt}\par\vspace{-4pt}
\label{alg:main-full}
\begin{algorithmic}[1]
\Require An $\iskdb$-free graph $G$.
\Ensure A proper $3$-coloring $\chi (G)$ of $G$.

\State $\mathcal{H},\mathcal{R},j \gets$ apply Algorithm~\ref{alg:alg2} to $G$
\State Initialize $\mathcal{T} \gets \emptyset$  \Comment{$\mathcal{T}$ is used to store pairs $(T_X, \chi_X)$}
\State Initialize $\mathcal{U} \gets \emptyset$  \Comment{$\mathcal{U}$ is used to store pairs $(H, \chi_H)$}
\State Initialize $\mathcal{H'} \gets \emptyset$ \Comment{$\mathcal{H'}$ is used to store graphs with proper $2$-cutsets}

\For{each $H \in \mathcal{H}$}
    \State $i\gets1$
    \If{$(T_{X(1)}, T_Y, \chi_{X(1)}^0, \chi_{X(1)}^1) \gets$ apply Algorithm~\ref{alg:process-basic} to $H$}
        \State $\mathcal{H'} \gets \mathcal{H'} \cup \{H\}$
        \State $\mathcal{T} \gets \mathcal{T} \cup \{(T_{X(1)}, \chi_{X(1)}^0),(T_{X(1)} , \chi_{X(1)}^1)\}$  \Comment{Store graphs with colorings}
        \While{$T_Y$ has a proper 2-cutset}
            \State $(T_{X(i)}, T_Y', \chi_{X(i)}^0, \chi_{X(i)}^1) \gets$ apply Algorithm~\ref{alg:process-basic} to $T_Y$
            \State $\mathcal{T} \gets \mathcal{T} \cup \{(T_{X(i)}, \chi_{X(i)}^0), (T_{X(i)} , \chi_{X(i)}^1)\}$
            \State $T_Y \gets T_Y'$
            \State $i\gets i+1$
        \EndWhile 
        \State $k(H)\gets$ $i-1$
        \State $\chi_Y \gets$ a proper $3$-coloring of $T_Y$
        \State $\mathcal{T} \gets \mathcal{T} \cup \{(T_Y, \chi_Y)\}$
    \Else
    \State $\mathcal{U} \gets \mathcal{U} \cup \{(H, \chi_H)\}$
    \EndIf
\EndFor

\For{each $H \in \mathcal{H'}$}               \Comment{Combine parts and colorings on proper $2$-cutsets}
    \For{$i$ in $[k(H)]$}
        \If{$\{a,b\} = V(T_{X(k(H)-i+1)}) \cap V(T_Y)$ is a proper 2-cutset}
            \If{$\chi_{X(k(H)-i+1)}^0(a) = \chi_Y(a)$ and $\chi_{X(k(H)-i+1)}^0(b) = \chi_Y(b)$}
                \State $\chi_Y(v) \gets 
                                \left\{
                                \begin{array}{ll}
                                \chi_{X(k(H)-i+1)}^0(v) & \text{if } v \in V(T_{X(k(H)-i+1)}), \\
                                \chi_Y(v)               & \text{if } v \in V(T_Y\setminus V(T_{X(k(H)-i+1)})).
                                \end{array}
                                \right.$
                \State $T_Y \gets T_{X(k(H)-i+1)} \cup T_Y$
            \Else
                \State $\chi_Y(v) \gets 
                                \left\{
                                \begin{array}{ll}
                                \chi_{X(k(H)-i+1)}^1(v) & \text{if } v \in V(T_{X(k(H)-i+1)}), \\
                                \chi_Y(v)               & \text{if } v \in V(T_Y\setminus V(T_{X(k(H)-i+1)})).
                                \end{array}
                                \right.$
                \State $T_Y \gets T_{X(k(H)-i+1)} \cup T_Y$
            \EndIf
        \EndIf
    \EndFor
    \State $\mathcal{U} \gets \mathcal{U} \cup \{(T_Y, \chi_{T_Y})\}$      \Comment{Now $T_Y=H$}
\EndFor

\For{$j$ in $[j(G)]$}               \Comment{Combine parts and colorings on clique-cutsets}
    \For{Each $H^{j(G)+1-j}$ in $\mathcal{H}$}
        \State Add the vertices of $\mathcal{R_H}^{\,j(G)+1-j}$ back to $H^{\,j(G)+1-j}$ in reverse order of their removal, and assign each such vertex in $\mathcal{R_H}^{j(G)+1-j}$ a color distinct from its neighbors
    \EndFor
    \State $H^{j(G)-j},\chi (H^{j(G)-j})\gets$ Combine all subgraphs $H^{j(G)+1-j}$ and their colorings $\chi(H^{j(G)+1-j})$ on the clique cutsets to obtain a proper 3-coloring of $G$.    
\EndFor
\State $G,\chi (G)\gets$ $H^{0},\chi (H^{0})$
\State \Return $\chi (G)$
\end{algorithmic}\vspace{-8pt}
\noindent\hrulefill




\section*{Acknowledgement} This research  was supported by National Key R\&D Program of China under Grant No. 2022YFA1006400 and National Natural Science Foundation of China under Grant No. 12571376.

\section*{Declaration}
\noindent$\textbf{Conflict~of~interest}$
The authors declare that they have no known competing financial interests or personal relationships that could have appeared to influence the work reported in this paper.
	
\noindent$\textbf{Data~availability}$
Data sharing not applicable to this paper as no datasets were generated or analysed during the current study.


\begin{thebibliography}{99}

 \bibitem{CCCFL2020} G. Chen, Y. Chen, Q. Cui, X. Feng and Q. Liu,
\newblock The chromatic number of graphs with no induced subdivision of $K_4$,
\newblock \emph{Graphs Combin.}, \textbf{36} (2020) 719--728.

\bibitem{Chen2021}
G. Chen, Y. Chen, Q. Cui, X. Feng and Q. Liu,
\newblock The chromatic number of \{$\text{ISK}_4$, diamond, bowtie\}-free graphs,
\newblock \emph{J. Graph Theory}, \textbf{96} (2021) 554--577.

\bibitem{Chudnovsky2019} 
M. Chudnovsky, C. Liu, O. Schaudt, S. Spirkl, N. Trotignon and K. Vu\v{s}kovi\'c,
\newblock  Triangle-free graphs that do not contain an induced subdivision of $K_4$ are $3$-colorable,
\newblock  \emph{J. Graph Theory},   \textbf{92} (2019) 67--95.


\bibitem{MC06} M. Chudnovsky, N. Robertson, P. Seymour and R. Thomas,
\newblock The strong perfect graph theorem, \newblock \emph{Ann. of Math.} \textbf{164} (2006) 51--229.

\bibitem{Chudnovsky2016} 
M. Chudnovsky, A. Scott, and P. Seymour,
\newblock  Induced subgraphs of graphs with large chromatic number. II. Three steps towards Gy\'{a}rf\'{a}s' conjectures,
\newblock  \emph{J. Combin. Theory Ser. B},   \textbf{118} (2016) 109--128.


\bibitem{Chudnovsky20192} 
M. Chudnovsky, A. Scott and P. Seymour,
\newblock   Induced subgraphs of graphs with large chromatic number.XII. Distant stars,
\newblock  \emph{ J. Graph Theory},   \textbf{ 92}  (2019) 237--254.

\bibitem{duffin1965} 
R. J. Duffin,
\newblock  Topology of series-parallel networks,
\newblock  \emph{J. Math. Anal. Appl.},   \textbf{ 10}  (1965) 303--318.

\bibitem{Erdos1959} 
P. Erd\H{o}s,
\newblock   Graph theory and probability,
\newblock  \emph{Can. J. Math.},   \textbf{11} (1959) 34--38.

\bibitem{g2} A. Gy\'{a}rf\'{a}s,
\newblock On Ramsey covering-numbers, 
\newblock \emph{Colloquia Mathematic Societatis J\'{a}nos Bolyai 10, Infinite and Finite Sets. North-Holland/American Elsevier, New York } (1975) 801--816.

\bibitem{g3} A. Gy\'{a}rf\'{a}s, E. Szemer\'edi and Zs. Tuza,
\newblock Induced subtrees in graphs of large chromatic number,
\newblock \emph{Discrete Mathematics} \textbf{ 30} (1980)  235--344.




\bibitem{Karp} R. M. Karp,
\newblock Reducibility among Combinatorial Problems.
\newblock \emph{in {\it Ideas that created the future--classic papers of computer science}}  (1972)  349--356, MIT Press, Cambridge
(2021) MR4679190.





\bibitem{Kierstead1994} H. A. Kierstead and S. G. Penrice,
\newblock Radius two trees specify $\chi$-bounded classes,
\newblock \emph{J. Graph Theory} \textbf{18} (1994)  119--129.

\bibitem{Kierstead2004} H. A. Kierstead and Y. Zhu,
\newblock Radius three trees in graphs with large chromatic number,
\newblock \emph{SIAM J. Discrete Math.} \textbf{18} (2004) 571--581.


\bibitem{DKDO2004} D. K\"uhn and D. Osthus,
\newblock Induced subdivisions in $K_{s,s}$-free graphs of large average degree,
\newblock\emph{ Combinatorica} \textbf{24} (2004) 287--304.

\bibitem{NKL2017} N. K. Le,
\newblock Chromatic number of $\mathrm{ISK_4}$-free graphs,
\newblock \emph{ Graphs Combin.}, \textbf{33} (2017) 1635--1646.


\bibitem{Lehot}      
P. G. H. Lehot,
\newblock An optimal algorithm to detect a line graph and output its root graph,     
\newblock  \emph{A, J. Assoc. Comput. Machin.},  \textbf{21 (4)} (1974) 569--575.



\bibitem{12}      
B. L\'ev\^eque, D. Lin, F. Maffray, N. Trotignon,
\newblock Detecting induced subgraphs,  
\newblock  \emph{Discrete Appl. Math.},  \textbf{157} (2009) 3540--3551. 




\bibitem{Leveque2012}      
B. L\'ev\^eque, F. Maffray and N. Trotignon,
\newblock On graphs with no induced subdivision of $K_4$,     
\newblock  \emph{J. Combin. Theory Ser. B},  \textbf{102} (2012) 924--947.

 \bibitem{AP2014} A. Pawlik, J. Kozik, T. Krawczyk, M. Laso\'n, P. Micek, W. T. Trotter and B. Walczak,
\newblock Triangle-free intersection graphs of line segments with large chromatic number,
\newblock\emph{ J. Combin. Theory Ser. B}, \textbf{105} (2014) 6--10.


\bibitem{scott1997} A. D. Scott,
\newblock Induced trees in graphs of large chromatic number,
\newblock\emph{ J. Graph Theory}, \textbf{24} (1997) 297--311.


\bibitem{scott2016} A. D. Scott and P. Seymour,
\newblock Induced subgraphs of graphs with large chromatic number. I. Odd holes,
\newblock\emph{ J. Combin. Theory Ser. B}, \textbf{121} (2016), 68--84.

\bibitem{ss2020} A. D. Scott and P. Seymour, 
\newblock Induced subgraphs of graphs with large chromatic number. XIII. New brooms,
\newblock \emph{Eur. J. Comb.}, \textbf{84} (2020) 103024.

\bibitem{ss} A. D. Scott and P. Seymour, 
\newblock A survey of $\chi$-boundedness,  
\newblock \emph{J. Graph Theory}, \textbf{95} (2020) 473--504.




\bibitem{Sumner1981} D. P. Sumner, 
\newblock Subtrees of a graph and the chromatic number,  
\newblock \emph{The Theory and Applications of Graphs},  (1981) 557--576.


\bibitem{Tarjan} R. E. Tarjan, 
\newblock Decomposition by clique separators,  
\newblock \emph{Discrete Math},  (1985) 221--232.


\bibitem{Trotignon2017} 
N. Trotignon and K. Vu\v{s}kovi\'c,
\newblock  On triangle-free graphs that do not contain a subdivision of the complete graph on four vertices as an induced subgraph, 
\newblock \emph{J. Graph Theory}, \textbf{84} (2017) 233--248.


\end{thebibliography}
\end{document}